\documentclass[12pt, reqno]{amsart}
\usepackage{amsmath,amssymb,amsthm,amsxtra, setspace}
\usepackage{amssymb}
\usepackage{stmaryrd}
\usepackage{alltt}
\usepackage{graphicx,type1cm,xcolor}
\usepackage{mathtools, accents}
\usepackage{mathrsfs}
\usepackage{hyperref}
\usepackage{alltt}
\usepackage{fouridx}
\usepackage{dsfont}
\usepackage{dsfont}
\usepackage{cancel}
\usepackage[mathcal]{euscript}
\usepackage[margin=1in]{geometry}
\usepackage{shadethm}
\usepackage{float}

\allowdisplaybreaks

\newcommand{\norm}[1]{\left\|#1\right\|}

\newcommand{\R}{\mathbb{R}}

\numberwithin{equation}{section}

\newshadetheorem{Theorem}{Theorem}
\newshadetheorem{Proposition}{Proposition}
\newshadetheorem{Lemma}{Lemma}
\newshadetheorem{Assumption}{Assumption}
\newshadetheorem{Corollary}{Corollary}
\newshadetheorem{Definition}{Definition}
\newshadetheorem{Example}{Example}
\newshadetheorem{Remark}{Remark}
\newshadetheorem{Note}{Note}
\newshadetheorem{Hypothesis}{Hypothesis}
\newshadetheorem{Explanation}{Explanation}
\newtheorem{theorem}{Theorem}[section]
\newtheorem{lemma}[theorem]{Lemma}
\newtheorem{proposition}[theorem]{Proposition}

\newtheorem{definition}[theorem]{Definition}

\newtheorem{remark}[theorem]{Remark}

\def\A{\mathcal{A}}
\def\bM{\mathcal{M}}
\def\N{\widetilde{\mathcal{N}}}
\def\H{\mathcal{H}}
\def\g{\gamma}
\def\umm{\left(\vartheta_m\right)_{tt}}
\def\um{\left(\vartheta_m\right)_t}
\def\E{E_\infty}

\def\l{\lambda_\infty}

\let\originalleft\left
\let\originalright\right
\renewcommand{\left}{\mathopen{}\mathclose\bgroup\originalleft}
\renewcommand{\right}{\aftergroup\egroup\originalright}

\newcommand{\ip}[2]{\fourIdx{}{0}{}{\!x}{\mathcal{ X}}}

\newcommand\dela[1]{}
\hypersetup{colorlinks=true,%
	citecolor=red,%
	filecolor=blue,%
	linkcolor=blue,%
}
\usepackage{graphicx}
\usepackage[utf8]{inputenc}
\usepackage[T1]{fontenc}
\mathtoolsset{showonlyrefs}
\usepackage{nomencl}
\makenomenclature
\usepackage{amssymb}

\usepackage[symbol]{footmisc}
\usepackage{centernot}
\usepackage{ragged2e}
\justifying
\usepackage{orcidlink}

\usepackage{todonotes}

\usepackage{xpatch}
\makeatletter   
\xpatchcmd{\@tocline}
{\hfil\hbox to\@pnumwidth{\@tocpagenum{#7}}\par}
{\ifnum#1<0\hfill\else\dotfill\fi\hbox to\@pnumwidth{\@tocpagenum{#7}}\par}
{}{}
\makeatother    
\makeatletter
\def\l@subsection{\@tocline{2}{0pt}{4pc}{6pc}{}}
\def\l@subsubsection{\@tocline{3}{0pt}{8pc}{8pc}{}}
\makeatother

\makeatletter
\def\l@section{\@tocline{1}{12pt}{0pt}{}{\bfseries}}% <- added
\makeatother
\usepackage{enumitem}

\title[Constrained Damped Nonlinear Wave Equation]{Global  Well-posedness and Asymptotic Analysis of a Damped Nonlinear Wave Equation with a Codimension-One Constraint}

\begin{document}
	\maketitle
	\begin{center}
		\author{Harsh Tiwari\footnote[4]{Department of Mathematics, Indian Institute of Technology Roorkee-IIT Roorkee, Haridwar Highway, Roorkee, Uttarakhand 247667, India.\\
				\textit{e-mail:} H. Tiwari: \email{harsh\_t@ma.iitr.ac.in; harshtiwari98765@gmail.com}.\\
				\textit{e-mail:} Manil T. Mohan: \email{maniltmohan@ma.iitr.ac.in; maniltmohan@gmail.com}.\\
				$\hspace{2mm} ^\ast$Corresponding author.
			}\orcidlink{0009-0004-4685-8878} and Manil T. Mohan$^\ast$\footnotemark[4]\orcidlink{0000-0003-3197-1136}}
	\end{center}

	\begin{abstract}
		%We establish the global well-posedness of strong solutions to a damped nonlinear wave equation under a codimension-one constraint.
	We prove the global existence and uniqueness of strong solutions to a constrained version of the damped nonlinear wave equation
$$ \vartheta_{tt}+\gamma \vartheta_t-\Delta\vartheta+|\vartheta|^{p-2}\vartheta=0 $$
	on a smooth bounded domain $\varOmega\subset\mathbb{R}^d$, where the	evolution is projected onto the tangent space of the Hilbert manifold
$$\bM	=	\left\{	\vartheta\in L^2(\varOmega):\|\vartheta\|_{L^2(\varOmega)}=1	\right\}, $$
	which is the unit sphere in $L^2(\varOmega)$. We assume that 
	$$p\in[2,\infty)\ \text{ for }\ d=1,2, \ \text{ while }\  2\leq p\leq \frac{2(d-1)}{d-2} \ \text{	for }\ d\geq3.$$  By employing the Faedo--Galerkin approximation method,	together with suitable a priori estimates and compactness arguments, we establish the global well-posedness of the problem. In particular, we show that the Hilbert manifold $\bM$ is invariant under the flow, and hence the $L^2$-constraint is preserved throughout the evolution. Using the \emph{Lusternik--Schnirelmann theory}, we show that the corresponding stationary problem possesses at least countably many stationary solutions. We further investigate the long-time behaviour of solutions and prove that, along a subsequence, every solution of the constrained problem converges to a stationary solution by invoking \emph{Webb's theorem} and \emph{Barbalat's lemma}. When the initial data are sufficiently close to the first eigenfunction of the associated stationary problem, we show that the unique strong solution converges in $H_0^1(\varOmega)$ to the unique positive ground-state solution.
	%For the strongly damped case, in which the damping term $\gamma \vartheta_t$	is replaced by $-\gamma\Delta \vartheta_t$, we obtain analogous global 	well-posedness results for $p\in[2,\infty)$ when $d=1,2$, and	$	2\leq p\leq \frac{2d}{d-2} $	when $d\geq3$. We also study the corresponding long-time behaviour	and establish asymptotic convergence results in this setting.
	\end{abstract}

	\begin{center}
		\tableofcontents
	\end{center}
	\section{Introduction}
	Let $(\mathcal{H},\|\cdot\|_{\mathcal{H}})$ be a Hilbert space equipped with the inner product $(\cdot,\cdot)_{\mathcal{H}}$, and let	$\mathcal{M}:=\{\xi\in\mathcal{H}:\|\xi\|_{\mathcal{H}}=1\}$
	denote the unit sphere in $\mathcal{H}$. For each $\xi\in\mathcal{M}$, the tangent space to $\mathcal{M}$ at $\xi$ is given by $T_{\xi}\mathcal{M}:=\{\eta\in\mathcal{H}:(\eta,\xi)_{\mathcal{H}}=0\}.$
	Let $\mathcal{F}:D(\mathcal{F})\subset\mathcal{H}\times\mathcal{H}\to\mathbb{R}$ be a possibly densely defined scalar field. Given an initial state $\vartheta_0\in\mathcal{M}$ and an initial velocity $\vartheta_1\in T_{\vartheta_0}\mathcal{M}$, consider the following abstract  initial value problem: 
	\begin{equation}\label{eq:wave}
		\begin{cases}
			\vartheta_{tt}(t)=\mathcal{F}(\vartheta (t),\vartheta_t(t)),\quad t>0,\\
			\vartheta(0)=\vartheta_0,\ \vartheta_t(0)=\vartheta_1.
		\end{cases}
	\end{equation}
	We assume that this initial value problem is globally well-posed; that is, for every admissible pair of initial data $(\vartheta_0,\vartheta_1)$, there exists a unique solution $\vartheta (t)$ defined for all $t\geq 0$. 	In general, the flow generated by \eqref{eq:wave} does not preserve the 	constraint manifold $\bM$, even if $\vartheta_0\in\bM$ and	$(\vartheta_0,\vartheta_1)_\H=0$. Indeed, differentiating the identity 	$\|\vartheta (t)\|_\H^2=1$ twice yields the necessary condition $(\vartheta_{tt},\vartheta)_\H=-\|\vartheta_t\|_\H^2.$	Hence, contrary to the first-order case, the tangency condition	$(\mathcal{F}(\vartheta,\vartheta_t),\vartheta)_\H=0,$ is not sufficient to guarantee the invariance of $\bM$.	Therefore, we now introduce the modified scalar field 	$\widehat{\mathcal{F}}:D(\mathcal{F})\subset\H\times\H\to\R,$ defined by
	\begin{equation}\label{mod of func}
		\widehat{\mathcal{F}}(\vartheta,\varrho)=\mathcal{F}(\vartheta,\varrho)-\bigl((\mathcal{F}(\vartheta,\varrho),\vartheta)_\H+\|\varrho\|_\H^2\bigr)\vartheta,\ \text{ for }\ (\vartheta,\varrho)\in (D(F)\cap\bM)\times\H.
	\end{equation}
	Clearly, $(\widehat{\mathcal{F}}(\vartheta,\varrho),\vartheta)_\H=-\|\varrho\|_\H^2$. Consequently, $$\frac{d^2}{dt^2}\|\vartheta (t)\|_\H^2
	=2(\vartheta_{tt},\vartheta)_\H+2\|\vartheta_t\|_\H^2=0.$$
	Therefore, if $
	\|\vartheta_0\|_\H=1\ \text{ and }\
	(\vartheta_0,\vartheta_1)_\H=0,$
	then $
	\|\vartheta (t)\|_\H=1,$
	for all $t$ in the interval of existence. In other words, the manifold
	$\bM$ is positively invariant under the flow generated by the constrained
	wave equation
	\begin{equation}\label{eq:modified-wave}
		\begin{cases}
			\vartheta_{tt}(t)=\widehat{\mathcal{F}}(\vartheta (t),\vartheta_t(t)),\quad t>0,\\
			\vartheta(0)=\vartheta_0,\  \vartheta_t(0)=\vartheta_1.
		\end{cases}
	\end{equation}
	\subsection{Motivation and formulation of the model}
	The physical motivation for the model under consideration is closely
	related to the non-relativistic limit of the Klein--Gordon equation.
	Several mathematical and physical studies have established that, as the
	speed of light $c$ becomes large, solutions of the relativistic
	Klein--Gordon equation converge, in a suitable sense, to solutions of
	the corresponding Schr\"odinger equation (\cite{RJC-81,SM+KN+T0-02,AZ-10}). A fundamental property of the
	Schr\"odinger dynamics is the conservation of the $L^2$-norm of the
	solution. This observation motivates the introduction of a modified
	Klein--Gordon equation whose dynamics likewise preserve the $L^2$-norm.
	The purpose of this modification is to construct a relativistic model
	that is more closely aligned with the limiting Schr\"odinger dynamics
	and may therefore provide a more accurate approximation over
	intermediate time scales.

%	A natural motivation for our model comes from the nonrelativistic limit of the Klein--Gordon equation. It is well known that, as the speed of light $c \to \infty$, the corresponding dynamics approaches that of the Schr\"odinger equation \cite{RJC-81,SM+KN+T0-02,AZ-10}. Since the $L^2$-norm is conserved for Schr\"odinger solutions, we are naturally led to impose $L^2$-norm conservation as a key constraint in our model. This conservation property plays a central role in the formulation and analysis of the model.

	Let $\varOmega\subset\mathbb{R}^d$, $d\in\mathbb{N}$, be a bounded domain with boundary of class $C^2$. We set  $\mathcal{H}=L^2(\varOmega)$ and let $-\Delta$ denote the Dirichlet Laplacian with
	$D(-\Delta)=H^2(\varOmega)\cap H_0^1(\varOmega).$
	We first present three examples that motivate the constrained problem considered below.
	\begin{itemize}
		\item If $\mathcal{F}_1(\vartheta,\vartheta_t)=\Delta\vartheta$, then
		\begin{equation}
			\widehat{\mathcal{F}}_1(\vartheta,\vartheta_t)
			=\Delta\vartheta-\left((\Delta\vartheta,\vartheta)_{L^2(\varOmega)}+\|\vartheta_t\|_{L^2(\varOmega)}^2\right)\vartheta
			=\Delta\vartheta-\left(\|\vartheta_t\| _{L^2(\varOmega)}^2-\|\nabla\vartheta\|_{L^2(\varOmega)}^2\right)\vartheta.
		\end{equation}
		\item For $\mathcal{F}_2(\vartheta,\vartheta_t)=-|\vartheta|^{p-2}\vartheta$, $p\geq2$, we have
		\begin{equation}
		\widehat{\mathcal{F}}_2(\vartheta,\vartheta_t)
			=-|\vartheta|^{p-2}\vartheta
			+\left(\|\vartheta\|_{L^p(\varOmega)}^p-\|\vartheta_t\|_{L^2(\varOmega)}^2\right)\vartheta.
		\end{equation}
		\item Finally, for $\mathcal{F}_3(\vartheta,\vartheta_t)=-\g \vartheta_t$, $\g\geq0$, the constraint
		$(\vartheta,\vartheta_t)_{L^2(\varOmega)}=0$ yields
		\begin{equation}
			\widehat{\mathcal{F}}_3(\vartheta,\vartheta_t)
			=-\g \vartheta_t-\|\vartheta_t\|_{L^2(\varOmega)}^2\vartheta.
		\end{equation}
	\end{itemize}
	Combining the above three contributions, we obtain the following $\widehat{F}$, which will be used throughout the paper:
	\begin{equation}\label{eq:proj}
		\widehat{\mathcal{F}}(\vartheta,\vartheta_t)
		=\Delta\vartheta-|\vartheta|^{p-2}\vartheta-\g \vartheta_t
		+\left(\|\nabla \vartheta\|_{L^2(\varOmega)}^2+\|\vartheta\|_{L^p(\varOmega)}^p-\|\vartheta_t\|_{L^2(\varOmega)}^2\right)\vartheta.
	\end{equation}
	\subsubsection{The Model}
	We consider the constraint manifold
	\begin{equation}
		\mathcal{M}
		=\left\{\vartheta\in L^2(\varOmega):\|\vartheta\|_{L^2(\varOmega)}=1\right\},
	\end{equation}
	and study the following constrained nonlinear damped wave equation:
	\begin{equation}\label{eq:mains}
		\left\{
		\begin{aligned}
			\vartheta_{tt}+\g \vartheta_t-\Delta\vartheta+|\vartheta|^{p-2}\vartheta
			&=\left(\|\nabla\vartheta\|_{L^2(\varOmega)}^2+\|\vartheta\|_{L^p(\varOmega)}^p-\|\vartheta_t\|_{L^2(\varOmega)}^2\right)\vartheta
			&&\text{ in }\ \varOmega\times(0,T),\\
			\vartheta&=0
			&&\text{ on }\ \partial\varOmega\times(0,T),\\
			\vartheta(0)&=\vartheta_0,\  \vartheta_t(0)=\vartheta_1
			&&\text{ in }\ \varOmega.
		\end{aligned}
		\right.
	\end{equation}
	The initial data are assumed to satisfy
	\begin{equation}\label{MID}
		\vartheta_0\in\mathcal{M},\  (\vartheta_0,\vartheta_1)_{L^2(\varOmega)}=0.
	\end{equation}
	For the well-posedness theory,  we assume
	\begin{equation}\label{eqn-restriction}
		p\in
		\begin{cases}
			[2,\infty), & d=1,2,\\
			\left[2,\dfrac{2(d-1)}{d-2}\right], & d\geq3,
		\end{cases}
	\end{equation}
	and we allow $\g\geq0$. In particular, the existence and uniqueness results remain valid in the absence of damping. For the asymptotic analysis, we impose the additional assumption $\g>0$, under which the damping provides the dissipation needed to study the long-time behavior of solutions.
	\subsection{Literature survey}

The well-posedness theory for nonlinear wave equations has been extensively developed. Lions \cite{Lions1969} established existence and uniqueness results for wave equations with polynomial nonlinearities on bounded domains using the Faedo--Galerkin method. Subsequently, the well-posedness theory for damped wave equations with external forcing was developed in several works; see, for example, \cite{LA+GP-69,HM-72,AH-81,AH-87,JL+WS-65}. The asymptotic behavior of damped wave equations has also received considerable attention. Narazaki \cite{T.N-04} employed $L^p$--$L^q$ estimates to derive decay estimates for semilinear damped wave equations, while Haraux and Zuazua \cite{AH+EZ-87} established related decay results for damped semilinear hyperbolic equations. Chill and Haraux \cite{RC+AH-04} obtained optimal decay estimates for wave equations in exterior domains and quantified the difference between solutions of abstract dissipative wave and heat equations. For exterior domains, Sobajima \cite{Sobajima} investigated the diffusion phenomenon for wave equations with effective space-dependent damping by combining energy estimates with properties of Kummer's confluent hypergeometric functions. More recently, nonlinear wave equations involving the $p$-Laplacian and logarithmic-type nonlinearities have been studied in \cite{Yang,HY+HY-22}, where global existence, energy decay, and sufficient conditions for finite-time blow-up were established. Asymptotic profiles for various classes of damped wave equations have also been investigated; see \cite{RI+AS-16,H.M-21,TN+KN-08,Webb80} and the references therein.

A complementary approach to investigating the long-time behavior relies on the \L ojasiewicz–Simon inequality \cite{SL-63} which has become an important tool in the asymptotic analysis of gradient-like evolution equations. Simon \cite{LS-83} extended the classical \L ojasiewicz inequality to suitable analytic energy functionals defined on infinite-dimensional Hilbert spaces by employing a Lyapunov–Schmidt reduction. The inequality was subsequently extended to broader classes of functions and, in particular, to analytic functionals on Banach spaces and on submanifold of a Banach space; see \cite{PMNF-20,KK-98,AH-81,FR-20}. These developments provide a useful framework for establishing the convergence of bounded trajectories to stationary states and, under suitable assumptions, for obtaining convergence rates. Haraux and Jendoubi \cite{AH+MAD-99, MAJ-98} investigated the asymptotic behavior of global bounded solutions to second-order evolution problems with analytic nonlinearities, employing the Łojasiewicz--Simon inequality together with Webb's theorem. In particular, in \cite{AH+MAD-98}, established the convergence of global bounded solutions of second-order gradient-like systems with analytic nonlinearities to equilibrium points by means of the Łojasiewicz--Simon gradient inequality. Thus, the \L ojasiewicz--Simon framework is particularly relevant for gradient-flow structures, including evolution equations subject to norm-preserving or other geometric constraints.

% In the context of constrained evolution equations, Rybka \cite{PR-06} employed a \L ojasiewicz-type inequality to investigate the asymptotic behavior of heat flows on manifolds. More recently, Bawalia et al. \cite{AB+ZB+MTM+PR} applied a \L{}ojasiewicz--Simon gradient inequality on a Hilbert submanifold to study a nonlinear heat equation with finite-codimensional constraints. 
%In particular, for bounded domains, even polynomial nonlinearities, and $1\leq d\leq3$, they showed that the unique global strong solution converges to a stationary state in $W^{2,q}(\varOmega)\cap W^{1,q}_0(\varOmega)$, with some restriction on $p$ and $q$.
% $2\leq q<\frac{2d}{d+4-4\beta},
%\ 1<\beta<\frac32.$

A complementary direction concerns evolution equations subject to norm-preserving constraints. In the parabolic setting, Rybka \cite{PR-06} studied heat flows on manifolds and their asymptotic behavior, while Caffarelli and Lin \cite{LC+FL-09} investigated nonlocal heat flows that preserve the $L^2$-energy, establishing well-posedness and studying their convergence behavior in the context of singularly perturbed systems. Ma and Cheng \cite{LM+LC-13} further developed the theory of $L^2$-norm-preserving flows on manifolds, addressing global existence, stability, asymptotic behavior, and gradient estimates. Brze\'zniak and Hussain \cite{ZB+JH-24} investigated constrained nonlinear heat equations and established global well-posedness, invariance of the sphere manifold, and a gradient-flow structure. 
%Hussain \cite{JH-23} studied a deterministic nonlinear gradient-type heat equation with polynomial nonlinearities and established global existence and uniqueness by means of the Faedo--Galerkin compactness method.

The study of nonlocal norm-preserving terms has subsequently been extended to a broader class of nonlinear parabolic equations. Antonelli et al. \cite{PA+PC+BS-24} investigated a nonlinear parabolic equation preserving the $L^2$-norm on both bounded domains as well as in the whole space, establishing local and global well-posedness together with detailed asymptotic results. In particular, they proved convergence toward stationary states and, for positive initial data on a ball, strong convergence to the ground state. Shakarov \cite{BS-25} subsequently considered a nonlinear parabolic equation with a nonlocal $L^2$-norm-preserving term and analyzed its well-posedness and convergence toward stationary states. Bawalia et al. \cite{AB+ZB+MTM} studied constrained nonlinear heat equations with polynomial nonlinearities and established global existence and uniqueness of strong solutions, as well as invariance of the $L^2$-unit sphere and strong convergence of positive solutions toward the unique positive ground state. Subsequently, Bawalia et al. \cite{AB+ZB+MTM+PR} applied a \L{}ojasiewicz--Simon gradient inequality to study a nonlinear heat equation (see also \cite{PR-06})  on a Hilbert Manifold  with finite-codimensional constraints.
%The aforementioned work of Bawalia et al. \cite{AB+ZB+MTM+PR} further demonstrates how the \L ojasiewicz--Simon inequality can be incorporated into this constrained setting to obtain convergence toward stationary states.

Norm-preserving constraints have also arisen naturally in other classes of evolution equations. Brze\'zniak et al. \cite{ZB+GD+MM-18} investigated constrained two-dimensional Navier--Stokes equations on $\mathbb{R}^2$ and $\mathbb{T}^2$, with particular emphasis on global well-posedness and the vanishing-viscosity limit. Geometric constraints are also fundamental in the study of wave maps on manifolds; see \cite{DT-04,DT-05}. In the stochastic setting, Brze\'zniak and Hussain \cite{ZB+JH-20} initiated the study of constrained stochastic nonlinear heat equations with Stratonovich-type noise, proving existence and uniqueness of mild solutions in smooth bounded two-dimensional domains. Bawalia et al. \cite{AB+ZB+MTM2} subsequently extended this framework to arbitrary spatial dimensions and polynomial nonlinearities of order $2\leq p<\infty$. For constrained stochastic Navier--Stokes equations, Brze\'zniak et al. \cite{ZB+GD-21} constructed martingale solutions driven by multiplicative Gaussian noise in Stratonovich form, established pathwise uniqueness, and obtained strong solutions through a Yamada--Watanabe-type argument. The manuscript \cite{ABH-26} presented  a rigorous and well-structured analysis of the sphere-constrained modified Swift--Hohenberg equation, establishing global well-posedness, energy dissipation, restricted global attractors, and equilibrium structure, with appropriate care regarding the limitations of the stability conclusions.

%More recently, constrained stochastic damped wave equations have received considerable attention. Brze\'zniak and Cerrai \cite{ZB+SC-25} established existence and uniqueness results for constrained stochastic nonlinear damped wave equations driven by multiplicative Gaussian noise and studied their limiting behavior. Subsequently, Cerrai and Xie \cite{SC+MX-25} considered a related model with an additional velocity-dependent diffusion term.
 Recently, constrained stochastic damped wave equations have received considerable attention. Brzeźniak and Cerrai \cite{ZB+SC-25} established the existence ad uniqueness of constrained stochastic nonlinear damped wave equation perturbed by multipicative guassian noise and also studied the limiting behaviour by the help of Smoluchowski-Kramers approximation and this result extended by Cerrai and  Xie  \cite{SC+MX-25} where the diffusion term contains additional velocity  and estblished their well-posedness and analyze it's limiting behaviour where perturbed small mass tends to 0.

	\subsection{Difficulties, approaches, and novelities}
The analysis of the constrained nonlinear damped wave equation presents
several difficulties that distinguish it from the standard nonlinear
damped wave equation and from unconstrained dissipative flows. The first
arises from the $L^2$-normalization constraint, which requires the solution
to remain on the constraint manifold throughout the evolution. This
constraint introduces a time-dependent Lagrange multiplier determined by
the solution itself. In particular, the term $-\|\vartheta_t\|_{L^2(\varOmega)}^2\vartheta$ in the
multiplier represents the curvature correction associated with the
constraint manifold. Its interaction with the damping term requires
careful treatment in the energy estimates. Moreover, preservation of the
constraint must be established both for the Galerkin approximations and
in the limiting process.

A second difficulty concerns the energy structure. Unlike the Hamiltonian
setting, the natural energy is not conserved, and its dissipation is not
directly apparent from the equation. The specific form of the Lagrange
multiplier, together with the propagated orthogonality condition
$(\vartheta_t,\vartheta)_{L^2(\varOmega)}=0$, yields the exact dissipation identity
\begin{equation}\label{eqn-energy-diss}
	\frac{d}{dt}E(\vartheta (t),\vartheta_t(t))
	=-\g\|\vartheta_t(t)\|_{L^2(\varOmega)}^2\leq 0,
\end{equation}
where 
\begin{align*}
	E(\vartheta (t),\vartheta_t(t)):=\frac12\|\vartheta_t(t)\|_{L^2(\varOmega)}^2+\frac12\|\nabla \vartheta (t)\|_{L^2(\varOmega)}^2+\frac1p\|\vartheta (t)\|_{L^p(\varOmega)}^p,
	\ \  t\ge0.
\end{align*}
This identity provides the fundamental global bound. For strong
solutions, however, higher-order estimates are required to control the
nonlinear terms. To obtain these estimates, we impose a suitable
restriction on the nonlinear exponent $p$ (see \eqref{eqn-restriction}), which allows the required
Sobolev estimates to be applied. Together with the local Lipschitz
properties of the nonlinearities and a constraint-preserving
Faedo--Galerkin scheme, these estimates provide the framework for
establishing global strong solutions while retaining the normalization
constraint (Theorem \ref{wdss}).

The main difficulty in the asymptotic analysis of
\eqref{eq:mains}--\eqref{MID} lies in establishing convergence as
$t\to\infty$. For $\gamma>0$ and the values of $p$ given in \eqref{p range}, the energy dissipation identity
\eqref{eqn-energy-diss} implies the existence of a sequence
$t_n\to\infty$ such that
\[
\vartheta_t(t_n)\to0
\ \text{ in }\ L^2(\varOmega),
\]
while, along a subsequence,
\[
\vartheta(t_n)\xrightarrow{w}\vartheta_\infty
\ \text{ weakly in }\ H_0^1(\varOmega)
\]
for some $\vartheta_\infty\in H_0^1(\varOmega)$ (Lemma~\ref{sub seq conv soln}).
However, this argument only provides convergence along a suitably
chosen sequence of times. It does not imply that every sequence
$t_n\to\infty$ leads to the same limit, nor does it establish the
existence of a limit for the full trajectory as $t\to\infty$.

To overcome this difficulty, we first establish, using semigroup
arguments, \emph{Webb's theorem} \cite[Proposition 3.2, (3.7)--(3.9)]{Webb79} (see \cite[Lemma~7.6.2]{AH+MAJ-15} also), and the Bolzano--Weierstrass theorem, that
the orbit
\[
\mathcal{O}(\vartheta_0,\vartheta_1)
:=
\{(\vartheta (t),\vartheta_t(t)):t\geq0\}
\]
is precompact (Theorem~\ref{apct} and Lemma \ref{compact B}). We then apply \emph{Barbalat's lemma}
\cite[Theorem 5]{BL} to (Proposition~\ref{convergence of ut to 0})
deduce that
\[
\vartheta_t(t)\to0
\ \text{ in }\ L^2(\varOmega)
\ \text{ as }\ t\to\infty. 
\]
 By means of a time translation argument for
$\vartheta(\cdot)$, we obtain stronger convergence properties (Lemmas~\ref{lem:uniform_shift}
and~\ref{lamda t+tk conv}), which allow us to show that 
\[
\vartheta (t_n)\to \vartheta_\infty
\ \text{ strongly  in }\ H_0^1(\varOmega)
\]
and  identify the limiting function
$\vartheta_\infty$ as a stationary solution (Theorem~\ref{thm-strong}), namely,
\begin{equation}\label{stationary-problem}
	\left\{
	\begin{aligned}
		-\Delta\vartheta_\infty+|\vartheta_\infty|^{p-2}\vartheta_\infty
		&=\lambda_\infty \vartheta_\infty
		&&\text{in }\varOmega,\\
		\vartheta_\infty&=0
		&&\text{on }\partial\varOmega,\\
		\|\vartheta_\infty\|_{L^2(\varOmega)}&=1,
	\end{aligned}
	\right.
\end{equation}
where
\[
\lambda_\infty
=
\|\nabla \vartheta_\infty\|_{L^2(\varOmega)}^2
+
\|\vartheta_\infty\|_{L^p(\varOmega)}^p.
\]
The above argument, however, does not by itself establish convergence
of the entire trajectory; it relies initially on the extraction of a
subsequence of times $t_n\to\infty$. This naturally raises the question of whether every solution converges, as \(t\to\infty\), to a stationary solution. The stationary problem \eqref{stationary-problem} has atleast countably many solutions by the Lusternik--Schnirelmann minmax theory (Lemma \ref{lem-PS} and Proposition \ref{PS}). Consequently, one cannot expect all trajectories to converge to one and the same stationary solution. Indeed, if $v$ is any stationary solution, then the initial data
$
\vartheta(0)=v,\ \vartheta_t(0)=0$
generate the stationary trajectory
$
\vartheta (t)\equiv v.
$
Thus, different stationary solutions correspond to different
stationary trajectories.

On the other hand, the stationary problem \eqref{stationary-problem} possesses a \emph{unique positive ground-state solution} $\varphi_1$ (Theorem \ref{thm:ground-state} and Lemma~\ref{isolation of phi}). In the case of the constrained heat equation, the maximum principle can be used to exploit positivity of the initial data and to establish convergence towards $\varphi_1$ (\cite[Theorem 4.3]{PA+PC+BS-24}, \cite[Theorem 5.13]{AB+ZB+MTM}). Such a maximum-principle argument is not available for the constrained damped wave equation. This presents a significant obstacle to proving convergence towards the ground state for arbitrary initial data.

We therefore restrict our attention to initial data sufficiently close
to the ground state $\varphi_1$ (Lemma~\ref{lem-energy}). For such
initial data, we prove that the corresponding solution remains trapped
in a suitable neighborhood of $\varphi_1$ (Lemma~\ref{trapped}) and, consequently (Theorem~\ref{GB}),
\[
\vartheta (t)\to\varphi_1 \ \text{ in }\  H_0^{1}(\varOmega)
\ \text{ as }\  t\to\infty. 
\]

To the best of our knowledge, a convergence result of this type, in
which convergence to a stationary state is established for the entire
trajectory rather than only along a subsequence of times, is not
available in the existing literature for constrained damped wave
equations. The overall novelty lies in combining well-posedness, asymptotic precompactness,
sequential convergence, the passage to full-trajectory convergence, and
the isolation of the first nonlinear eigenvalue to characterize the
unique asymptotic state without relying on a maximum-principle argument.

	\subsection{Main results} 
	We begin by defining the notion of a strong solution to problem~\eqref{eq:mains}--\eqref{MID}. Throughout, $\mathcal{A}$ denotes	the Dirichlet Laplacian introduced in Subsection~\ref{sub-laplace}.
	\begin{definition}\label{def-strong}
		Let us fix $p$ as in \eqref{eqn-restriction}. 	Let $T>0$ and let
		\begin{equation*}
			\vartheta_0\in D(\A)\cap\mathcal{M},\ 
			\vartheta_1\in H_0^1(\varOmega)\ \text{ and }\ 
			(\vartheta_0,\vartheta_1)_{L^2(\varOmega)}=0.
		\end{equation*}
		Then a function
		\begin{equation*}
			\vartheta\in C([0,T]; H^1_0(\varOmega)\cap \bM) \cap C_w([0,T]; D(\A)),
		\end{equation*}
		with 
		\begin{equation*}
			\vartheta_t\in C([0,T];L^2(\varOmega))
			\cap C_w([0,T];H_0^1(\varOmega)),
		\end{equation*}
		is called \emph{strong solution} if it satisfies the following:
		\begin{itemize}
			\item $\vartheta_{tt}\in L^\infty(0,T;L^2(\varOmega))$ and	$\N(\vartheta,\vartheta_t)\in L^\infty(0,T;L^2(\varOmega)),$
			where
			\begin{equation*}
				\N(\vartheta,\vartheta_t)=
				-\A\vartheta-|\vartheta|^{p-2}\vartheta-\g \vartheta_t
				+\left(\|\nabla\vartheta\|_{L^2(\varOmega)}^2+\|\vartheta\|_{L^p(\varOmega)}^p-\|\vartheta_t\|_{L^2(\varOmega)}^2\right)\vartheta .
			\end{equation*}
			\item $\vartheta_{tt}=\N(\vartheta,\vartheta_t)  \ \text{ in }\   L^2(\varOmega)\ \text{ for a.e. }\ t\in (0,T).$\vspace{0.2cm}
			\item The initial conditions hold:
			\begin{equation*}
				\vartheta(0)=\vartheta_0\ \text{ in }\ D(\A),\ 
				\vartheta_t(0)=\vartheta_1\ \text{ in }\ H^1_0(\varOmega).
			\end{equation*}
		\end{itemize}
	\end{definition}
	
	The following theorem provides the existence and uniqueness of strong solutions to problem~\eqref{eq:mains}--\eqref{MID}.
	
	\begin{theorem}[Existence and uniqueness of strong solutions]\label{wdss}
		Let us fix $p$ as in \eqref{eqn-restriction} and  $T>0$ be given. Then for any initial data
		\begin{equation}\label{eq:IVP}
			\vartheta_0 \in D(\A)\cap \mathcal{M},
			\ 
			\vartheta_1 \in H_0^1(\varOmega),
		\end{equation}
		there exists a unique strong solution $	\vartheta:[0,T]\to H_0^1(\varOmega)\cap \bM$
		such that$$	\vartheta\in C([0,T]; H^1_0(\varOmega)\cap \bM) \cap C_w([0,T]; D(\A)),$$
		together with $$\vartheta_t\in C([0,T]; L^2(\varOmega))\cap C_w([0,T];H^1_0(\varOmega)),$$
		and $\vartheta$ satisfies \eqref{eq:mains}-\eqref{MID}  in the sense of Definition \ref{def-strong}. Moreover, the solution remains on the manifold $\bM$, that is,
		$$\vartheta (t)\in \bM,\ \text{ for all }\ t\in[0,T].$$
	\end{theorem}
	
The proof of Theorem~\ref{wdss} is given in Section~\ref{WP}. 
For the asymptotic analysis, we fix our $p$ as 
	\begin{equation}\label{p range}
	p\in
	\begin{cases}
		[2,\infty), & d=1,2,\\
		\left[2,\dfrac{2(d-1)}{d-2}\right), & d\ge 3.
	\end{cases}
\end{equation}
The following result describes the asymptotic behaviour of solutions to \eqref{eq:mains}--\eqref{MID} along a sequence of times $t_n\to\infty$.
	
		\begin{theorem}\label{thm-strong}
	Let $\vartheta_0\in D(\mathcal{A})\cap\mathcal{M}$ and $\vartheta_1\in H_0^1(\varOmega)$ with $(\vartheta_0,\vartheta_1)=0$. 	Let $\vartheta$ be the unique strong solution of \eqref{eq:mains}--\eqref{MID}. Then there exists a sequence $\{t_n\}_{n\in\mathbb{N}}$, $t_n\to\infty$ as $n\to\infty$ and $\vartheta_{\infty}\in H_0^1(\varOmega)\cap\mathcal{M}$  such that 
		\begin{equation*}
			\vartheta (t_n)\to  \vartheta_\infty\ \text{ strongly in }\ H_0^1(\varOmega).
		\end{equation*}
		Moreover, $\vartheta_\infty$ is a  weak solution of the following system:
		\begin{equation}\label{stationary pde}
			\left\{
			\begin{aligned}
				\A\vartheta_\infty+|\vartheta_\infty|^{p-2}\vartheta_\infty
				&=\lambda_\infty \vartheta_\infty
				&&\text{ in }\ H^{-1}(\varOmega),\\
				%\vartheta_{\infty}&=0  &&\text{ on }\ \partial\varOmega,\\
				\|\vartheta_\infty\|_{L^2(\varOmega)}&=1,
			\end{aligned}
			\right. 
		\end{equation}
		where $	\lambda_{\infty}= \|\nabla\vartheta_\infty\|^2 _{L^2(\varOmega)}+\|\vartheta_\infty\|^p_{L^p(\varOmega)}.$
	\end{theorem}

	The proof of Theorem~\ref{sub-sec-asy} is given in
	Subsection~\ref{sub-sec-asy}. We next establish a result showing that,
	if the initial data are sufficiently close to the positive ground state
	$\varphi_1$, then the entire trajectory $\vartheta (t)$ converges to
	$\varphi_1$ as $t\to\infty$.

	\begin{theorem}[Global convergence]\label{GB}
		Suppose that the strong solution $\vartheta$ of \eqref{eq:mains}--\eqref{MID}, with
		initial data
		$	(\vartheta_0,\vartheta_1)\in D(\A)\times H_0^1(\varOmega)
		\subset H_0^1(\varOmega)\times L^2(\varOmega),	$
		satisfies
		\[
		\|\vartheta_0-\varphi_1\|_{H_0^1(\varOmega)}<r, \ \ 0< r<\frac{1}{2}\|\varphi_{1}\|_{H_0^1(\varOmega)},
		\]
	where $\varphi_1$ is the unique positive ground state solution of \eqref{stationary pde}. 	If $\vartheta_0$ and $\vartheta_1$ are chosen in such a way that 
		\begin{align}\label{eqn-energy-d}
			E(\vartheta(0),\vartheta_t(0))
			=
			\frac12\|\vartheta_1\|_{L^2(\varOmega)}^2
			+\frac12\|\nabla \vartheta_0\|_{L^2(\varOmega)}^2
			+\frac1p\|\vartheta_0\|_{L^p(\varOmega)}^p
			<
			E(\varphi_1,0)+\eta,
		\end{align}
		for some $\eta>0$.  
		Then   $(\vartheta,\vartheta_t)$  converges strongly to the equilibrium $(\varphi_1,0),$ as $t\to\infty$, that is,
		\begin{equation}
			(\vartheta (t),\vartheta_t(t))
			\to
			(\varphi_1,0)\ \text{ in } \ H_0^1(\varOmega)\times L^2(\varOmega)\ \text{ as }\  t\to\infty.
		\end{equation}
	\end{theorem}
	
	The proof of Theorem \ref{GB} is provided in Subsection \ref{sub-long-time}. 

	\subsection{Organization of the paper}

	The rest of the paper is organized as follows. In Section \ref{pre}, we introduce the functional setting and recall the basic properties of the Dirichlet Laplacian and the Sobolev embeddings used in the analysis. Section \ref{LLC} establishes the local Lipschitz continuity of the nonlinear mapping, which is then used to prove the well-posedness of the problem in Theorem \ref{wdss}. In Section \ref{WP}, we construct the Faedo--Galerkin approximations and establish the preservation of the constraint, uniform energy estimates, and higher-order estimates. Using suitable compactness arguments, we prove the existence and uniqueness of global strong solutions of problem \eqref{eq:mains}--\eqref{MID}.
	
	In Section \ref{AA}, we study the long-time behavior of the solutions of \eqref{eq:mains}--\eqref{MID}. We establish the dissipation of energy given in \eqref{eqn-energy-diss} and analyze its limiting behavior and the asymptotic dynamics along suitable sequences of times. In particular, Lemma \ref{sub seq conv soln} provides the sequential convergence required for the subsequent analysis, while Theorem \ref{apct} establishes the asymptotic precompactness of the trajectories which help us to characterizes the $\omega$-limit set. We further investigate the associated stationary nonlinear elliptic problem and the corresponding ground-state solutions (Theorem \ref{thm:ground-state} and Lemma \ref{isolation of phi}). Finally, the proof of Theorem~\ref{GB} is presented in Section \ref{AA}. The theorem establishes the stabilization and global convergence of the	solutions towards the ground-state set under suitable assumptions.
	
	Lastly, in Section \ref{appendix}, we prove that the stationary problem admits a countable family of distinct normalized solutions; see Proposition \ref{PS}. The proof relies on the Lusternik--Schnirelmann theory \cite{PD+JM-07} , the Palais--Smale condition, and the Krasnoselskii genus \cite{AA+AM-07}. This result plays an important role in establishing our main result, Theorem \ref{GB}.

	\section{Preliminaries}\label{pre}
	
	In this section, we introduce the functional spaces and notation used throughout the paper, including the Lebesgue, Sobolev, and Bochner spaces, together with their corresponding norms and inner products. Moreover, we discuss the locally Lipschitz property of linear and nonlinear operators. 
\subsection{Functional framework}

Let $\varOmega\subset\mathbb{R}^d$ be a bounded domain with smooth boundary. For $1\leq p\leq\infty$, we denote by $L^p(\varOmega),$ the space of equivalence classes of measurable functions $\psi:\varOmega\to\mathbb{R}$ for which $\|\psi\|_{L^p(\varOmega)}<\infty$, where
\begin{equation}
	\|\psi\|_{L^p(\varOmega)}=
	\begin{cases}
		\left(\displaystyle\int_\varOmega|\psi(x)|^p d x\right)^{1/p}, & 1\leq p<\infty,\\[2mm]
		\displaystyle\operatorname*{ess\,sup}_{x\in\varOmega}|\psi(x)|, & p=\infty.
	\end{cases}
\end{equation}
The space $L^2(\varOmega)$ is equipped with the inner product $ (\psi,\varphi)
=\int_\varOmega\psi(x)\varphi(x) d x.$ We denote by $H_0^1(\varOmega),$ the Sobolev space of functions $\psi\in L^2(\varOmega)$ whose first-order weak derivatives belong to $L^2(\varOmega)$ and whose trace vanishes on $\partial\varOmega$. It is equipped with the norm $\|\psi\|_{H_0^1(\varOmega)}=\|\nabla\psi\|_{L^2(\varOmega)}$, which is equivalent to the standard $H^1(\varOmega)$ norm by the Poincar\'e inequality. The dual space of $H_0^1(\varOmega)$ is denoted by $H^{-1}(\varOmega)$. We denote by $\langle g,\psi\rangle,$ the duality pairing between $g\in H^{-1}(\varOmega)$ and $\psi\in H_0^1(\varOmega)$. The norm on $H^{-1}(\varOmega)$ is defined by
\begin{equation}
	\|g\|_{H^{-1}(\varOmega)}
	=\sup_{\psi\in H_0^1(\varOmega)\setminus\{0\}}
	\frac{|\langle g,\psi\rangle|}{\|\psi\|_{H_0^1(\varOmega)}}.
\end{equation}

Throughout the paper, $\hookrightarrow$ and $\hookrightarrow\hookrightarrow$ denote continuous and compact embeddings, respectively. In particular, by the Sobolev embedding \cite[Theorem 5.6.7, 5.7.2]{DM+DZ-98},
\begin{align}\label{embedding}
\begin{array}{c@{\qquad}c}
	H_0^1(\varOmega)\hookrightarrow L^r(\varOmega)
	&
	H_0^1(\varOmega)\hookrightarrow\hookrightarrow L^r(\varOmega)
	\\[2mm]
	\begin{cases}
		1\le r\le\infty, & d=1,\\
		1\le r<\infty, & d=2,\\
		1\le r\le\dfrac{2d}{d-2}, & d\ge3,
	\end{cases}
	&
	\begin{cases}
		1\le r\le\infty, & d=1,\\
		1\le r<\infty, & d=2,\\
		1\le r<\dfrac{2d}{d-2}, & d\ge3.
	\end{cases}
\end{array}
\end{align}

Let $\mathcal{Z}$ be a Banach space. For $0<T\leq \infty$ and $1\leq p<\infty$, we denote by $L^p(0,T;\mathcal{Z})$ the space of strongly measurable functions $\psi:(0,T)\to \mathcal{Z}$ such that
\begin{equation}
	\|\psi\|_{L^p(0,T;\mathcal{Z})}
	=\left(\int_0^T\|\psi(t)\|_\mathcal{Z}^p  dt\right)^{1/p}<\infty.
\end{equation}
For $p=\infty$, we write $L^\infty(0,T;\mathcal{Z})$ for the space of strongly measurable functions satisfying
$\|\psi\|_{L^\infty(0,T;\mathcal{Z})}=\operatorname*{ess\,sup}_{t\in(0,T)}\|\psi(t)\|_\mathcal{Z}<\infty.$ When $\mathcal{Z}$ is a Hilbert space, the Bochner space $L^2(0,T;\mathcal{Z})$ is equipped with the inner product
\begin{equation}
	(\psi,\varphi)_{L^2(0,T;\mathcal{Z})}
	=\int_0^T(\psi(t),\varphi(t))_\mathcal{Z}  dt.
\end{equation}
We denote by $C([0,T];\mathcal{Z})$ the space of strongly continuous functions $\psi:[0,T]\to \mathcal{Z}$, endowed with the norm $\|\psi\|_{C([0,T];\mathcal{Z})}=\max_{t\in[0,T]}\|\psi(t)\|_\mathcal{Z}.$ Finally, $C_w([0,T];\mathcal{Z})$ denotes the space of weakly continuous functions from $[0,T]$ into $\mathcal{Z}$, that is,
\begin{equation}
	C_w([0,T];\mathcal{Z})
	=\left\{\psi:[0,T]\to \mathcal{Z}:
	t\mapsto\langle\upsilon,\psi(t)\rangle
	\text{ is continuous for every }\upsilon\in \mathcal{Z}^*\right\}.
\end{equation}
In particular, $C([0,T];\mathcal{Z})\hookrightarrow C_w([0,T];\mathcal{Z})$.
\subsection{The Dirichlet Laplacian}\label{sub-laplace}
We define the bilinear form
\begin{equation}
	a:H_0^1(\varOmega)\times H_0^1(\varOmega)\to\mathbb{R},
	\ 
	a(\varphi,\psi):=(\nabla\varphi,\nabla\psi).
\end{equation}
Then, it is immediate that 
\begin{equation}
	a(\varphi,\varphi)
	=\|\nabla\varphi\|_{L^2(\varOmega)}^2,
	\ 
	\varphi\in H_0^1(\varOmega).
\end{equation}
The bilinear form $a$ is continuous and coercive on $H_0^1(\varOmega)$.
Hence, by the Lax--Milgram theorem, the operator
$
	A:H_0^1(\varOmega)\to H^{-1}(\varOmega)
$
defined by
\begin{equation}
	\langle A\varphi,\psi\rangle
	=a(\varphi,\psi),
	\ 
	\varphi,\psi\in H_0^1(\varOmega),
\end{equation}
is an isomorphism from $H_0^1(\varOmega)$ onto $H^{-1}(\varOmega)$. We denote by $\mathcal A$ the realization of $A$ in $L^2(\varOmega)$, with domain
$
	D(\mathcal A)
	:=
	\{\varphi\in H_0^1(\varOmega):A\varphi\in L^2(\varOmega)\},
$
and set
\begin{equation}
	\mathcal A\varphi:=A\varphi,
	\  \varphi\in D(\mathcal A).
\end{equation}
Thus, $\mathcal A$ is the negative Laplacian with homogeneous
Dirichlet boundary conditions:
$	\mathcal A\varphi=-\Delta\varphi.
$
Since $\varOmega$ is smooth and bounded, elliptic regularity yields
$
	D(\mathcal A)
	=H^2(\varOmega)\cap H_0^1(\varOmega).
$
Consequently,
\begin{equation}
	\mathcal A=-\Delta:
	D(\mathcal A)\subset L^2(\varOmega)\to L^2(\varOmega)
\end{equation}
is a densely defined, self-adjoint, positive definite operator with
compact resolvent.
%
%
%
%
%
%We define the bilinear form
%\begin{equation}
%	a:H_0^1(\varOmega)\times H_0^1(\varOmega)\to\mathbb{R}
%	\ \text{ by }\ 
%	a(\varphi,\psi)=(\nabla\varphi,\nabla\psi).
%\end{equation}
%Note that 
%\begin{equation}
%	a(\varphi,\varphi)=\|\nabla\varphi\|_{L^2(\varOmega)}^2,
%	\  \text{ for all }\ \varphi\in H_0^1(\varOmega).
%\end{equation}
%By the Riesz representation theorem, there exists a unique bounded linear operator
%\begin{equation}
%	A:H_0^1(\varOmega)\to H^{-1}(\varOmega)
%\end{equation}
%such that
%\begin{equation}
%	\langle A\varphi,\psi\rangle=a(\varphi,\psi),
%	\ 
%	\text { for all }\ \varphi,\psi\in H_0^1(\varOmega).
%\end{equation}
%Moreover, $A$ is an isomorphism from $H_0^1(\varOmega)$ onto $H^{-1}(\varOmega).$ We set
%\begin{equation}
%	D(\A)=\{\varphi\in H_0^1(\varOmega): A\varphi\in L^2(\varOmega)\},
%	\ 
%	A\varphi=\A\varphi,\  \text { for all }\ \varphi\in D(\A).
%\end{equation}
%Thus, $\A$ coincides with the negative Laplacian subject to homogeneous Dirichlet boundary conditions, namely, $\A=-\Delta.$
%Since $\varOmega$ has a smooth boundary, elliptic regularity gives $D(\A)=H^2(\varOmega)\cap H_0^1(\varOmega).$
%Consequently,
%\begin{equation}
%	\A=-\Delta:D(\A)\to L^2(\varOmega)
%\end{equation}
%is a densely defined, self-adjoint, positive definite operator with compact resolvent. 
Hence, there exist an increasing sequence of positive eigenvalues
\begin{equation}
	0<\sigma_1\leq\sigma_2\leq\cdots\leq\sigma_n\leq\cdots,\qquad \sigma_n\to\infty,
\end{equation}
and an orthonormal basis $\{w_n\}_{n=1}^\infty$ of $L^2(\varOmega)$ satisfying
\begin{equation}
	\A e_n=\sigma_n w_n,
	\ 
	w_n\in D(\A),
	\  n\in\mathbb{N}.
\end{equation}
Elliptic regularity results also imply $\|\varphi\|_{D(\mathcal{A})}:=\|\Delta\varphi\|_{L^2(\varOmega)}$. 
	\subsection{Local Lipschitz property}\label{LLC}
	
	Throughout this subsection, we assume that $\g\ge0$ and 
	\begin{equation}\label{values of p}
	\left\{
	\begin{aligned}
		p\in[2,\infty), \qquad d=1,2,3,4,\\
		2\leq p\leq \frac{2(d-2)}{d-4}, \qquad d\ge5.
	\end{aligned}
	\right.
\end{equation}
	We define the nonlinear operator $\N:D(\A)\times L^2(\varOmega)\to L^2(\varOmega)$
	by
	\begin{equation}\label{eq:NLE}
		\N(\vartheta,\varrho)
		=
		-\A\vartheta
		-|\vartheta|^{p-2}\vartheta
		-\g\varrho
		+
		\left(
		\|\nabla \vartheta \|_{L^2(\varOmega)}^2
		+\|\vartheta\|_{L^p(\varOmega)}^p
		-\|\varrho\|_{L^2(\varOmega)}^2
		\right)\vartheta.
	\end{equation}
	Since
$	D(\A)=H^2(\varOmega)\cap H_0^1(\varOmega),$
	elliptic regularity, together with the relevant Sobolev embeddings, ensures that each term in \eqref{eq:NLE} belongs to $L^2(\varOmega)$. Hence, the nonlinear operator $\N$ is well defined.
	\begin{lemma}\label{lem:LL}
		For the values of $p$ given in \eqref{values of p}, the mapping $\N:D(\A)\times L^2(\varOmega)\to L^2(\varOmega)$
		is locally Lipschitz continuous. More precisely, for every $r>0$, there exists
		a constant $C_r>0$ such that
		\begin{equation}\label{eq:LLP}
			\|\N(\vartheta_1,\varrho_1)-\N(\vartheta_2,\varrho_2)\|_{L^2(\varOmega)}
			\le
			C_r
			\left(
			\|\vartheta_1-\vartheta_2\|_{D(\A)}
			+\|\varrho_1-\varrho_2\|_{L^2(\varOmega)}
			\right),
		\end{equation}
		whenever, $\|(\vartheta_i,\varrho_i)\|_{D(\A)\times L^2}\le r,\  i=1,2.$
	\end{lemma}
	\begin{proof}
		Let $(\vartheta_1,\varrho_1),(\vartheta_2,\varrho_2)\in D(\A)\times L^2(\varOmega)$. Then
		\begin{align}
			&
			\|\N(\vartheta_1,\varrho_1)-\N(\vartheta_2,\varrho_2)\|_{L^2(\varOmega)}
			\nonumber\\
			&\le
			\|\A(\vartheta_1-\vartheta_2)\|_{L^2(\varOmega)}
			+\g\|\varrho_1-\varrho_2\|_{L^2(\varOmega)}
			+\||\vartheta_1|^{p-2}\vartheta_1-|\vartheta_2|^{p-2}\vartheta_2\|_{L^2(\varOmega)}
			\nonumber\\
			&\quad
			+
			\left|
			\|\nabla \vartheta_1\|_{L^2(\varOmega)}^{2}
			-\|\nabla \vartheta_2\|_{L^2(\varOmega)}^{2}
			+\|\vartheta_1\|_{L^p(\varOmega)}^{p}
			-\|\vartheta_2\|_{L^p(\varOmega)}^{p}
			-\|\varrho_1\|_{L^2(\varOmega)}^{2}
			+\|\varrho_2\|_{L^2(\varOmega)}^{2}
			\right|
			\|\vartheta_2\|_{L^2(\varOmega)}
			\nonumber\\
			&\quad
			+
			\left|
			\|\nabla \vartheta_1\|_{L^2(\varOmega)}^{2}
			+\|\vartheta_1\|_{L^p(\varOmega)}^{p}
			-\|\varrho_1\|_{L^2(\varOmega)}^{2}
			\right|
			\|\vartheta_1-\vartheta_2\|_{L^2(\varOmega)}.
			\label{eq:split}
		\end{align}
		Applying the elementary inequality
		\begin{equation*}
			|a^2-b^2|
			\leq (|a|+|b|)|a-b|,
		\end{equation*}
		along with
		\begin{equation}\label{identity 1}
			\bigl||x|^{p-2}x-|y|^{p-2}y\bigr|
			\leq C\bigl(|x|^{p-2}+|y|^{p-2}\bigr)|x-y|,
		\end{equation}
		H\"older's inequality, and the Sobolev embeddings \eqref{embedding}, we obtain
		\begin{align*}
		&	\|\N(\vartheta_1,\varrho_1)-\N(\vartheta_2,\varrho_2)\|_{L^2(\varOmega)}
		\nonumber\\	&\le
			(1+\g)
			\|(\vartheta_1,\varrho_1)-(\vartheta_2,\varrho_2)\|_{D(\A)\times L^2(\varOmega)}
		\\&\quad	+C
			\left(
			\|\vartheta_1\|_{L^{2(p-1)}(\varOmega)}^{p-2}
			+
			\|\vartheta_2\|_{L^{2(p-1)}(\varOmega)}^{p-2}
			\right)
			\|\vartheta_1-\vartheta_2\|_{L^{2(p-1)}(\varOmega)}
			\\
			&\quad
			+
			\left(
			\|\nabla \vartheta_1\|_{L^2(\varOmega)}
			+
			\|\nabla \vartheta_2\|_{L^2(\varOmega)}
			\right)
			\|\nabla(\vartheta_1-\vartheta_2)\|_{L^2(\varOmega)}
			\|\vartheta_2\|_{L^2(\varOmega)}
		\\&\quad	+
			C
			\left(
			\|\vartheta_1\|_{L^p(\varOmega)}^{p-1}
			+
			\|\vartheta_2\|_{L^p(\varOmega)}^{p-1}
			\right)
			\|\vartheta_1-\vartheta_2\|_{L^2(\varOmega)}\,
			\|\vartheta_2\|_{L^2(\varOmega)}
			\\
			&\quad
			+
			\left(
			\|\varrho_1\|_{L^2(\varOmega)}
			+
			\|\varrho_2\|_{L^2(\varOmega)}
			\right)
			\|\varrho_1-\varrho_2\|_{L^2(\varOmega)}
			\|\vartheta_2\|_{L^2(\varOmega)}
		\\&\quad	+
			\left(
			\|\nabla \vartheta_1\|_{L^2(\varOmega)}^{2}
			+
			\|\vartheta_1\|_{L^p(\varOmega)}^{p}
			+
			\|\varrho_1\|_{L^2(\varOmega)}^{2}
			\right)
			\|\vartheta_1-\vartheta_2\|_{L^2(\varOmega)}.
		\end{align*}
		Since $\|(\vartheta_i,\varrho_i)\|_{D(\A)\times L^2}\le r$, all norms involving
		$\vartheta_i$ and $\varrho_i$ are bounded by a constant depending only on $r$. Therefore,
		there exists a constant $C_r>0$ such that
		\[
		\|\N(\vartheta_1,\varrho_1)-\N(\vartheta_2,\varrho_2)\|_{L^2(\varOmega)}
		\le
		C_r
		\left(
		\|\vartheta_1-\vartheta_2\|_{D(\A)}
		+
		\|\varrho_1-\varrho_2\|_{L^2(\varOmega)}
		\right),
		\]
		which proves \eqref{eq:LLP}.
	\end{proof}
%	\begin{remark}
%		The operator $\N$ satisfies the local Lipschitz condition for the for the following range of $ p: $
%		\begin{equation*}
%			\left\{
%			\begin{aligned}
%				p\in[2,\infty), \qquad d=1,2,3,4,\\
%				2<p<\frac{2(d-2)}{d-4}, \qquad d\ge5.
%			\end{aligned}
%			\right.
%		\end{equation*}
%	\end{remark}	
	\section{Well-posedness}\label{WP}
	The goal of this section is to prove the existence and uniqueness of strong solutions of \eqref{eq:mains}--\eqref{MID}. The value of $p$ is fixed as in \eqref{eqn-restriction}. 
%		\begin{equation}\label{value of p-1}
%		p\in
%		\begin{cases}
%			[2,\infty), & d=1,2,\\
%			\left[2,\dfrac{2(d-1)}{d-2}\right], & d\geq3,
%		\end{cases}
%	\end{equation}
	The 
	analysis is based on the Faedo-Galerkin approximation scheme, in combination with several key functional analytic tools: the \emph{Banach-Alaoglu theorem} \cite[Theorem C.17]{JCR-20}, the \emph{Aubin--Lions compactness theorem} \cite{JS-87}, and the \emph{Strauss Lemma} \cite[Lemma 1.4, pg-178]{RT-01}. 
%	\subsection{Existence and uniqueness of strong solutions}
%	Let us now provide a proof of Theorem \ref{wdss} by 
	\begin{proof}[Proof of Theorem  \ref{wdss}]
		Let $\{\sigma_n\}_{n\in\mathbb{N}}$ be the nondecreasing sequence of positive eigenvalues of the operator $\A$, with $\sigma_n \to \infty$ as $n \to \infty$. Let $\{w_n\}_{n\in\mathbb{N}}$ denote an associated orthonormal basis of $L^2(\varOmega)$ consisting of eigenfunctions of $\A$, which is also orthogonal basis in $H_0^1(\varOmega)$.
		
		For $m \in \mathbb{N}$, set  $	V_m := \mathrm{span}\{w_1,\dots,w_m\} \subset L^2(\varOmega),$ 	and let $\mathrm{P}_m : L^2(\varOmega) \to V_m$ denote the $L^2$-orthogonal projection, given  by
		\begin{equation}\label{eq:pro}
			\mathrm{P}_m \vartheta = \sum_{i=1}^m (\vartheta,w_i)\, w_i, \  \vartheta \in L^2(\varOmega).
		\end{equation}	
		We seek a Galerkin approximation $\vartheta_m : [0,T] \to V_m$ of the form
		\begin{equation}\label{eq:projection}
			\vartheta_m(t) = \sum_{k=1}^m \xi_k^{(m)}(t)\, w_k,
		\end{equation}
		where $ \xi_k^{(m)}: [0,T] \to \R $, $1 \le k \le m$, are scalar coefficient smooth functions satisfying the following finite-dimensional system of ordinary diﬀerential equations:
		\begin{equation}\label{eq:NLE pro}
			\begin{cases}
				(\umm,w_k)= (-\A\vartheta_m,w_k) - (|\vartheta_m|^{p-2}\vartheta_m,w_k) - \g(\um,w_k)\\
				\qquad\qquad+ \left(-\norm{\um} _{L^2(\varOmega)}^{2}
				+ \norm{\nabla \vartheta_m} _{L^2(\varOmega)}^{2}
				+ \norm{\vartheta_m}_{L^p(\varOmega)}^{p}\right)(\vartheta_m,w_k),\quad\text{ in }\ \varOmega\times(0,T), \\
			\vartheta_m(0)=\vartheta_{0m}=\frac{\mathrm{P}_m\vartheta_0}{\norm{\mathrm{P}_m\vartheta_0} _{L^2(\varOmega)}},\\
				\um(0)=\vartheta_{1m}=\mathrm{P}_m\vartheta_1-(\mathrm{P}_m\vartheta_1,\vartheta_{0m})\vartheta_{0m},
			\end{cases}
		\end{equation}
		so that $\vartheta_{0m}\in\bM\ \text{and}\ (\vartheta_{0m},\vartheta_{1m})_{L^2(\varOmega)}=0$. The coeﬃcients $\xi_k^{(m)}(\cdot)$  is chosen in such a way that $\xi_k^{(m)}(0)= \frac{1}{\norm{\mathrm{P}_m\vartheta_0} _{L^2(\varOmega)}}\left(\vartheta_0,w_k\right)\ \text{and}\  \left(\xi_k^{(m)}\right)_t(0)=(\vartheta_1,w_k)-(\mathrm{P}_m\vartheta_1,\vartheta_{0m})(\vartheta_0,w_k).$
		Since $\vartheta_0\in D(\A)$, it admits the spectral representation
		\begin{equation}
			\vartheta_0=\sum_{n=1}^{\infty}(\vartheta_0,w_n) w_n,
		\end{equation}
		where $\{w_n\}_{n\ge1}$ satisfying $\mathcal{A} w_n=\sigma_n w_n,\ w_n|_{\partial\varOmega}=0.$
		Since $\vartheta_0\in D(\A)$, we have
		\begin{equation}
			\mathcal{A} \vartheta_0=\sum_{n=1}^{\infty}\sigma_n (\vartheta_0,w_n) w_n\in L^2(\varOmega),
		\end{equation}
		Therefore,
		$\|\vartheta_0-\mathrm{P}_m \vartheta_0\| _{L^2(\varOmega)}+\|\mathcal{A}(\vartheta_0-\mathrm{P}_m \vartheta_0)\| _{L^2(\varOmega)}
		\to 0,\ \text{ as }\ m\to\infty.$
		Since, $\mathrm{P}_m \vartheta_0\to \vartheta_0\ \text{ in }D(\A)\ \text{as}\ m\to\infty \ $ and 
		\begin{equation}\label{ normproj1}
			\lim_{m\to\infty}\|\mathrm{P}_m \vartheta_0\| _{L^2(\varOmega)}^2
			=
			\lim_{m\to\infty}
			\sum_{i=1}^{m}|(\vartheta_0,w_i)|^2
			=
			\sum_{i=1}^{\infty}|(\vartheta_0,w_i)|^2
			=
			\|\vartheta_0\| _{L^2(\varOmega)}^2
			=
			1,
		\end{equation} 
		we have
		\begin{equation}
			\|\vartheta_{0m}-\vartheta_0\|_{D(\A)}\leq\frac{\|\mathrm{P}_m \vartheta_0-\vartheta_0\|_{D(\A)}}{\|\mathrm{P}_m \vartheta_0\| _{L^2(\varOmega)}}
			+\|\vartheta_0\|_{D(\A)}\left|\frac{1}{\|\mathrm{P}_m \vartheta_0\| _{L^2(\varOmega)}}-1\right|\to 0,\ \text{ as }\ m\to\infty.					
		\end{equation}
		Therefore, 
		\begin{equation}\label{ineq:pro1}
			\vartheta_{0m}\to \vartheta_0\ \text{ in }\ D(\A)\cap\bM.
		\end{equation}
		It remains to approximate the initial datum $\vartheta_1$. In order to preserve the orthogonality condition, we have defined
		$\vartheta_{1m}=\mathrm{P}_m\vartheta_1-(\mathrm{P}_m\vartheta_1,\vartheta_{0m})\vartheta_{0m}$.
		Since $\mathrm{P}_m \vartheta_1\to \vartheta_1$ in $H_0^1(\varOmega)$ and
		$\vartheta_{0m}\to \vartheta_0$ in $H_0^1(\varOmega)$, we have
		\begin{equation}\label{ineq:pro2}
			\|\vartheta_{1m}-\vartheta_1\|_{H_0^1}\leq	\|\mathrm{P}_m \vartheta_1-\vartheta_1\|_{H_0^1}+|(\mathrm{P}_m \vartheta_1,\vartheta_{0m})|\,\|\vartheta_{0m}\|_{H_0^1}.
		\end{equation}
		Moreover,
		\begin{equation}
			(\mathrm{P}_m \vartheta_1,\vartheta_{0m})=(\mathrm{P}_m \vartheta_1-\vartheta_1,\vartheta_{0m})+(\vartheta_1,\vartheta_{0m}-\vartheta_0)+(\vartheta_1,\vartheta_0).
		\end{equation}
		Since $(\vartheta_1,\vartheta_0)=0$, $\mathrm{P}_m \vartheta_1\to \vartheta_1$ in $L^2(\varOmega)$ and
		$\vartheta_{0m}\to \vartheta_0$ in $L^2(\varOmega)$, it follows that
		\begin{equation}\label{ineq:pro3}
			(\mathrm{P}_m \vartheta_1,\vartheta_{0m})\to0 \ \text{ and }\  	\vartheta_{1m}\to \vartheta_1	\ \text{ in }\ H_0^1(\varOmega).
		\end{equation}
		
		The local existence and uniqueness of solutions to the system of ordinary
		differential equations in \eqref{eq:NLE pro} follow from the
		\emph{Picard-Lindel\"of theorem} and the local Lipschitz continuity of the
		associated nonlinear operator in \eqref{eq:LLP}. Hence, there exists a
		unique local solution
		$\vartheta_m\in C^2([0,\widetilde{T}_m];V_m)$
		for some \(0<\widetilde{T}_m<T\). The uniform energy estimates established below
		allow us to extend this local solution to the whole interval \([0,T]\).
		Furthermore, the analysis carried out below shows that
		$\vartheta_m(t)\in\mathcal{M}, \text{ for every }\ t\in[0,T].$
		\vskip 0.1 cm
		\noindent 
		\textbf{Step 1:} \textit{$\vartheta_m(t)\in\bM,$ for all $t\in [0,T]$}. Multiplying \eqref{eq:NLE pro} by $\xi_k^{(m)}(\cdot)$ and summing over $1\le k\le n$, we obtain
			\begin{align}
				(\umm,\vartheta_m)&=-(\A\vartheta_m,\vartheta_m) -\g(\um,\vartheta_m)-(|\vartheta_m|^{p-2}\vartheta_m,\vartheta_m)\\
				&\quad+\left(-\norm{\um} _{L^2(\varOmega)}^{2}+ \norm{\nabla \vartheta_m}^{2}	 _{L^2(\varOmega)}+ \norm{\vartheta_m}_{L^p(\varOmega)}^{p}\right)(\vartheta_m,\vartheta_m),\label{inner prod um}
			\end{align}
		for a.e. $t\in (0,T)$. Therefore, we infer 
		\begin{align}
				\frac{1}{2}\frac{d^2}{dt^2}\|\vartheta_m\|^2 _{L^2(\varOmega)}-\|\um\|^2 _{L^2(\varOmega)}&= -\|\nabla \vartheta_m\|^2 _{L^2(\varOmega)}-\frac{1}{2}\frac{d}{dt}\|\vartheta_m\|^2 _{L^2(\varOmega)}-\|\vartheta_m\|^p_{L^p(\varOmega)}\\
				&\quad+\left(\norm{\nabla \vartheta_m} _{L^2(\varOmega)}^{2}	+ \norm{\vartheta_m}_{L^p(\varOmega)}^{p}-\|\um\|^2 _{L^2(\varOmega)}\right)\|\vartheta_m\|^2 _{L^2(\varOmega)},
		\end{align}
		which can be rewritten as
		\begin{equation*}
			\begin{aligned}
			&\frac{1}{2}\frac{d^2}{dt^2}(\|\vartheta_m\|^2_{L^2(\varOmega)}-1)+\frac{1}{2}\frac{d}{dt}(\|\vartheta_m\|^2 _{L^2(\varOmega)}-1)\\
			&= \left(\norm{\nabla \vartheta_m} _{L^2(\varOmega)}^{2}+\hspace{-0.1cm} \norm{\vartheta_m}_{L^p(\varOmega)}^{p}-\|\um\|^2 _{L^2(\varOmega)}\right)(\|\vartheta_m\|^2 _{L^2(\varOmega)}-1),
				\end{aligned}
		\end{equation*}
	for a.e. $t\in[0,T]$.	Now, define
		\begin{equation}\label{def al la}
			\alpha_m(t)= \|\vartheta_m(t)\| ^2-1\ \text{ and }\
			\beta_m(t)= \left( \norm{\vartheta_m}_{L^p(\varOmega)}^{p}-\|\um\|^2 _{L^2(\varOmega)}+\|\nabla \vartheta_m\|^2 _{L^2(\varOmega)}\right).
		\end{equation}
		Then, the above identity together with the initial conditions in \eqref{eq:NLE pro} yields the following second-order ordinary differential equation:
		\begin{equation}\label{ode}
			\left\{
			\begin{aligned}
				& (\alpha_m)_{tt}(t)+(\alpha_m)_t(t)=2\beta_m(t)\alpha_m(t),\qquad t\in (0,T),\\
				& \alpha_m(0)=0,\\
				& (\alpha_m)_t(0)=0.
			\end{aligned}
			\right. 
		\end{equation}
		Since,  $\beta_m\in C([0,T])$, then, by the classical existence and uniqueness theorem for ordinary differential equations, the above initial value problem admits a unique solution. Since $\alpha_m\equiv0$ is a solution satisfying the prescribed initial conditions, it follows by uniqueness that
		\begin{equation*}
			\alpha_m(t)=0,\ \text{ for all}\ t\in[0,T].
		\end{equation*}
		Therefore,
		\begin{equation}\label{norm1}
			\|\vartheta_m(t)\|_{L^2(\varOmega)}^2=1,\  \text{ for all }\  t\in[0,T].
		\end{equation}
		Hence, $\vartheta_m(t)\in \bM, \text{ for all }  t\in[0,T].$
		This completes the proof of the invariance of $\bM$ under the approximate flow.
		Since \(V_m\) is finite-dimensional, all norms on \(V_m\) are equivalent. Therefore, the above estimate yields the existence of a solution \(\vartheta_m\) on a maximal interval \([0,\widetilde{T}_m)\), where \(0<\widetilde{T}_m< T\). The \emph{a priori} estimates established below are uniform with respect to \(m\), which exclude finite-time blow-up and consequently imply that \(\widetilde{T}_m=T\). Hence, the approximate solution \(\vartheta_m\) is well defined on the entire interval \([0,T]\).
		Moreover, since \(\vartheta_m(t)\in \bM\), for all \(t\in[0,T]\), it follows that
		\begin{equation}\label{product0}
			\frac{1}{2}\frac{d}{dt}\|\vartheta_m(t)\| _{L^2(\varOmega)}^{2}	=(\um(t),\vartheta_m(t))=0,	\quad \text{for a.e. } t\in[0,T].
		\end{equation}
		
		\vskip 0.1 cm
		\noindent 
		\textbf{Step 2:} \textit{Uniform energy estimates for $\vartheta_m(\cdot)$.} Multiplying \eqref{eq:NLE pro} by $\left(\xi_k^{(m)}\right)_t(\cdot)$ and summing over $k=1,\ldots,m$, we obtain
		\begin{equation}
			(\umm,\um)	+(\A\vartheta_m,\um)	+\g (\um,\um)	+(|\vartheta_m|^{p-2}\vartheta_m,\um)=\beta_m (\vartheta_m,\um),
		\end{equation}
		for a.e. \(t\in[0,T]\) where $\beta_m$ is defined in \eqref{def al la}.  Using the relation \eqref{product0} in above equation, we deduce  
		\begin{equation}
			\frac{1}{2}\frac{d}{dt}\|\um\| _{L^2(\varOmega)}^{2}+\frac{1}{2}\frac{d}{dt}\|\nabla \vartheta_m\| _{L^2(\varOmega)}^{2}+\g \|\um\| _{L^2(\varOmega)}^{2}+\frac{1}{p}\frac{d}{dt}\|\vartheta_m\|_{L^p(\varOmega)}^{p}=0,
		\end{equation}
		for a.e. \(t\in[0,T]\). Consequently,
		\begin{equation}
			\frac{d}{dt}\left(\frac{1}{2}\|\um\| _{L^2(\varOmega)}^{2}+\frac{1}{2}\|\nabla \vartheta_m\| _{L^2(\varOmega)}^{2}	+\frac{1}{p}\|\vartheta_m\|_{L^p(\varOmega)}^{p}\right)+\g \|\um\| _{L^2(\varOmega)}^2=0,	\  \text{ a.e. } t\in[0,T].
		\end{equation}
			By integrating the preceding equality over the interval $(0,t)$ and invoking \eqref{embedding}, \eqref{ normproj1} and \eqref{ineq:pro2}, we infer 
				\begin{align}
					&
					\left(
					\frac{1}{2}\|\um(t)\| _{L^2(\varOmega)}^{2}
					+\frac{1}{2}\|\nabla \vartheta_m(t)\| _{L^2(\varOmega)}^{2}
					+\frac{1}{p}\|\vartheta_m(t)\|_{L^p(\varOmega)}^{p}
					\right)
					+\g\int_{0}^{t}\|\um(s)\| _{L^2(\varOmega)}^{2}  ds \\
					&\qquad=
					\left(
					\frac{1}{2}\|\um(0)\| _{L^2(\varOmega)}^{2}
					+\frac{1}{2}\|\nabla \vartheta_m(0)\| _{L^2(\varOmega)}^{2}
					+\frac{1}{p}\|\vartheta_m(0)\|_{L^p(\varOmega)}^{p}
					\right) \\
					&\qquad\le
					C\left(
					\|\vartheta_1\| _{L^2(\varOmega)}^{2}
					+\|\vartheta_0\|_{H_0^1(\varOmega)}^{2}+\|\vartheta_0\|_{H_0^1(\varOmega)}^{p}
					\right),\label{energy bound 1}
				\end{align}
			where $C>0$ denotes a constant independent of $m$.
			It follows immediately from \eqref{energy bound 1} that the right-hand side is independent of $m$. Consequently, we obtain the following uniform estimate: 
				\begin{align}
					&
					\sup_{t\in[0,T]}
					\left(
					\frac{1}{2}\|\um(t)\| _{L^2(\varOmega)}^{2}
					+\frac{1}{2}\|\nabla \vartheta_m(t)\| _{L^2(\varOmega)}^{2}
					+\frac{1}{p}\|\vartheta_m(t)\|_{L^p(\varOmega)}^{p}
					\right)
					+\g\int_{0}^{T}\|\um(s)\| _{L^2(\varOmega)}^{2}  ds \\
					&\qquad\le
					C\left(
					\|\vartheta_1\| _{L^2(\varOmega)}^{2}
					+\|\vartheta_0\|_{H_0^1(\varOmega)}^{2}+\|\vartheta_0\|_{H_0^1(\varOmega)}^{p}
					\right).\label{energy bound 2}
				\end{align}
			In particular, estimate \eqref{energy bound 2} implies the uniformly boundedness of $\beta_m(\cdot)$ through its definition in \eqref{def al la}.
			
			\begin{remark}
				Observe that the constant appearing on the right-hand side of \eqref{energy bound 2} is independent of the final time $T$. Therefore, the estimate extends to arbitrary time intervals, and hence
			
					\begin{align}
						&
						\sup_{t\ge 0}
						\left(
						\frac{1}{2}\|\um(t)\| _{L^2(\varOmega)}^{2}
						+\frac{1}{2}\|\nabla \vartheta_m(t)\| _{L^2(\varOmega)}^{2}
						+\frac{1}{p}\|\vartheta_m(t)\|_{L^p(\varOmega)}^{p}
						\right)
						+\g\int_{0}^{\infty}\|\um(s)\| _{L^2(\varOmega)}^{2}  ds \\
						&\qquad\le
						C\left(
						\|\vartheta_1\| _{L^2(\varOmega)}^{2}
						+\|\vartheta_0\|_{H_0^1(\varOmega)}^{2}+\|\vartheta_0\|_{H_0^1(\varOmega)}^{p}
						\right).\label{energy bound 3}
					\end{align}
			
			\end{remark}

				\vskip 0.1 cm
			\noindent 
				\textbf{Step 3.}  \textit{Higher order estimates.} Multiplying \eqref{eq:NLE pro} with $\sigma_k (\xi_k^{(m)})_t(\cdot)$ and summing over $k=1,\ldots,m$, we find
			\begin{equation}
				\begin{aligned}
					(\umm,\A\um)&=-(\A\vartheta_m,\A\um)-(|\vartheta_m|^{p-2}\vartheta_m,\A\um)-\g(\um,\A\um)\\	&\qquad+\beta_m(t)(\vartheta_m,\A\um).
				\end{aligned}
			\end{equation}
			An application of integration by parts yields 
				\begin{align}
					&\frac{1}{2}\frac{d}{dt}\|\nabla\um\| _{L^2(\varOmega)}^{2}+\frac{1}{2}\frac{d}{dt}\|\A\vartheta_m\| _{L^2(\varOmega)}^{2}
					+\g\|\nabla\um\| _{L^2(\varOmega)}^{2}\nonumber\\&=-(|\vartheta_m|^{p-2}\vartheta_m,\A\um)
					 +\beta_m(t)(\vartheta_m,\A\um),
				\end{align}
		for a.e. $t\in[0,T]$. 	To estimate the nonlinear term, we integrate by parts and use the chain rule to get
			\begin{equation}
				\begin{aligned}
					(|\vartheta_m|^{p-2}\vartheta_m,\A\um)&=(\nabla(|\vartheta_m|^{p-2}\vartheta_m),\nabla\um) =(p-1)(|\vartheta_m|^{p-2}\nabla \vartheta_m,\nabla\um).		
				\end{aligned}
			\end{equation}
			Applying H\"older's inequality, we deduce  
			\begin{equation}
				(|\vartheta_m|^{p-2}\nabla \vartheta_m,\nabla\um)\le\|\vartheta_m\|_{L^{d(p-2)}(\varOmega)}^{p-2}\|\nabla \vartheta_m\|_{L^\frac{2d}{(d-2)}(\varOmega)}\|\nabla\um\| _{L^2(\varOmega)}.
			\end{equation}
			Furthermore, by the Cauchy-Schwarz inequality,
			\begin{equation}
				\beta_m (\vartheta_m,\A\um)\le|\beta_m |\,\|\nabla \vartheta_m\| _{L^2(\varOmega)}\,\|\nabla\um\| _{L^2(\varOmega)}.
			\end{equation}
			Combining the above estimates, we arrive at
			\begin{equation}
				\begin{aligned}
					&\frac{1}{2}\frac{d}{dt}\|\nabla\um\| _{L^2(\varOmega)}^{2}+\frac{1}{2}\frac{d}{dt}\|\A\vartheta_m\| _{L^2(\varOmega)}^{2}
					+\g\|\nabla\um\| _{L^2(\varOmega)}^{2}\\
					&\le(p-1)\|\vartheta_m\|_{L^{d(p-2)}(\varOmega)}^{p-2}\|\nabla \vartheta_m\|_{L^\frac{2d}{(d-2)}(\varOmega)}	\|\nabla\um\| _{L^2(\varOmega)}+|\beta_m |\,\|\nabla \vartheta_m\| _{L^2(\varOmega)}\,\|\nabla\um\| _{L^2(\varOmega)}.
				\end{aligned}
			\end{equation}
			For the values of $p$ given in \eqref{eqn-restriction},  the Sobolev embedding yields (see \eqref{embedding}), $\|\vartheta_m\|_{L^{d(p-2)}(\varOmega)}\leq C\|\nabla \vartheta_m\|_{L^2(\varOmega)}$. 
			Using \eqref{energy bound 2}, we get  uniform boundedness of $\beta_m$, and using the above Sobolev embedding, we infer 
			\begin{equation}
				\begin{aligned}
					&\frac{1}{2}\frac{d}{dt}\|\nabla\um\| _{L^2(\varOmega)}^{2}+\frac{1}{2}\frac{d}{dt}\|\A\vartheta_m\| _{L^2(\varOmega)}^{2}
					+\g\|\nabla\um\| _{L^2(\varOmega)}^{2}\\
					&\le K_1\|\nabla \vartheta_m\| _{L^2(\varOmega)}^{p-2}\|\nabla \vartheta_m\|_{L^\frac{2d}{(d-2)}}
					\|\nabla\um\| _{L^2(\varOmega)}+K_2\|\nabla \vartheta_m\| _{L^2(\varOmega)}\|\nabla\um\| _{L^2(\varOmega)},
				\end{aligned}
			\end{equation}
			for some positive constants $K_1$ and $K_2$ independent of $m$.
			Furthermore, by virtue of \eqref{energy bound 2}, the Poincar\'e and Young's inequalities, we observe 
			\begin{equation}
				\begin{aligned}
					&K_1\|\nabla \vartheta_m\| _{L^2(\varOmega)}^{p-2}	\|\nabla \vartheta_m\|_{L^\frac{2d}{(d-2)}(\varOmega)}\|\nabla\um\| _{L^2(\varOmega)}+K_2\|\nabla \vartheta_m\| _{L^2(\varOmega)}
					\|\nabla\um\| _{L^2(\varOmega)} \\
					&\qquad\leq K_3\left(\|\A\vartheta_m\| _{L^2(\varOmega)}^{2}+\|\nabla\um\| _{L^2(\varOmega)}^{2}\right),
				\end{aligned}
			\end{equation}
			where $K_3>0$ is a constant independent of $m$. Consequently,
			\begin{equation}\label{norm bound 1}
				\frac{1}{2}\frac{d}{dt}\left(\|\nabla\um\| _{L^2(\varOmega)}^{2}+\|\A\vartheta_m\| _{L^2(\varOmega)}^{2}  \right)\
				+\g\|\nabla\um\| _{L^2(\varOmega)}^{2}\leq K_3\left(\|\A\vartheta_m\| _{L^2(\varOmega)}^{2} + \|\nabla\um\| _{L^2(\varOmega)}^{2}\right).
			\end{equation}
%			Thus,
%			\begin{equation}\label{norm bound 2}
%				\frac{1}{2}\frac{d}{dt}\left(\|\nabla \vartheta_m\| _{L^2(\varOmega)}^{2}+\|\Delta\vartheta_m\| _{L^2(\varOmega)}^{2}\right)
%				\le K_3\left(\|\nabla \vartheta_m\| _{L^2(\varOmega)}^{2}+\|\Delta\vartheta_m\| _{L^2(\varOmega)}^{2}\right).
%			\end{equation}
			An application of Gr\"onwall’s inequality to \eqref{norm bound 1} gives, for all $t\in0,T],$
			\begin{equation}\label{gronwall bound}
				\|\nabla\um(t)\| _{L^2(\varOmega)}^{2}	+\|\A\vartheta_m(t)\| _{L^2(\varOmega)}^{2}\le	\left(	\|\nabla \um(0)\| _{L^2(\varOmega)}^{2}
				+\|\A\vartheta_m(0)\| _{L^2(\varOmega)}^{2}\right)\exp{(2K_3T)}.
			\end{equation}
			Integrating \eqref{norm bound 1} over $(0,t)$ and using \eqref{gronwall bound}, \eqref{ineq:pro1},and \eqref{ineq:pro3}, we obtain
				\begin{align}
					&\|\nabla \um(t)\| _{L^2(\varOmega)}^{2}+\|\A\vartheta_m(t)\| _{L^2(\varOmega)}^{2}+2\g\int_{0}^{t}\|\nabla \um(s)\| _{L^2(\varOmega)}^{2}  ds \\
					&\qquad \le	2K_3\int_{0}^{t}\left(\|\nabla \um(s)\| _{L^2(\varOmega)}^{2}+\|\A\vartheta_m(s)\| _{L^2(\varOmega)}^{2}\right)  ds \\
					&\qquad \le2K_3\left(\|\nabla \um(0)\| _{L^2(\varOmega)}^{2}+\|\A\vartheta_m(0)\| _{L^2(\varOmega)}^{2}\right)\int_{0}^{t} \exp^{2K_3T}  ds \\
					&\qquad \le K_4	\left(\|\vartheta_1\| _{H^1_0(\varOmega)}^{2}	+\|\vartheta_0\|_{D(\A)}^{2}\right)
					T\exp^{2K_3T},\label{norm bound 3}
				\end{align}
			where $K_4>0$ is a constant independent of $m$.

\vskip 0.1 cm
\noindent 
			\textbf{Step 4:} \textit{Passing to the limit as $m\to\infty$.}
			By the uniform estimates \eqref{energy bound 2} and \eqref{norm bound 3}, the sequences
			$\{\vartheta_m\}_{m=1}^\infty$, $\{\um\}_{m=1}^\infty$ and $\{\umm\}_{m=1}^\infty$ are uniformly bounded in the corresponding functional spaces. Therefore, by the \emph{Banach-Alaoglu} theorem, there exist a subsequence, still denoted by $\{\vartheta_m\}$, and a function $\vartheta$ such that
			\begin{equation}\label{passing limit 1}
				\left\{
				\begin{aligned}
					\vartheta_m &\overset{w^*}{\to} \vartheta
					&&\text{ in }L^\infty(0,T;D(\A))
					\cap L^\infty(0,\infty;H_0^1(\varOmega)),\\[1mm]
					(\vartheta_m)_t &\overset{w^*}{\to} \vartheta_t
					&&\text{ in }L^\infty(0,T;H_0^1(\varOmega))
					\cap L^\infty(0,\infty;L^2(\varOmega)),\\[1mm]
					(\vartheta_m)_t &\overset{w}{\to} \vartheta_t
					&&\text{ in }L^2(0,T;H_0^1(\varOmega))
					\cap L^2(0,\infty;L^2(\varOmega)).
				\end{aligned}
				\right.
			\end{equation}
			Since $H_0^1(\varOmega)\hookrightarrow H^{-1}(\varOmega) \text{ and }  L^2(\varOmega)\hookrightarrow H^{-1}(\varOmega)$, it follows from \eqref{passing limit 1} that $\{\vartheta_m\}_{m=1}^\infty$ and $\{\um\}_{m=1}^\infty$ are uniformly bounded in $L^\infty(0,\infty;H^{-1}(\varOmega))$. Consequently, by \eqref{inner prod um}, the sequence $\{\umm\}_{m=1}^\infty$ is uniformly bounded in $L^\infty(0,\infty;H^{-1}(\varOmega))$. Hence, by the \emph{Banach-Alaoglu} theorem, there exists a subsequence, still denoted by $\{\umm\}$, such that
			\begin{equation}\label{passing limit 2}
				\left\{
				\begin{aligned}
					&\umm \overset{w^*}{\to} \vartheta_{tt}&&  \text{ in } L^\infty(0,\infty;H^{-1}(\varOmega)),\\
					&\umm \overset{w^*}{\to} \vartheta_{tt}&&  \text{ in } L^\infty(0,T;L^2(\varOmega)).			
				\end{aligned}
				\right.
			\end{equation}
			Furthermore, by the \emph{Aubin--Lions compactness theorem}, we obtain %Lions-Mangenes theorem, we obtain
			\begin{equation}\label{cont limit}
					\vartheta_m \to \vartheta    \  \text{ in } \ C([0,T];H_0^1(\varOmega))\ \text{ and }\ 
					\um \to \vartheta_t  \  \text{ in } \ C([0,T];L^2(\varOmega)).
			\end{equation}
			Invoking the \emph{Strauss Lemma},  we get 
			\begin{align}
			\vartheta\in 	C_w([0,T];D(\A))\ \text{ and }\ \vartheta_t\in C_w([0,T];H^1_0(\varOmega)).
			\end{align}
			
			We now pass to the limit in the Galerkin system \eqref{eq:NLE pro}. 
%			\begin{equation}\label{eq:NLE pro}
%				\left\{
%				\begin{aligned}
%					(\umm,w_k)&=(\Delta\vartheta_m,w_k)-(|\vartheta_m|^{p-2}\vartheta_m,w_k)-\g(\um,w_k)\\
%					&\quad+(-\|\um\| _{L^2(\varOmega)}^{2}+\|\nabla \vartheta_m\| _{L^2(\varOmega)}^{2}+\|\vartheta_m\|_{L^p(\varOmega)}^{p})(\vartheta_m,w_k),\qquad t\in(0,T),\\
%					\vartheta_m(0)&=\vartheta_{0m}=\frac{\mathrm{P}_m \vartheta_0}{\|\mathrm{P}_m \vartheta_0\| _{L^2(\varOmega)}},\\
%					\um(0)&=\vartheta_{1m}=\mathrm{P}_m \vartheta_1-(\mathrm{P}_m \vartheta_1,\vartheta_{0m})\vartheta_{0m}.
%				\end{aligned}
%				\right.
%			\end{equation}
			Now, by virtue of \eqref{passing limit 1}, \eqref{passing limit 2} and \eqref{cont limit}, we obtain
			\begin{equation}\label{non linear con 0}
				\left\{
				\begin{aligned}
					(\umm,w_k) &\to (\vartheta_{tt},w_k) \  &&\text{ as } m\to\infty,\\
					(\A\vartheta_m,w_k) &\to (\A\vartheta,w_k) \ &&\text{ as } m\to\infty,\\
					(\g(\um),w_k) &\to (\g(\vartheta_t),w_k) \ &&\text{ as } m\to\infty.
				\end{aligned}
				\right.
			\end{equation}
			To pass to the limit in the nonlinear source term appearing in \eqref{eq:NLE pro}, we first establish the convergence of
			\begin{equation}
				|\vartheta_m|^{p-2}\vartheta_m \to |\vartheta|^{p-2}\vartheta
				\quad \text{in } L^2(\varOmega).
			\end{equation}
			Define, 
				$F(s)=|s|^{p-2}s.$ 
			By the mean value theorem, for almost every $x\in\varOmega$ and $t\in(0,T)$, we have
			\begin{equation}
				\begin{aligned}
					\bigl||\vartheta_m|^{p-2}\vartheta_m-|\vartheta|^{p-2}\vartheta\bigr|&=|F(\vartheta_m)-F(\vartheta)| \\
					&=\left|\int_0^1 \frac{d}{d\theta}F\bigl(\theta \vartheta_m+(1-\theta)\vartheta\bigr) d \theta\right| \\
					&=\left|\int_0^1 F'\bigl(\theta \vartheta_m+(1-\theta)\vartheta\bigr)(\vartheta_m-\vartheta) d \theta\right| \\
					&\le (p-1)|\vartheta_m-\vartheta|\int_0^1\bigl|\theta \vartheta_m+(1-\theta)\vartheta\bigr|^{p-2} d \theta \\
					&\le (p-1)|\vartheta_m-\vartheta|\bigl(|\vartheta_m|+|\vartheta|\bigr)^{p-2}.\vspace{-5.3cm}
				\end{aligned}
			\end{equation}
			Consequently,
			\begin{equation}\label{non linear con 1}
				\bigl\||\vartheta_m|^{p-2}\vartheta_m-|\vartheta|^{p-2}\vartheta\bigr\|_{L^2(\varOmega)}^2\le K_5\bigl\||\vartheta_m-u|(|\vartheta_m|+|u|)^{p-2}
				\bigr\|_{L^2(\varOmega)}^2,
			\end{equation}
			where $K_5>0$ is some constant  independent of $m$. Applying H\"older's inequality with conjugate exponents $\frac{d}{(d-2)}$ and $\frac{d}{2}$ for $ d\geq3$, we obtain
			\begin{equation}
				\bigl\||\vartheta_m|^{p-2}\vartheta_m-|\vartheta|^{p-2}\vartheta\bigr\|_{L^2(\varOmega)}^2\le K_6\|\vartheta_m-\vartheta\|_{L^\frac{2d}{(d-2)}(\varOmega)}^2
				\left(\|\vartheta_m\|_{L^{d(p-2)}(\varOmega)}^{2(p-2)}+\|\vartheta\|_{L^{d(p-2)}(\varOmega)}^{2(p-2)}\right),
			\end{equation}
			where $K_6>0$ is some constant  independent of $m$.
			Since $d\geq 3$ and $2\leq p\leq \frac{2(d-1)}{d-2}$, we have
			$d(p-2)\leq \frac{2d}{d-2}$. Combining this observation with
			\eqref{embedding}, \eqref{energy bound 2}, and \eqref{cont limit} in
			\eqref{non linear con 1}, we deduce 
			\begin{equation}\label{non linear con 1*}
				\bigl\||\vartheta_m|^{p-2}\vartheta_m-|\vartheta|^{p-2}\vartheta\bigr\|_{L^2(\varOmega)}^2\le K_7\|\vartheta_m-\vartheta\|_{H_0^1(\varOmega)}^2\to 0
				\  \text{ as } \ m\to\infty.
			\end{equation}
			where $K_7>0$ is some constant independent of $m$. For
			$d=1,2$, the above result follows trivially. Therefore,
			\begin{equation}\label{non linear con 2}
				\int_{0}^{T}\bigl\||\vartheta_m|^{p-2}\vartheta_m-|\vartheta|^{p-2}\vartheta\bigr\|_{L^2(\varOmega)}^2\le K_7T\|\vartheta_m-\vartheta\|_{L^{\infty}(0,T;H_0^1(\varOmega))}^2\to 0\  \text{ as } \ m\to\infty.
			\end{equation}
			Now, by \eqref{norm1}, \eqref{energy bound 2}, and \eqref{cont limit}, we obtain
			\begin{align}
					&\int_{0}^{T}\Bigl\|\|\um(t)\|_{L^2(\varOmega)}^{2}\vartheta_m(t)-\|\vartheta_t(t)\|_{L^2(\varOmega)}^{2}\vartheta (t)\Bigr\|_{L^2(\varOmega)}^{2}  dt \\
					&=\int_{0}^{T}\Bigl\|\bigl(\|\um(t)\|_{L^2(\varOmega)}^{2}-\|\vartheta_t(t)\|_{L^2(\varOmega)}^{2}\bigr)\vartheta_m(t)+\|\vartheta_t(t)\|_{L^2(\varOmega)}^{2}\bigl(\um(t)-\vartheta (t)\bigr)\Bigr\|_{L^2(\varOmega)}^{2}  dt \\
					&\leq2\int_{0}^{T}\Bigl\|\bigl(\|\um(t)\|_{L^2(\varOmega)}^{2}-\|\vartheta_t(t)\|_{L^2(\varOmega)}^{2}\bigr)\vartheta_m(t)\Bigr\|_{L^2(\varOmega)}^{2}  dt \\
					&\qquad+2\int_{0}^{T}\Bigl\|\|\vartheta_t(t)\|_{L^2(\varOmega)}^{2}\bigl(\vartheta_m(t)-\vartheta (t)\bigr)\Bigr\|_{L^2(\varOmega)}^{2}  dt \\
					&\leq2\sup_{t\in[0,T]}\bigl(\|\um(t)\|_{L^2(\varOmega)}+\|\vartheta_t(t)\|_{L^2(\varOmega)}\bigr)^2\int_{0}^{T}\|\um(t)-\vartheta_t(t)\|^2_{L^2(\varOmega)}\,\|\vartheta_m(t)\|_{L^2(\varOmega)}^{2}  dt \\
					&\qquad+2\sup_{t\in[0,T]}\|\vartheta_t(t)\|_{L^2(\varOmega)}^{2}\int_{0}^{T}\|\vartheta_m(t)-\vartheta (t)\|_{L^2(\varOmega)}^{2}  dt \\
					&\leq2\sup_{t\in[0,T]}\bigl(\|\um(t)\|_{L^2(\varOmega)}+\|\vartheta_t(t)\|_{L^2(\varOmega)}\bigr)^2\|\um-\vartheta_t\|^2_{L^2(0,T;L^2(\varOmega))} \\
					&\qquad+2\big(\sup_{t\in[0,T]}	\|\vartheta_t(t)\|_{L^2(\varOmega)}^{2}\big)\|\vartheta_m-\vartheta\|^2_{L^2(0,T,L^2(\varOmega))}\\
					& \to 0 \ \text{ as }\ m\to\infty.\label{non linear con 3}
				\end{align}
			Similarly, by \eqref{norm1}, \eqref{energy bound 2}, and \eqref{cont limit}, we have
			\begin{align}
					&\int_{0}^{T}\Bigl\|\|\nabla \vartheta_m(t)\|_{L^2(\varOmega)}^{2}\vartheta_m(t)-\|\nabla \vartheta (t)\|_{L^2(\varOmega)}^{2}\vartheta (t)\Bigr\|_{L^2(\varOmega)}^{2}  dt \\
					&=\int_{0}^{T}\Bigl\|\bigl(\|\nabla \vartheta_m(t)\|_{L^2(\varOmega)}^{2}-\|\nabla \vartheta (t)\|_{L^2(\varOmega)}^{2}\bigr)\vartheta_m(t)
					+\|\nabla \vartheta (t)\|_{L^2(\varOmega)}^{2}\bigl(\vartheta_m(t)-\vartheta (t)\bigr)\Bigr\|_{L^2(\varOmega)}^{2}  dt \\
					&\leq2\int_{0}^{T}\Bigl\|\bigl(\|\nabla \vartheta_m(t)\|_{L^2(\varOmega)}^{2}-\|\nabla \vartheta (t)\|_{L^2(\varOmega)}^{2}\bigr)\vartheta_m(t)
					\Bigr\|_{L^2(\varOmega)}^{2}  dt \\
					&\qquad+2\int_{0}^{T}\Bigl\|\|\nabla \vartheta (t)\|_{L^2(\varOmega)}^{2}\bigl(\vartheta_m(t)-\vartheta (t)\bigr)\Bigr\|_{L^2(\varOmega)}^{2}  dt \\
					&\leq2\sup_{t\in[0,T]}\bigl(\|\nabla \vartheta_m(t)\|_{L^2(\varOmega)}+\|\nabla \vartheta (t)\|_{L^2(\varOmega)}\bigr)^2	\int_{0}^{T}
					\|\nabla \vartheta_m(t)-\nabla \vartheta (t)\|^2_{L^2(\varOmega)}\|\vartheta_m(t)\|_{L^2(\varOmega)}^{2}  dt \\
					&\qquad+2\sup_{t\in[0,T]}\|\nabla \vartheta (t)\|_{L^2(\varOmega)}^{2}\int_{0}^{T}\|\vartheta_m(t)-\vartheta (t)\|_{L^2(\varOmega)}^{2}  dt \\
					&\leq2\sup_{t\in[0,T]}\bigl(\|\nabla \vartheta_m(t)\|_{L^2(\varOmega)}+\|\nabla \vartheta (t)\|_{L^2(\varOmega)}\bigr)^2\|\vartheta_m-\vartheta\|^2_{L^2(0,T;H_0^1)} \\
					&\qquad+2\big(\sup_{t\in[0,T]}\|\nabla \vartheta (t)\|_{L^2(\varOmega)}^{2}\big)\|\vartheta_m-\vartheta \|_{L^2(0,T;L^2(\varOmega))}^{2}\\
					&\to 0\ \text{ as }\ m\to\infty . \label{non linear con 4}
				\end{align}
			Finally, using \eqref{norm1}, \eqref{energy bound 2}, and \eqref{cont limit}, we obtain for the values of $p$ given in \eqref{eqn-restriction} that 
			\begin{align}
					&\int_{0}^{T}\Bigl\|\|\vartheta_m(t)\|_{L^p(\varOmega)}^{p}\vartheta_m(t)-\|\vartheta (t)\|_{L^p(\varOmega)}^{p}\vartheta (t)\Bigr\|_{L^2(\varOmega)}^{2}  dt \\
					&=\int_{0}^{T}\Bigl\|\bigl(\|\vartheta_m(t)\|_{L^p(\varOmega)}^{p}-\|\vartheta (t)\|_{L^p(\varOmega)}^{p}\bigr)\vartheta_m(t)+\|\vartheta (t)\|_{L^p(\varOmega)}^{p}
					\bigl(\vartheta_m(t)-\vartheta (t)\bigr)\Bigr\|_{L^2(\varOmega)}^{2}  dt \\
					&\leq2\int_{0}^{T}
					\Bigl\|\bigl(\|\vartheta_m(t)\|_{L^p(\varOmega)}^{p}-\|\vartheta (t)\|_{L^p(\varOmega)}^{p}\bigr)\vartheta_m(t)\Bigr\|_{L^2(\varOmega)}^{2}  dt \\
					&\qquad+2\int_{0}^{T}\Bigl\|\|\vartheta (t)\|_{L^p(\varOmega)}^{p}\bigl(\vartheta_m(t)-\vartheta (t)\bigr)\Bigr\|_{L^2(\varOmega)}^{2}  dt \\
					&\leq2K_8\sup_{t\in[0,T]}\Bigl(\|\vartheta_m(t)\|_{L^p(\varOmega)}^{2(p-1)}+\|\vartheta (t)\|_{L^p(\varOmega)}^{2(p-1)}\Bigr)\int_{0}^{T}
					\|\vartheta_m(t)-\vartheta (t)\|^2_{L^p(\varOmega)}\|\vartheta_m(t)\|_{L^2(\varOmega)}^{2}  dt \\
					&\qquad+2\sup_{t\in[0,T]}\|\vartheta (t)\|_{L^p(\varOmega)}^{2p}\int_{0}^{T}\|\vartheta_m(t)-\vartheta (t)\|_{L^2(\varOmega)}^{2}  dt \\
					&\leq2K'_8\sup_{t\in[0,T]}\Bigl(\|\vartheta_m(t)\|_{L^p(\varOmega)}^{2(p-1)}+\|\vartheta (t)\|_{L^p(\varOmega)}^{2(p-1)}\Bigr)
					\|\vartheta_m-\vartheta\|^2_{L^2(0,T;L^p(\varOmega))} \\
					&\qquad +2\big(\sup_{t\in[0,T]}\|\vartheta (t)\|_{L^p(\varOmega)}^{p}\big)\|\vartheta_m-\vartheta \|_{L^2(0,T;L^2(\varOmega))}^{2}\\
					&\to 0\ 
					  \text{ as }\ m\to\infty , \label{non linear con 5}
				\end{align}
			where $K_8,K'_8$ are some positive constant independent of $m$.
			
			\vskip 0.1cm 
			\noindent 
			\textbf{Step 5: Existence of a strong solution.}
			Combining \eqref{non linear con 0}, \eqref{non linear con 2},
			\eqref{non linear con 3}, \eqref{non linear con 4}, and
			\eqref{non linear con 5} with \eqref{eq:NLE pro},
			\eqref{ineq:pro1}, \eqref{ineq:pro3}, and \eqref{cont limit},
			we pass to the limit as $m\to\infty$ and obtain
					\begin{equation}\label{eq:existence}
				\left\{
				\begin{aligned}
					(\vartheta_{tt},\psi)&=-( \mathcal{A}\vartheta,\psi)-(|\vartheta|^{p-2}\vartheta,\psi)-\g(\vartheta_t,\psi) \\
					&\quad+\Bigl(-\|\vartheta_t\|_{L^2(\varOmega)}^{2}+\|\nabla \vartheta\|_{L^2(\varOmega)}^{2}+\|\vartheta\|_{L^p(\varOmega)}^{p}\Bigr)(\vartheta,\psi),
					\qquad\text{ in }\ \varOmega\times(0,T),\\
					\vartheta(0)&=\vartheta_0,\\
					\vartheta_t(0)&=\vartheta_1,
				\end{aligned}
				\right.
			\end{equation}
		for every $\psi\in V_n$. By a standard density argument, \eqref{eq:existence}
		extends to every $\psi\in L^2(\varOmega)$. Hence, $\vartheta$ is a strong solution of
		problem \eqref{eq:mains}-\eqref{MID}.
		Moreover, by \eqref{cont limit} and \eqref{ineq:pro1}, together with $
			\|\vartheta_m(t)\|_{L^2(\varOmega)}^2=1,
			\text{ for all }m\in\mathbb{N}, t\in[0,T],$
		we obtain
		\begin{equation}\label{Norm =1}
			\|\vartheta (t)\|_{L^2(\varOmega)}^2
			=\lim_{m\to\infty}\|\vartheta_m(t)\|_{L^2(\varOmega)}^2
			=1,
			\  t\in[0,T].
		\end{equation}
		Therefore, $\vartheta (t)\in\mathcal{M},$ $ t\in[0,T],$	which shows that the constraint is preserved by the limiting solution. Finally, it follows that
			\begin{equation}
				\left\{
				\begin{aligned}
					&\vartheta\in C_w\bigl((0,T;D(\A))\bigr)\cap C\bigl([0,T];H_0^1(\varOmega)\cap\mathcal{M}\bigr),\\
					&\vartheta_t\in C\bigl([0,T];L^2(\varOmega)\bigr)\cap C_w\big([0,T];H^1_0(\varOmega)\big).
				\end{aligned}
				\right.
			\end{equation}
			We establish the existence of a strong solution.

\vskip 0.1cm
\noindent 
			\textbf{Step 6: Uniqueness.}
			Let \(\vartheta\) and \(\widetilde{\vartheta} \) be two strong solutions of problem \eqref{eq:mains}-\eqref{MID} corresponding to the same initial data \(\vartheta_0\) and \(\vartheta_1\).  Define for $t\in[0,T]$ 
			\begin{equation}\label{def_uniqueness}
				\begin{aligned}
					&\lambda_\vartheta (t)
					=
					-\|\vartheta_t(t)\|_{L^2(\varOmega)}^2
					+\|\nabla \vartheta (t)\|_{L^2(\varOmega)}^2
					+\|\vartheta (t)\|_{L^p(\varOmega)}^p,\\
					&\text{and}\
					\lambda_{\widetilde{\vartheta}}(t)
					=
					-\|\widetilde{\vartheta}_t(t)\|_{L^2(\varOmega)}^2
					+\|\nabla \widetilde{\vartheta}(t)\|_{L^2(\varOmega)}^2
					+\|\widetilde{\vartheta}(t)\|_{L^p(\varOmega)}^p.
				\end{aligned}
			\end{equation}
			Setting \(w=\vartheta-\widetilde{\vartheta}\), we find that \(w\) satisfies
			\begin{equation}\label{eq_uniqueness}
				\left\{
				\begin{aligned}
					w_{tt}(t)
					&=
					-\mathcal{A} w(t)
					-\g w_t(t)
					+\Bigl(|\widetilde{\vartheta}(t)|^{p-2}\widetilde{\vartheta}(t)-|\vartheta (t)|^{p-2}\vartheta (t)\Bigr) \\
					&\quad
					+\lambda_\vartheta (t)\vartheta (t)-\lambda_{\widetilde{\vartheta}}(t)\widetilde{\vartheta}(t)
					\ \text{ in }\ \varOmega\times(0,T),\\
					w(0)&=0\qquad \text{in } \varOmega,\\
					w_t(0)&=0 \qquad \text{in } \varOmega,
				\end{aligned}
				\right.
			\end{equation}
			where $\lambda_\vartheta, \lambda_{\widetilde{\vartheta}}$ are defined in \eqref{def_uniqueness}. On taking the  \(L^2\)-inner product of \eqref{eq_uniqueness} with \(w_t(\cdot)\), we obtain
			\begin{equation}\label{eq_uniqueness_energy}
				\begin{aligned}
					(w_{tt}(t),w_t(t))
					&=
					-(\A w(t),w_t(t))
					-\g\|w_t(t)\|_{L^2(\varOmega)}^2 \\
					&\quad
					+\Bigl(|\widetilde{\vartheta}(t)|^{p-2}\widetilde{\vartheta}(t)-|\vartheta (t)|^{p-2}\vartheta (t),w_t(t)\Bigr) \\
					&\quad
					+\lambda_\vartheta (t)(\vartheta (t),w_t(t))
					-\lambda_{\widetilde{\vartheta}}(t)(\widetilde{\vartheta}(t),w_t(t)),
				\end{aligned}
			\end{equation}
	for a.e. $t\in(0,T)$. 		Using  the identity
			\begin{equation*}
				\lambda_\vartheta(\vartheta,w_t)
				-\lambda_{\widetilde{\vartheta}}(\widetilde{\vartheta},w_t)
				=
				\lambda_{\vartheta}(w,w_t)
				+
				\bigl(\lambda_\vartheta-\lambda_{\widetilde{\vartheta}}\bigr)(\widetilde{\vartheta},w_t),
			\end{equation*}
			and integration by parts together with relation \eqref{identity 1} on the above equality   \eqref{eq_uniqueness_energy}, we deduce  
				\begin{align}
					&\frac{1}{2}\frac{d}{dt}
					\Bigl(
					\|w_t(t)\|_{L^2(\varOmega)}^2
					+\|w(t)\|_{H_0^1(\varOmega)}^2
					\Bigr)
					+\g\|w_t(t)\|_{L^2(\varOmega)}^2\\
					&\le
					K_9
					\Bigl(
					(|\vartheta (t)|+|\widetilde{\vartheta}(t)|)^{p-2}|w(t)|,
					\,w_t(t)
					\Bigr) 
					+
					|\lambda_\vartheta (t)|
					\|w(t)\|_{L^2(\varOmega)}
					\|w_t(t)\|_{L^2(\varOmega)} \\
					&\qquad+
					|\lambda_\vartheta (t)-\lambda_{\widetilde{\vartheta}}(t)|
					\|\widetilde{\vartheta}(t)\|_{L^2(\varOmega)}
					\|w_t(t)\|_{L^2(\varOmega)},\label{ineq_uniqueness_1}
				\end{align}
			where \(K_9>0\) is a positive constant.
			Next, we estimate the difference \(\lambda_\vartheta-\lambda_{\widetilde{\vartheta}}\). By definition, we have 
				\begin{align}
					|\lambda_\vartheta (t)-\lambda_{\widetilde{\vartheta}}(t)|
					&=
					\Bigl|
					\|\widetilde{\vartheta}_t(t)\|_{L^2(\varOmega)}^2-\|\vartheta_t(t)\|_{L^2(\varOmega)}^2
					+\|\nabla \vartheta (t)\|_{L^2(\varOmega)}^2-\|\nabla \widetilde{\vartheta}(t)\|_{L^2(\varOmega)}^2 \\
					&\qquad
					+\|\vartheta (t)\|_{L^p(\varOmega)}^p-\|\widetilde{\vartheta}(t)\|_{L^p(\varOmega)}^p
					\Bigr| \\
					&\le
					K_{10}
					\Bigl[
					(\|\vartheta_t(t)\|_{L^2(\varOmega)}+\|\widetilde{\vartheta}_t(t)\|_{L^2(\varOmega)})
					\|w_t(t)\|_{L^2(\varOmega)}
					\\
					&\qquad
					+
					(\|\nabla \vartheta (t)\|_{L^2(\varOmega)}+\|\nabla \widetilde{\vartheta}(t)\|_{L^2(\varOmega)})
					\|\nabla w(t)\|_{L^2(\varOmega)}
					\\
					&\qquad
					+
					\bigl(
					\|\vartheta (t)\|_{L^{2(p-1)}(\varOmega)}^{p-1}
					+
					\|\widetilde{\vartheta}(t)\|_{L^{2(p-1)}(\varOmega)}^{p-1}
					\bigr)
					\|w(t)\|_{L^2(\varOmega)}
					\Bigr]
					\\
					&\le
					K_{11}
					\Bigl(
					\|w_t(t)\|_{L^2(\varOmega)}
					+
					\|w(t)\|_{H_0^1(\varOmega)}
					+
					\|w(t)\|_{L^2(\varOmega)}
					\Bigr)
					\\
					&\le
					K_{12}
					\Bigl(
					\|w_t(t)\|_{L^2(\varOmega)}
					+
					\|w(t)\|_{H_0^1(\varOmega)}
					\Bigr),\label{ineq_uniqueness_2}
				\end{align}
			where \(K_{10}\), \(K_{11}\), and \(K_{12}\) are positive constants.
			
			Substituting  \eqref{ineq_uniqueness_2} into \eqref{ineq_uniqueness_1}, and using H\"older's inequality together with \eqref{energy bound 1}, we obtain
			\begin{equation}
				\begin{aligned}
					&\frac{1}{2}\frac{d}{dt}
					\Bigl(
					\|w_t(t)\|_{L^2(\varOmega)}^2
					+
					\|w(t)\|_{H_0^1}^2
					\Bigr)
					+\g\|w_t(t)\|_{L^2(\varOmega)}^2\\
					&\le
					K_{14}
					\bigl(
					\|\vartheta (t)\|_{L^{d(p-2)}(\varOmega)}^{p-2}
					+
					\|\widetilde{\vartheta}(t)\|_{L^{d(p-2)}(\varOmega)}^{p-2}
					\bigr)
					\|w(t)\|_{L^\frac{2d}{(d-2)}(\varOmega)}
					\|w_t(t)\|_{L^2(\varOmega)}
					\\
					&\qquad
					+
					K_{13}
					\|w(t)\|_{L^2(\varOmega)}
					\|w_t(t)\|_{L^2(\varOmega)}
					\\
					&\qquad
					+
					K_{12}
					\bigl(
					\|w_t(t)\|_{L^2(\varOmega)}
					+
					\|w(t)\|_{H_0^1(\varOmega)}
					\bigr)
					\|\widetilde{\vartheta}(t)\|_{L^2(\varOmega)}
					\|w_t(t)\|_{L^2(\varOmega)},
				\end{aligned}
			\end{equation}
		for a.e. $t\in(0,T)$,	where \(K_{13}\) and \(K_{14}\) are some positive constants.
			Applying Young's inequality together with \eqref{embedding},  \eqref{Norm =1} and \eqref{energy bound 1} in the above inequality, we arrive at
			\begin{equation}\label{ODE_uniqueness}
				\left\{
				\begin{aligned}
					\frac{d}{dt}
					\Bigl(
					\|w_t(t)\|_{L^2(\varOmega)}^2
					+
					\|w(t)\|_{H_0^1(\varOmega)}^2
					\Bigr)
					&\le
					K_{15}
					\Bigl(
					\|w_t(t)\|_{L^2(\varOmega)}^2
					+
					\|w(t)\|_{H_0^1(\varOmega)}^2
					\Bigr),
					\\
					w(0)&=0,
					\\
					w_t(0)&=0,
				\end{aligned}
				\right.
			\end{equation}
			where \(K_{15}>0\) is a some positive constant.
			Applying Gr\"onwall's inequality to \eqref{ODE_uniqueness} and using the initial conditions, we obtain
			\begin{equation}
				\|w(t)\|_{H_0^1(\varOmega)}^2
				+
				\|w_t(t)\|_{L^2(\varOmega)}^2
				=0,
				\ \text{ for all }\ 
				t\in[0,T].
			\end{equation}
			Hence,
			$	w(t)\equiv 0,
				\ 
				t\in[0,T]\ \text{ in }\ H^1_0(\varOmega),$
		together with $w_t(t)\equiv 0,
			\ 
			t\in[0,T]\ \text{ in }\ L^2(\varOmega)$
 proves the uniqueness of strong solutions. 
		\end{proof}
		\section{Asymptotic Analysis}\label{AA}
The uniform-in-time energy bound in \eqref{energy bound 3} provides a natural starting point for studying the long-time behaviour of strong solutions to \eqref{eq:mains}-\eqref{MID}. The energy dissipation due to the damping term plays a key role in the asymptotic analysis and suggests that the solution approaches a stationary state as $t\to\infty$.
In this section, we investigate the asymptotic behaviour of strong solutions and establish their convergence as $t\to\infty$. We further characterize the possible limiting states and derive the stationary equation satisfied by any such limit.
\subsection{Subsequential asymptotic analysis}
		We fix a bounded smooth domain $\varOmega\subset\mathbb{R}^d$ with
		$C^2$ boundary. 
		Throughout the discussion, the exponent $p$ is assumed to satisfy \eqref{p range}.
		We are concerned with the asymptotic dynamics of solutions to the following problem posed on $(0,T)\times\varOmega$:
		\begin{equation}\label{asym1}
			\left\{
			\begin{aligned}
				\vartheta_{tt}+\g \vartheta_t+\A\vartheta+|\vartheta|^{p-2}\vartheta&=\lambda \vartheta
				&& \text{ in }\ \varOmega\times(0,T),\\
%				u&=0
%				&& \text{on }\partial\varOmega\times(0,T),\\
				\vartheta(0)&=\vartheta_0
				&& \text{ in }\ \varOmega,\\
				\vartheta_t(0)&=\vartheta_1
				&& \text{ in }\ \varOmega.
			\end{aligned}
			\right.
		\end{equation}
		where for $t\in[0,T]$ 
			\begin{equation}\label{def lam asym}
			\lambda(t)=-\|\vartheta_t(t)\|_{L^2(\varOmega)}^2+\|\nabla \vartheta (t)\|_{L^2(\varOmega)}^2+\|\vartheta (t)\|_{L^p(\varOmega)}^p,
		\end{equation}
		we assume that $\|\vartheta_0\|_{L^2(\varOmega)}=1,\ (\vartheta_0,\vartheta_1)=0.$
%		\begin{equation}\label{asym2}
%			\|\vartheta_0\|_{L^2(\varOmega)}^2=1,\qquad(\vartheta_0,\vartheta_1)_{L^2(\varOmega)}=0.
%		\end{equation}
		Throughout this section, we write $\vartheta (t):=\vartheta(\cdot,t).$
%		\begin{equation}
%			\vartheta (t):=u(\cdot,t).
%		\end{equation}
		\begin{definition}
			The energy functional associated with \eqref{asym1} is defined by
			\begin{equation}\label{energy}
				E(\vartheta (t),\vartheta_t(t)):=\frac12\|\vartheta_t(t)\|_{L^2(\varOmega)}^2+\frac12\|\nabla \vartheta (t)\|_{L^2(\varOmega)}^2+\frac1p\|\vartheta (t)\|_{L^p(\varOmega)}^p,
				\ \  t\ge0.
			\end{equation}
		\end{definition}
		
		\begin{lemma}\label{sub seq conv soln}
		Let $\vartheta_0\in D(\mathcal{A})\cap\mathcal{M}$ and $\vartheta_1\in H_0^1(\varOmega)$ with $(\vartheta_0,\vartheta_1)=0$. 	Let $\vartheta$ be the unique strong solution of \eqref{asym1}. Then there exists a sequence $\{t_n\}_{n\in\mathbb{N}}$, $t_n\to\infty$ as $n\to\infty$ and $\vartheta_{\infty}\in H_0^1(\varOmega)\cap\mathcal{M}$  such that 
		\begin{align}\label{eqn-limit}
			\vartheta (t_n)\xrightarrow{w} \vartheta_{\infty}\  \text{ in }\ H_0^1(\varOmega)\ \text{ and }\ 	\lambda(t_n)\to \lambda_{\infty} \ \text{ as } \ n\to\infty. 
		\end{align} 
		\end{lemma}
		
		\begin{proof}
		Taking the $L^2$-inner product of  \eqref{asym1} by $\vartheta_t(t)$ and integrating over $\varOmega$, we obtain
	\begin{equation}
		\begin{aligned}
			(\vartheta_{tt}(t),\vartheta_t(t))+\g (\vartheta_t(t),\vartheta_t(t))+(\A\vartheta(t),\vartheta_t(t))+\bigl(|\vartheta (t)|^{p-2}\vartheta (t),\vartheta_t(t)\bigr)=
			\lambda(t)\,(\vartheta (t),\vartheta_t(t)),
		\end{aligned}
	\end{equation}
for a.e. $t\in[0,T]$. By an application of \emph{Lions--Magenes Lemma} (\cite[Lemma 1.2, Chapter 3]{RT-01})
	Since $(\vartheta (t),\vartheta_t(t))_{L^2(\varOmega)}=0$, for all $t\ge0$, it follows that
	\begin{equation}\label{asym3}
		\frac{d}{dt}
		\left(\frac12\|\vartheta_t(t)\|_{L^2(\varOmega)}^2 + \frac12\|\nabla \vartheta (t)\|_{L^2(\varOmega)}^2 + \frac1p\|\vartheta (t)\|_{L^p(\varOmega)}^p
		\right)	+ \g\|\vartheta_t(t)\|_{L^2(\varOmega)}^2=0.
	\end{equation}
	Consequently, \eqref{asym3} can be rewritten as
	\begin{equation}\label{asym4}
		\frac{d}{dt}E(\vartheta (t),\vartheta_t(t))+\g\|\vartheta_t(t)\|_{L^2(\varOmega)}^2=0.
	\end{equation}
	Therefore, $\frac{d}{dt}E(\vartheta (t),\vartheta_t(t))=	-\g\|\vartheta_t(t)\|_{L^2(\varOmega)}^2\le 0.$
	Hence, the energy functional $E(\vartheta (t),\vartheta_t(t))$ is non-increasing along trajectories of \eqref{asym1}.
	Now integrating \eqref{asym4} with respect to time from $0\ \text{to}\ t$, we have
	\begin{equation}\label{asym5}
		\begin{aligned}
			E(\vartheta (t),\vartheta_t(t))+\g\int_{0}^{t}\|\vartheta_t(s)\|_{L^2(\varOmega)}^2\ ds=	E(0)\le  K_{16}\big( \norm{\vartheta_0}^2_{H^1_0}+\norm{\vartheta_0}^p_{H^1_0}+\norm{\vartheta_1}^2_{L^2(\varOmega)} \big)<\infty,
		\end{aligned}
	\end{equation}
		where $E(0):=	E(\vartheta (0),\vartheta_t(0))$. Since $E(\vartheta (t),\vartheta_t(t))$ is non-increasing and bounded below, there exists $E_{\infty}\geq 0$ such that 
		\begin{equation}\label{Energy tends}
			E(\vartheta (t),\vartheta_t(t))\searrow \E \  \text{ as }\  t\to\infty.
		\end{equation}
		Now, the right-hand side of \eqref{asym5} is independent of time. Therefore, letting $t\to\infty$ in \eqref{asym5}, we obtain
		\begin{equation}
			\E+\g\int_{0}^{\infty}\|\vartheta_t(s)\|_{L^2(\varOmega)}^2\ ds=E(0).
		\end{equation}
		This implies 
		\begin{equation}\label{velocity bound}
			\int_{0}^{\infty}\|\vartheta_t(s)\|_{L^2(\varOmega)}^2\ ds=\frac{E(0)-\E}{\g}<\infty.
		\end{equation}
		By the help of  \eqref{energy} and \eqref{asym5}, we have
		\begin{equation}\label{asym6}
				\|\vartheta_t(t)\|_{L^2(\varOmega)}^2\le 2E(0),\
				\|\nabla \vartheta (t)\|_{L^2(\varOmega)}^2\le 2E(0),\
				\|\vartheta (t)\|_{L^p(\varOmega)}^p\le pE(0).
		\end{equation}
With  the help of \emph{Banach-Alaoglu theorem} \cite[Theorem C.17]{JCR-20}, we can extract an  increasing sequence $\{t_n\}_{n=1}^{\infty}$, $t_n\to\infty\ \text{as}\ n\to\infty$,  such that 
		\begin{equation}\label{weak asym conv}
			\vartheta (t_n) \overset{w}{\to}\vartheta_\infty\  \text{ as }\ n\to\infty\ \text{ in }\ H^1_0(\varOmega).
		\end{equation}
		It follows from \eqref{velocity bound} and \eqref{asym6} that, after relabelling the subsequence, we obtain
		\begin{equation}\label{velocity asym conv}
			\vartheta_t(t_n)\to0 \ \text{ as }\ n\to\infty,\  \text{ in }\ L^2(\varOmega).
		\end{equation}
		For the values of $p$ specified in \eqref{p range}, the embedding $H_0^1(\varOmega)\hookrightarrow L^p(\varOmega)$ is compact. Consequently, we have
		\begin{equation}\label{Lp asym conv}
			\vartheta (t_n)\to\vartheta_\infty \  \text{ as }\ n\to\infty\  \text{ in }\ L^p(\varOmega).
		\end{equation}
		Since $\vartheta (t)\in\mathcal{M},$ for all $t\geq 0$, using   \eqref{Lp asym conv}, we get 
		\begin{equation}
	1=	\lim_{n\to\infty}\|\vartheta (t_n)\|_{L^2(\varOmega)}^2=	\norm{\vartheta_\infty}^2 _{L^2(\varOmega)}.
		\end{equation}
			Taking $t=t_n$ in \eqref{energy} and passing $n\to\infty$	with the help of \eqref{Energy tends}, \eqref{velocity asym conv} and \eqref{Lp asym conv}, we deduce  
		\begin{equation}\label{Enery asym 2}
			\lim_{n\to\infty}\Bigl (\frac12\|\nabla \vartheta (t_n)\|_{L^2(\varOmega)}^2\Bigr )+\frac1p\|\vartheta_\infty\|_{L^p(\varOmega)}^p=\E.
		\end{equation}
		Let us now discuss the convergence of $\lambda(t_n)$, where $\lambda(t)$ is defined in \eqref{def lam asym}. Therefore, 
		\begin{equation}\label{asym lamda}
			\lambda(t_n)=\bigr (-\|\vartheta_t(t_n)\|_{L^2(\varOmega)}^2+\|\nabla \vartheta (t_n)\|_{L^2(\varOmega)}^2+\|\vartheta (t_n)\|_{L^p(\varOmega)}^p\bigr ).
		\end{equation}
		Using \eqref{asym6} in \eqref{asym lamda} and then applying the \emph{Bolzano--Weierstrass theorem}, we infer that there exists 	$\l\in\mathbb{R}$ such that, after passing to a subsequence and
		relabelling it,	$	\lambda(t_n)\to\l	\ \text{ as }\ n\to\infty.	$ Taking limit $n\to\infty$ in \eqref{asym lamda} and using \eqref{velocity asym conv} and \eqref{Lp asym conv}, we find 
		\begin{equation}\label{eqn-conv-4}
			\lim_{n\to\infty}\lambda(t_n)=\lim_{n\to\infty}\bigl (\|\nabla \vartheta (t_n)\|_{L^2(\varOmega)}^2\bigr )+\|\vartheta_\infty\|^p_{L^p(\varOmega)}. 
		\end{equation}
		The convergence given in \eqref{Enery asym 2} implies 
		\begin{equation}\label{l def}
			\begin{aligned}
				\lim_{n\to\infty}\lambda(t_n)=2\E-\frac2p\|\vartheta_\infty\|_{L^p(\varOmega)}^p+\|\vartheta_\infty\|^p_{L^p(\varOmega)}=2\E +\frac{(p-2)}{p}\|\vartheta_\infty\|^p_{L^p(\varOmega)}=:\lambda_{\infty},
			\end{aligned}
		\end{equation}
		which completes the proof. 
		\end{proof}

		\begin{theorem}\label{apct}
	%		Let $\varOmega\subset\mathbb{R}^{d}$ be a bounded smooth domain, 
%			and let
%			$u=u(x,t)$ be a global solution of
%			\begin{equation}\label{asym precom}
%				\vartheta_{tt}+\g \vartheta_t+\A\vartheta+|\vartheta|^{p-2}\vartheta
%				=\lambda(t)u
%				\ \text{ in }\ \varOmega\times(0,\infty),
%			\end{equation}
%%			subject to the homogeneous Dirichlet boundary condition
%%			\begin{equation}
%%				u=0
%%				\quad\text{on }\partial\varOmega\times(0,\infty),
%%			\end{equation}
%			and the initial conditions
%			\begin{equation}
%				\vartheta(0)=\vartheta_0\in D(\A),
%				\qquad
%				\vartheta_t(0)=\vartheta_1\in H_0^1(\varOmega),
%			\end{equation}
%			where
%			\begin{equation}
%				(\vartheta_0,\vartheta_1)_{L^2(\varOmega)}=0,
%				\qquad
%				\|\vartheta (t)\|_{L^2(\varOmega)}=1,
%				\quad\text{for all }t\geq0.
%			\end{equation}
				Let $\vartheta$ be the unique strong solution of problem  \eqref{asym1}. Define 
			\begin{equation}
				\mathcal{X}:=H_0^1(\varOmega)\times L^2(\varOmega),
				\ 
				%\mathcal{U} (t):=(\vartheta (t),\vartheta_t(t)), \ t\geq 0. 
				\mathcal{U}_0:=(\vartheta _0,\vartheta_1). 
			\end{equation}
			Then the orbit
			\begin{equation}
				\mathcal{O}(\vartheta_0,\vartheta_1)
				=
				\left\{
				(\vartheta (t),\vartheta_t(t)):t\geq0
				\right\}
			\end{equation}
			is precompact in $\mathcal{X},$ that is, $\overline{\mathcal{O}(\vartheta_0,\vartheta_1)}^{\mathcal{X}}$ is compact in $\mathcal{X}$.
		\end{theorem}

		\begin{proof}
			We divide the proof into the following steps:
			\vskip 0.1cm
			\noindent 
			\textbf{Step 1:} \emph{Semigroup decomposition.}
			Consider the linear operator
			\begin{equation}
				G=
				\begin{pmatrix}
					0 & I\\
					-\A& -\g I
				\end{pmatrix}, \ \ D(G)=D(\A)\times H_0^1(\varOmega).
			\end{equation}
		 The associated linear damped wave equation
			$$ \vartheta_{tt}+\g \vartheta_{t}+\A\vartheta=0,$$
			generates a strongly continuous semigroup $\{e^{Gt}:t\ge 0\}$ on $\mathcal{X}.$ By Duhamel principle, we have 
			\begin{equation}
				\vartheta (t)=e^{Gt}\mathcal{U}_0
				+
				\int_{0}^{t}
				e^{G(t-s)}
				\begin{pmatrix}
					0\\
					-|\vartheta(s)|^{p-2}\vartheta(s)+\lambda(s)\vartheta(s)
				\end{pmatrix}
				ds,\ \ t\geq 0. 
			\end{equation}
			We define, $T_1(t):=e^{Gt}.$ Then, we have 
			\begin{equation}
				\vartheta (t)=T_{1}(t)	\mathcal{U}_0+\int_{0}^{t}
				T_1(t-s)
				\begin{pmatrix}
					0\\
						-|\vartheta(s)|^{p-2}\vartheta(s)+\lambda(s)\vartheta(s)
				\end{pmatrix}
				ds.
			\end{equation}
			\vskip 0.1cm
			\noindent 
			\textbf{Step 2:} \emph{Decay of $T_1$.}
			Consider the linear damped wave equation
%			\begin{equation}\label{asym precompt 1}
%				\left\{
%				\begin{aligned}
%					&\rho_{tt}+\g \rho_{t}-\Delta \rho=0&&\quad\text{ in }\ (0,\infty)\times\varOmega,\\
%					&\rho=0&&\quad\text{ on }\ (0,\infty)\times {\partial\varOmega},\\
%					&\rho(0)=\rho_0,\ 	\rho_{t}(0)=\rho_{1}&&\quad\text{ in }\ \varOmega.\\
%				\end{aligned}
%				\right.
%			\end{equation}
				\begin{equation}\label{asym precompt 1}
				\left\{
				\begin{aligned}
					&\rho_{tt}+\g \rho_{t}+\A \rho=0&&\quad\text{ in }\ (0,\infty)\times\varOmega,\\
					%&\rho=0&&\quad\text{ on }\ (0,\infty)\times {\partial\varOmega},\\
					&\rho(0)=\rho_0,\ 	\rho_{t}(0)=\rho_{1}&&\quad\text{ in }\ \varOmega.\\
				\end{aligned}
				\right.
			\end{equation}
			Let  $\chi(t):=(\rho(t),\rho_{t}(t)),$ $ \chi(0)=\chi_0=(\rho_{0},\rho_{1})\in\mathcal{X},$
			so that
		$
			\chi(\cdot)=T_1(\cdot)\chi_{0}\in C([0,\infty);\mathcal{X}). 
		$
		%	where $T_1(t)$ denotes the semigroup generated by $G$ on $X := H^1_0(\varOmega)\times L^2(\varOmega).$
			We shall show that there exist constants $C,\omega>0$ such that
			$$ \|T_1(t)\|_{\mathcal{L}(\mathcal{X})}\le Ce^{-\omega t}.$$
			Define the energy functional
			\begin{equation}
				E^\rho(t)=
				\frac{1}{2}\|\rho_t(t)\|_{L^2(\varOmega)}^2+\frac{1}{2}\|\nabla \rho(t)\|_{L^2(\varOmega)}^2,\ \ t\geq 0. 
			\end{equation}
		Taking the $L^2$-inner product with the equation \eqref{asym precompt 1} by $\rho_t$, we obtain 
			\begin{equation}\label{eqn-inner-product}
				\langle\rho_{tt},\rho_t\rangle+\g \|\rho_t\|_{L^2(\varOmega)}^2+\langle\A \rho,\rho_t\rangle=0.
			\end{equation}
%			Since,
%			\begin{equation}
%				(\rho_{tt},\rho_t)=\frac{d}{dt}\left(\frac{1}{2}\|\rho_t\|_{L^2(\varOmega)}^2\right)\quad\text{and}\quad 	(-\Delta \rho,\rho_t)=\frac{d}{dt}\left(\frac{1}{2}\|\nabla \rho\|_{L^2(\varOmega)}^2\right),
%			\end{equation}
		Then by using \emph{Lions--Magenes Lemma}, 	it follows that
			\begin{equation}
				\frac{d}{dt}\left(\frac{1}{2}\|\rho_t\|_{L^2(\varOmega)}^2+\frac{1}{2}\|\nabla \rho\|_{L^2(\varOmega)}^2\right)=
				-\g \|\rho_t\|_{L^2(\varOmega)}^2.
			\end{equation}
			Therefore, for all $t\geq 0$, we have 
			\begin{equation}
				E^\rho_t(t)=-\g \|\rho_t(t)\|_{L^2(\varOmega)}^2.\label{energy-decay}
			\end{equation}
			Introduce the \emph{modified Lyapunov functional} by 
			\begin{equation}
				\widehat{F}(t)=E^\rho(t)+\varepsilon (\rho(t),\rho_t(t)),
			\end{equation}
		for all $t\geq 0$, 	where $\varepsilon>0$ will be chosen sufficiently small. Using Young's and Poincar\'e's inequalities, we obtain
			\begin{equation}\label{YPLf}
				|(\rho,\rho_t)|\le\delta \|\rho_t\|_{L^2(\varOmega)}^2+\frac{1}{4\delta\sigma_1} \|\nabla \rho\|_{L^2(\varOmega)}^2,
			\end{equation}
			where $\delta>0$ and $\sigma_1$ is the first eigenvalue of the Dirichlet Laplacian. Consequently, for $\varepsilon>0$ sufficiently small, there exist positive constants $c_1$ and $c_2$ such that
			\begin{equation}
				c_1 E^\rho(t)\le \widehat{F}(t)\le c_2 E^\rho(t),
				\label{equivalence}
			\end{equation}
		for all $t\geq 0$.	We know that  $\frac{d}{dt}(\rho,\rho_t)	=\|\rho_t\|_{L^2(\varOmega)}^2+\langle \rho,\rho_{tt}\rangle.$
			Since, $\rho_{tt}=-\g \rho_t-\A \rho$ in $H^{-1}(\varOmega)$, we have
			\begin{equation}
				\langle\rho,\rho_{tt}\rangle=	-\g (\rho,\rho_t)+\langle\rho,-\A \rho\rangle.
			\end{equation}
		Therefore, we have 
			\begin{equation}\label{LFD}
				\frac{d}{dt}(\rho,\rho_t)=
				\|\rho_t\|_{L^2(\varOmega)}^2- \|\nabla \rho\|_{L^2(\varOmega)}^2-\g (\rho,\rho_t).
			\end{equation}
			Differentiating $\hat{F}$ with respect to $t$ and using \eqref{energy-decay} and \eqref{LFD}, we obtain
			\begin{equation}
				\widehat{F}_t= E^\rho_t+\varepsilon\frac{d}{dt} (\rho,\rho_t)=
				-(\g-\varepsilon)\|\rho_t\|_{L^2(\varOmega)}^2-\varepsilon \|\nabla \rho\|_{L^2(\varOmega)}^2-
				\varepsilon \g (\rho,\rho_t).
			\end{equation}
			By the help of \eqref{YPLf}, we have
			\begin{equation}
				\widehat{F}_t 
				\le
				-(\g-\varepsilon-\varepsilon \g \delta)
				\|\rho_t\|_{L^2(\varOmega)}^2-	
				\varepsilon
				\left(1-\frac{\g}{4\delta\lambda_1}\right)
				\|\nabla \rho\|_{L^2(\varOmega)}^2.
			\end{equation}
			Choosing first $\delta>\frac{\g}{4\sigma_1}>0$ and then $0<\varepsilon<\frac{4\g \sigma_1}{4\sigma_1+\g^2}$ sufficiently small, we obtain
			\begin{equation}
				\widehat{F}_t 
				\le
				-c
				\left(
				\|\rho_t\|_{L^2(\varOmega)}^2
				+
				\|\nabla \rho\|_{L^2(\varOmega)}^2
				\right)=
				-2cE^\rho,
			\end{equation}
			for some constant $c>0$. Using \eqref{equivalence}, we infer that,
			$
			E^\rho(t)\ge\frac{1}{c_2}F(t),
			$
			and therefore
			$
				\widehat{F}_t(t)\le-\frac{2c}{c_2}F(t).
			$
			Setting
			$
			\omega=\frac{2c}{c_2},
			$
			we obtain
			$
				\widehat{F}_t(t)\le-\omega 	\widehat{F}(t).
			$
			By Gr\"onwall's inequality,
			\begin{equation}
				\widehat{F}(t)
				\le
					\widehat{F}(0)e^{-\omega t}, \ \ t\geq 0. 
			\end{equation}
			Using again the equivalence of $	\widehat{F}$ and $E^\rho$, we conclude that
			$
			E^\rho(t)
			\le
			C e^{-\omega t} E^\rho(0),
			$
			 and 
			\begin{equation}
				\|(\rho(t),\rho_t(t))\|_{H_0^1(\varOmega)\times L^2(\varOmega)}\le (C e^{-\omega t})^\frac{1}{2}\|(\rho_0,\rho_1)\|_{H_0^1(\varOmega)\times L^2(\varOmega)}.
			\end{equation}
			Equivalently,
			\begin{equation}\label{exp decay}
				\|T_1(t)\|_{\mathcal{L}(\mathcal{X})}\le Ce^{-\omega t},
			\end{equation}
		so that 
			\begin{equation}
				\sup_{\|\chi_0\|_{\mathcal{X}}\le R}
				\|T_1(t)\chi_0\|_{\mathcal{X}}\le CR e^{-\omega t}\to 0,
				\ \ t\to\infty.
			\end{equation}
			Thus, the assumptions required in part (i) of \emph{} \cite[Proposition 3.2, (3.7)--(3.9)]{Webb79} are satisfied. Furthermore, upon setting $T_2=0$, the hypothesis in part (ii) is trivially fulfilled.
			Now we claim $\Theta(\cdot,\cdot):\mathcal{X}\to \mathcal{X}$ defined by  
			$$\Theta(\vartheta,\varrho)=\begin{pmatrix}
				0\\
				-|\vartheta|^{p-2}\vartheta-\lambda(\vartheta,\varrho)\vartheta
			\end{pmatrix}
			$$ is compact, where 
			\begin{align}\label{eqn-lambda}
			\lambda(\vartheta,\varrho)=-\|\varrho\|_{L^2(\varOmega)}^2+\|\nabla\vartheta\|_{L^2(\varOmega)}^2+\|\vartheta\|_{L^p(\varOmega)}^p.
		\end{align}
			The compactness of $\Theta(\cdot,\cdot)$ follows once we establish the 	compactness of the operator
			\[
			\widetilde{\Theta}:H_0^1(\varOmega)\times L^2(\varOmega)\to L^2(\varOmega),
			\ \ 
			\widetilde{\Theta}(\vartheta,\varrho)
			=
			-|\vartheta|^{p-2}\vartheta-\lambda(\vartheta,\varrho)\vartheta,
			\]
			which is proved in Proposition \ref{compact B}. In particular, for any
			sequence of times $t_n\to\infty$, Proposition \ref{compact B} allows us
			to extract a subsequence such that
			$
			\widetilde{\Theta}(\vartheta (t_n),\vartheta_t(t_n))
			$
			converges in $L^2(\varOmega)$. Moreover, since $\lambda(t_n)$ is a bounded
			sequence of real numbers by the   estimates given in \eqref{asym6}, the
			\emph{Bolzano--Weierstrass theorem} yields, after passing to a further
			subsequence and relabelling it, some $\lambda_*\in\mathbb{R}$ such that
		$	\lambda(t_n)\to\lambda_*$ as $n\to\infty$. 
			Thus, condition (iii) (\cite[ Eq. (3.9)]{Webb79}) is also satisfied. Consequently, by
			\emph{}, the positive orbit
			\[
			\mathcal{O}(\vartheta_0,\vartheta_1)
			=
			\left\{(\vartheta (t),\vartheta_t(t)):t\geq0\right\}
			\]
			is precompact in $\mathcal{X}$. 
		\end{proof}
		\begin{proposition}\label{compact B}
			Assume that $p$ satisfies \eqref{p range}. Define the mapping $\tilde{\Theta}:H_0^1(\varOmega)\times L^2(\varOmega)\to L^2(\varOmega)$
			by
			\begin{equation*}
				\tilde{\Theta}(u,\tilde{\rho})=-|u|^{p-2}u+\lambda(u,\tilde{\rho})u,
			\end{equation*}
			where $\lambda(u,\tilde{\rho})$ is defined in \eqref{eqn-lambda}. 
%			\begin{equation*}
%				\lambda(u,\tilde{\rho})
%				=-\|\tilde{\rho}\|_{L^2(\varOmega)}^2
%				+\|\nablau\|_{L^2(\varOmega)}^2
%				+\|\vartheta\|_{L^p(\varOmega)}^p.
%			\end{equation*}
			Then $\tilde{\Theta}$ maps bounded subsets of
			$H_0^1(\varOmega)\times L^2(\varOmega)$ into relatively compact subsets of $L^2(\varOmega)$.
		\end{proposition}
		\begin{proof}
			Let  $\{(\vartheta_n,\tilde{\rho}_n)\}_{n\ge 1}
			\subset H_0^1(\varOmega)\times L^2(\varOmega)
			$
			be a bounded sequence. Then there exists \(D>0\) such that
			\begin{equation*}
				\|\vartheta_n\|_{H_0^1(\varOmega)}
				+\|\tilde{\rho}_n\|_{L^2(\varOmega)}
				\le D,
				\ n\ge1.
			\end{equation*}
			Since $H_0^1(\varOmega)$ is reflexive, after extraction of a subsequence, we may assume that
			\begin{equation*}
				\vartheta_n \overset{w}{\to} u
				\ \text{ as }\ n\to\infty\ \text{ in }\ H_0^1(\varOmega).
			\end{equation*}
			By the Rellich-Kondrachov theorem \cite[Theorem 6, pg-190]{DM+DZ-98},
			\begin{equation*}
				\vartheta_n \to u
				\ \text{ strongly in }\ L^q(\varOmega) \ \text{ for any }\ 1\le q<\frac{2d}{d-2}.
			\end{equation*}
		Therefore, 	we obtain $\vartheta_n\to u \text{ strongly in }L^{2(p-1)}(\varOmega),$ since $ 2(p-1)<\frac{2d}{d-2}.$
			Using the inequality \eqref{identity 1} together with  Hölder's inequality, we get
			\begin{align*}
				\bigl\|
				|\vartheta_n|^{p-2}\vartheta_n-|u|^{p-2}u
				\bigr\|_{L^2(\varOmega)}
				&\le
				C
				\Bigl(
				\|\vartheta_n\|_{L^{2(p-1)}(\varOmega)}^{p-2}
				+
				\|\vartheta\|_{L^{2(p-1)}(\varOmega)}^{p-2}
				\Bigr)
				\|\vartheta_n-u\|_{L^{2(p-1)}(\varOmega)}
				\\
				&\to 0 \ \text{ as }\ n\to\infty.
			\end{align*}
			Hence, $
				|\vartheta_n|^{p-2}\vartheta_n\to |u|^{p-2}u \ \text{ strongly in } \ L^2(\varOmega).$
			Let us now define
			$$\lambda_n=-\|\tilde{\rho}_n\|_{L^2(\varOmega)}^2+\|\nabla \vartheta_n\|_{L^2(\varOmega)}^2+\|\vartheta_n\|_{L^p(\varOmega)}^p.$$
			Since \(\{\vartheta_n\}\) and \(\{\tilde{\rho}_n\}\) are bounded, it follows that $|\lambda_n|\le \tilde{D},$ for some $\tilde{D}>0$.
			Therefore, after passing to a subsequence if necessary, there exists
			$\lambda_*\in\R,$ such that $\lambda_n\to\lambda_* \text{ as }\ n\to\infty.$
			Since, $ \vartheta_n\to u \text{ strongly in } L^2(\varOmega),$ we have
			\begin{equation*}
				\|\lambda_n\vartheta_n-\lambda_*u\|_{L^2(\varOmega)}
				\le
				|\lambda_n|
				\,\|\vartheta_n-u\|_{L^2(\varOmega)}
				+
				|\lambda_n-\lambda_*|
				\,\|\vartheta\|_{L^2(\varOmega)}\to0\ \text{ as }\ n\to\infty.
			\end{equation*}
			Therefore, $\lambda_n\vartheta_n\to\lambda_* u\ \text{ strongly in }\ L^2(\varOmega).$
			Combining the above convergences yields
			\begin{equation*}
				\begin{aligned}
				\left\|
				\tilde{\Theta}(\vartheta_n,\tilde{\rho}_n)-\bigl(-|u|^{p-2}u+\lambda_* u\bigr)\right\|_{L^2(\varOmega)}&\le\bigl\||\vartheta_n|^{p-2}\vartheta_n-|u|^{p-2}u\bigr\|_{L^2(\varOmega)}+\|\lambda_n\vartheta_n-\lambda_* u\|_{L^2(\varOmega)}\\
					&\to 0\ \text{ as } \ n\to\infty.
				\end{aligned}
			\end{equation*}
			Hence the sequence $
			\{\tilde{\Theta}(\vartheta_n,\tilde{\rho}_n)\}$ possesses a convergent subsequence in $L^2(\varOmega)$. Therefore the image of every bounded subset of $H_0^1(\varOmega)\times L^2(\varOmega)$
			is relatively compact in \(L^2(\varOmega)\). Consequently, $\tilde{\Theta}:H_0^1(\varOmega)\times L^2(\varOmega)\to L^2(\varOmega),$ is a compact operator.
		\end{proof}
		
	\begin{remark}
		The $\lambda_*$ need not necessarily be equal to $\l$ obtained  in \eqref{eqn-limit}. Indeed, by the \emph{Bolzano--Weierstrass} theorem, there exists a convergent subsequence, and different choices of convergent subsequences may lead to different limits. Hence, $\l$ obtained from the subsequence in Lemma \ref{sub seq conv soln} need not necessarily coincide with the $\lambda_*$ obtained in Proposition \ref{compact B}.
	\end{remark}
The following proposition is motivated by \cite[Theorem~7.6.1]{AH+MAJ-15}, where the authors established the asymptotic convergence of the velocity to zero for a semilinear wave equation with linear damping and without any constraints.

		\begin{proposition}\label{convergence of ut to 0}
			As a consequence of the asymptotic precompactness of
			$\{\vartheta_t(t):t\geq0\}$ in $L^2(\varOmega)$, we have
		\begin{align}\label{eqn-conv}
				\vartheta_t(t)\to0\ \text{ strongly in }\ L^2(\varOmega)\ \text{ as }\ t\to\infty.
		\end{align}
		\end{proposition}
		\begin{proof}
		By \eqref{velocity asym conv}, there exists a subsequence $\{\vartheta_t(t_n)\}_{n=1}^\infty$ converging strongly to $0$ in $L^2(\varOmega)$. We claim that the above estimate holds uniformly as $t\to\infty$. To this end, recalling \eqref{asym5}, we obtain
			\begin{equation}\label{asym u and ut in H and L}
				\vartheta\in L^\infty(0,\infty;H^1_0(\varOmega))
				\ \text{ and }\ 
				\vartheta_t\in L^{\infty}(0,\infty;L^2(\varOmega))\cap L^2(0,\infty;L^2(\varOmega)),
			\end{equation}
	and 	 it follows that  $\vartheta,\vartheta_t\in  L^\infty(0,\infty;L^2(\varOmega)) \hookrightarrow L^\infty(0,\infty;H^{-1}(\varOmega)).$
			Moreover, since $p$ satisfies the range specified in \eqref{p range},
			the Sobolev embedding \eqref{embedding} yields $ L^\infty(0,\infty;L^{2p-2}(\varOmega))
			\hookrightarrow L^\infty(0,\infty;H^{-1}(\varOmega)).$
			Consequently, $|\vartheta|^{p-2}\vartheta\in L^\infty(0,\infty;H^{-1}(\varOmega)). $
			Furthermore, since the operator
			$\A:H^1_0(\varOmega)\to H^{-1}(\varOmega)$ is continuous, there exists a
			constant $K_{17}>0$ such that
		\begin{equation*}
			\|\A\vartheta\|_{H^{-1}(\varOmega)}
			\le K_{17}\|\vartheta\|_{H^1_0(\varOmega)},
			\end{equation*}
			and hence $\A\vartheta\in L^\infty(0,\infty;H^{-1}(\varOmega)).$ Therefore,
			\begin{equation}\label{asym utt bound}
				\vartheta_{tt}=-\A\vartheta-\gamma \vartheta_t-|\vartheta|^{p-2}\vartheta+\Bigl(-\|\vartheta_t\|_{L^2(\varOmega)}^2+\|\nabla\vartheta\|_{L^2(\varOmega)}^2+\|\vartheta\|_{L^p(\varOmega)}^p\Bigr)\vartheta\in L^\infty(0,\infty;H^{-1}(\varOmega)).
			\end{equation}
			Since $\vartheta_t\in L^2(0,\infty;L^2(\varOmega)) \hookrightarrow L^2(0,\infty;H^{-1}(\varOmega))$ and
			$\vartheta_{tt}\in L^\infty(0,\infty;H^{-1}(\varOmega))$, by \emph{Barbalat's lemma}
			\cite[Theorem 5]{BL}, we obtain
			\begin{equation}\label{H^-1 limit}
				\vartheta_t(t)\to 0\ \text{ in }H^{-1}(\varOmega)\  \text{ as }\ t\to\infty .
			\end{equation}
			Let $\{t_n\}_{n=1}^{\infty}$ be an arbitrary sequence such that $t_n\to\infty$ as $n\to\infty$. Since $\vartheta_t(t_n)$ is bounded in $L^2(\varOmega)$ (see \eqref{asym6}), there exists a subsequence, still denoted by $\{t_n\}_{n=1}^{\infty}$, and $y\in L^2(\varOmega)$ such that
			\begin{equation}
				\vartheta_t(t_n)\overset{w}{\to} y\ \text{ in }L^2(\varOmega)\ \text{ as }\ n\to\infty.
			\end{equation}
			Since $L^2(\varOmega)\hookrightarrow\hookrightarrow H^{-1}(\varOmega)$, we also have
		$	\vartheta_t(t_n)\to y \text{ in }H^{-1}(\varOmega)\ \text{ as }\ n\to\infty.$
			On the other hand, from \eqref{H^-1 limit}, we infer 
			\begin{equation}
				\vartheta_t(t_n)\to 0\ \text{ in }\ H^{-1}(\varOmega)\ \text{ as }\ n\to\infty.
			\end{equation}
			By the uniqueness of limits in $H^{-1}(\varOmega),$ we obtain $y=0$. Therefore every weak accumulation point of ${\vartheta_t}(\cdot)\ \text{in}\ L^2(\varOmega)$ is zero, and consequently
			\begin{equation}\label{weak con ut 0}
				\vartheta_t(t)\overset{w}{\to}0 \ \text{ in }\ L^2(\varOmega) \ \text{ as }\ t\to\infty.
			\end{equation}
		The asymptotic  precompactness (see Theorem \ref{apct}) and  the fact that $\vartheta \in L^\infty(0,\infty;L^2(\varOmega))$ imply that for any arbitrary  sequence $\{t_n\}_{n=1}^\infty$ such that $t_n\to\infty$ as $n\to\infty$, there exists a subsequence, still denoted by $\{t_n\}_{n=1}^{\infty}$ and $v\in L^2(\varOmega)$ such that
			\begin{equation}
				\vartheta_t(t_n)\to v\ \text{ in }\ L^2(\varOmega)\  \text{ as }\ n\to\infty.
			\end{equation}
			By using  result  \eqref{weak con ut 0}, we have $v=0$ and the conclusion \eqref{eqn-conv} follows. 
		\end{proof}
		The convergence in \eqref{eqn-conv} shows that the velocity vanishes asymptotically in $L^2(\varOmega)$. To further analyze the asymptotic behavior of the solution, we consider time-translated trajectories. This will allow us to establish the strong convergence in $H_0^1(\varOmega)$  of a suitable subsequence of the sequence obtained in Lemma \ref{sub seq conv soln}  and to identify the equation satisfied by its limit.
\begin{lemma}\label{lem:uniform_shift}
	Let $\tau>0$ and let $\{t_n\}_{n=1}^{\infty}$ be the sequence introduced in Lemma \ref{sub seq conv soln}. Define the time-translated functions
	\begin{equation}\label{uk(t) def}
		\vartheta_n(t):=\vartheta(t+t_n),\ t\in[0,\tau].
	\end{equation}
	Then, for $p$ satisfying \eqref{p range}, we obtain 
	\begin{equation}\label{eqn-conv-1}
		\sup_{t\in[0,\tau]}\|\vartheta_n(t)-\vartheta_\infty\|_{L^p(\varOmega)}\to0\ \text{ as }\ n\to\infty.
	\end{equation}
	Moreover, we have 
		\begin{equation}\label{uk limit asym H^1_0}
		\sup_{t\in[0,\tau]}\|\vartheta_n(t)-\vartheta_\infty\|_{H_0^1(\varOmega)}\to 0\ \text{ as }\ n\to\infty,
	\end{equation}
\end{lemma}
\begin{proof}
	For every $t\in[0,\tau]$, using the triangle inequality and the H\"older's inequality, we infer 
	\begin{equation}\label{eqn-first}
		\begin{aligned}
			\|\vartheta_n(t)-\vartheta_\infty\|_{L^2(\varOmega)}
			&\leq \|\vartheta_n(t)-\vartheta (t_n)\|_{L^2(\varOmega)}
			+\|\vartheta (t_n)-\vartheta_\infty\|_{L^2(\varOmega)}
			\\
			&\leq \int_{t_n}^{t_n+\tau}\|\vartheta_t(s)\|_{L^2(\varOmega)}ds+\|\vartheta (t_n)-\vartheta_\infty\|_{L^2(\varOmega)}\\&
			\leq \sqrt{\tau}
			\left(\int_{t_n}^{t_n+\tau}\|\vartheta_t(s)\|_{L^2(\varOmega)}^2  ds\right)^{1/2}
			+\|\vartheta (t_n)-\vartheta_\infty\|_{L^2(\varOmega)}.
		\end{aligned}
	\end{equation}
		Since \eqref{velocity bound} holds, the first term tends to zero, while
		\eqref{Lp asym conv} implies that the second term tends to zero. Hence,
	\begin{equation}\label{uk limit asym}
		\sup_{t\in[0,\tau]}\|\vartheta_n(t)-\vartheta_\infty\|_{L^2(\varOmega)}\to 0 \ \text{ as } \ n\to\infty.
	\end{equation}
	Moreover, by Hölder's inequality,
	\begin{equation}
		\|\vartheta_n(t)-\vartheta_\infty\|_{L^p(\varOmega)}^p
		\leq
		\|\vartheta_n(t)-\vartheta_\infty\|_{L^2(\varOmega)}
		\|\vartheta_n(t)-\vartheta_\infty\|_{L^{2(p-1)}(\varOmega)}^{p-1}.
	\end{equation}
	Since $\{\vartheta_n\}$ is uniformly bounded in $H_0^1(\varOmega)$, \eqref{asym6}
	and the Sobolev embedding \eqref{embedding} imply that
	$
	\|\vartheta_n(t)-\vartheta_\infty\|_{L^{2(p-1)}(\varOmega)}
$
	is uniformly bounded for $p$ satisfying \eqref{p range}. Hence, by
	\eqref{uk limit asym}, we conclude that \eqref{eqn-conv-1} holds.

%		\begin{remark}\label{conv in lp}
%			By Hölder's inequality,	$\|\vartheta_n(t)-\vartheta_\infty\|_{L^p(\varOmega)}^p\leq\|\vartheta_n(t)-\vartheta_\infty\| _{L^2(\varOmega)}\|\vartheta_n(t)-\vartheta_\infty\|_{L^{2(p-1)}}^{p-1}.$
%			Since, $\{\vartheta_n\}$ is uniformly bounded in $H_0^1(\varOmega)$ as in \eqref{asym6}, the Sobolev embedding \eqref{embedding} with \eqref{p range} gives, 
%		$\sup_{t\in[0,\tau]}\|\vartheta_n(t)-\vartheta_\infty\|_{L^p(\varOmega)}\to 0,\ \text{ as } k\to\infty.$
%		\end{remark}
%		\begin{lemma}\label{asym  of u in H10}
%			By theorem \eqref{aymp precomp theorem} and Lemma \eqref{lem:uniform_shift}, we have
%			\begin{equation}\label{uk limit asym H^1_0}
%				\sup_{t\in[0,\tau]}\|\vartheta_n(t)-\vartheta_\infty\|_{H_0^1(\varOmega)}\to 0,\ \text{ as }\ n\to\infty,
%			\end{equation}
%			where,	$\vartheta_n(t):=u(t_n+t).$
%			\end{lemma}
		
				Assume, by contradiction, that \eqref{uk limit asym H^1_0} does not hold.
				Then there exist $\varepsilon>0$, a subsequence (still denoted by
				$\vartheta_n$), and points $s_n\in[0,\tau]$ such that
				\begin{equation}\label{contra_shift}
					\|\vartheta_n(s_n)-\vartheta_\infty\|_{H_0^1(\varOmega)}=\|\vartheta(t_n+s_n)-\vartheta_\infty\|_{H_0^1(\varOmega)}
					\ge \varepsilon,\ \text{ for all }\ n.
				\end{equation}
				Since $[0,\tau]$ is compact, there exists a subsequence, again denoted by
				$s_n$, and a point $s_*\in[0,\tau]$ such that $$s_n\to s_*\ \text{ as }\ n\to\infty.$$
				By Theorem \ref{apct}, the semiflow is asymptotically
				precompact in $H_0^1(\varOmega)\times L^2(\varOmega)$. Hence, the sequence $\vartheta(t_n+s_n)$
				admits a subsequence, still denoted by $\vartheta(t_n+s_n)$, such that
				\begin{equation*}
					\vartheta(t_n+s_n)\to v\ \text{ as }\ n\to\infty \ \text{ strongly in }\ H_0^1(\varOmega),
				\end{equation*}
				for some $v\in H_0^1(\varOmega)$. Since the embedding $H_0^1(\varOmega)\hookrightarrow L^2(\varOmega)$ is continuous, we also have
				\begin{equation}\label{L2limitv}
					\vartheta(t_n+s_n)\to v\ \text{ strongly in }L^2(\varOmega).
				\end{equation}
		As discussed above, the convergence of the first term in \eqref{eqn-first} yields
				\begin{equation}\label{shiftL2}
					\|\vartheta(t_n+s_n)-\vartheta (t_n)\|_{L^2(\varOmega)}
					\to0.
				\end{equation}
				Since, $\vartheta (t_n)\to \vartheta_\infty
				\ \text{ in }L^2(\varOmega)$ (see \eqref{Lp asym conv}), 
				it follows from \eqref{shiftL2} that
				\begin{equation*}
					\vartheta(t_n+s_n)
					\to \vartheta_\infty
					\ \text{ in }\ L^2(\varOmega).
				\end{equation*}
				Combining this with \eqref{L2limitv} and using the uniqueness of limits in
				$L^2(\varOmega)$, we conclude that, $v=\vartheta_\infty.$ Therefore,
				\begin{equation*}
					\vartheta(t_n+s_n)\to \vartheta_\infty\ \text{ strongly in }H_0^1(\varOmega),
				\end{equation*}
				which contradicts \eqref{contra_shift}. Hence, 
				\eqref{uk limit asym H^1_0} follows.
		\end{proof}

		\begin{lemma}\label{lamda t+tk conv}
		Let $\{t_n\}_{n=1}^\infty$ be the sequence obtained  in Lemma \ref{sub seq conv soln}, with $t_n\to\infty$. Then there exists $\l\in\mathbb{R}$, as defined in \eqref{l def}, such that
		\begin{equation}\label{eqn-conv-2}
			\lambda(t+t_n)\to\l \ \text{ as }\ n\to\infty,\ \text{ for every }\ t\in[0,\tau].
		\end{equation}
	%	for every $t\in[0,\tau]$.
		\end{lemma}
		\begin{proof}
			By the definition of $\lambda(\cdot)$  given in \eqref{def lam asym}, we have 
			\begin{equation}\label{eqn-lambda-n}
				\lambda(t+t_n)=-\|\vartheta_t(t+t_n)\|_{L^2(\varOmega)}^2+\|\nabla \vartheta(t+t_n)\|_{L^2(\varOmega)}^2+\|\vartheta(t+t_n)\|_{L^p(\varOmega)}^p,
			\end{equation}
		for $t\in[0,\tau]$.	By the help of \eqref{eqn-conv}, we clearly have $-\|\vartheta_t(t+t_n)\|_{L^2(\varOmega)}^2\to 0\ \text{ as }\ n\to\infty$. We infer from \eqref{Energy tends} that  $E(\vartheta (t),\vartheta_t(t))\searrow\E$.  Now using the result \eqref{Lp asym conv} and \eqref{eqn-conv}, we get
			\begin{align}
	E_{\infty}&=	\lim_{n\to\infty} E(\vartheta_n(t),(\vartheta_n)_t(t))\nonumber	\\&	= \lim_{n\to\infty} \left[\frac12\|\vartheta_t(t+t_n)\|_{L^2(\varOmega)}^2+\frac12\|\nabla \vartheta(t+t_n)\|_{L^2(\varOmega)}^2+\frac1p\|\vartheta(t+t_n)\|_{L^p(\varOmega)}^p,\right]	
		\nonumber\\&=  \lim_{n\to\infty}\frac{1}{2}\|\nabla \vartheta(t+t_n)\|_{L^2(\varOmega)}^2+\frac{1}{p}\|\vartheta_\infty\|^p_{L^p(\varOmega)}. 
	\end{align}
Therefore, it is immediate that 
\begin{align}\label{nabla u asym}
				\lim_{n\to\infty}\|\nabla \vartheta(t+t_n)\|_{L^2(\varOmega)}^2=2\E-\frac{2}{p}\|\vartheta_\infty\|^p_{L^p(\varOmega)}.
				\end{align}
			Now taking $n\to\infty$ in \eqref{eqn-lambda-n} and using \eqref{nabla u asym}, we get
			\begin{equation*}
				\begin{aligned}
					\lim_{n\to\infty}\lambda(t+t_n)&=2\E-\frac{2}{p}\|\vartheta_\infty\|^p_{L^p(\varOmega)}+\|\vartheta_\infty\|^p_{L^p(\varOmega)},\\
					&=2\E +\frac{p-2}{p}\|\vartheta_\infty\|^p_{L^p(\varOmega)}\\
					&=\l,
				\end{aligned}
			\end{equation*}
		where we have used \eqref{l def} also, and the required convergence \eqref{eqn-conv-2} follows. 
		\end{proof}

	\subsection{Proof of Theorem \ref{thm-strong}}\label{sub-sec-asy}
	With the help of Lemmas \ref{sub seq conv soln}, \ref{lem:uniform_shift} and \ref{lamda t+tk conv},	let us now provide a proof of Theorem \ref{thm-strong}. 
		\begin{proof}[Proof of Theorem \ref{thm-strong}]  Under the assumptions of Theorem~\ref{thm-strong}, let
			$\{t_n\}_{n\in\mathbb{N}}$ be the sequence obtained in
			Lemma~\ref{sub seq conv soln} such that
			$t_n\to\infty$ as $n\to\infty$.
			Let $\varphi\in H^1_0(\varOmega)$ and $\tau>0$ be fixed. Clearly the function $\vartheta_n(\cdot)$ as defined in \eqref{uk(t) def} satisfies
			\begin{equation}\label{uk(t) asym eqn}
				(\vartheta_n)_{tt}+\g (\vartheta_n)_t+\mathcal{A} \vartheta_n+|\vartheta_n|^{p-2}\vartheta_n=\lambda(t+t_n)\vartheta_n \ \text{ in } \ H^{-1}(\varOmega). 
			\end{equation}
			Now testing \eqref{uk(t) asym eqn} with $\varphi$ and on integrating with respect to time over $[0,\tau]$, we get
			\begin{equation}\label{uk(t) asym eqn 2}
				\begin{aligned}
					&\underbrace{\int_{0}^{\tau}
						\left\langle (\vartheta_n)_{tt}(s), \varphi \right\rangle \, ds}_{\coloneqq I_1}
					+\underbrace{\g \int_{0}^{\tau}
						\bigl((\vartheta_n)_t(s), \varphi\bigr)\, ds}_{\coloneqq I_2} \\
					&\quad
					+\underbrace{\int_{0}^{\tau}\int_{\varOmega}
						\nabla \vartheta_n(s)\cdot \nabla\varphi \, dx\, ds}_{\coloneqq I_3}
					+\underbrace{\int_{0}^{\tau}\int_{\varOmega}
						|\vartheta_n(s)|^{p-2}\vartheta_n(s)\varphi \, dx\, ds}_{\coloneqq I_4} \\
					&=
					\underbrace{\int_{0}^{\tau}
						\lambda(t+t_n)
						\int_{\varOmega}
						\vartheta_n(s)\varphi \, dx\, ds}_{\coloneqq I_5} .
				\end{aligned}
			\end{equation}
			\textbf{$I_1$:} By \emph{Lions--Magenes} lemma and using \eqref{eqn-conv}, we get
			\begin{equation}
				\begin{aligned}
					\left |\int_{0}^{\tau}
					\left\langle (\vartheta_n)_{tt}(s), \varphi \right\rangle \, ds\right|&\le\big|\langle (\vartheta_n)_t(\tau),\varphi\rangle\big| + \big|\langle (\vartheta_n)_t(0),\varphi\rangle\big|\\
					&\le 2\sup_{t\in[0,\tau]}\norm{\vartheta_t(t+t_n)}_{H^{-1}(\varOmega)}\|\varphi\|_{H^1_0(\varOmega)}\to0 \ \text{ as }\ n\to\infty.
				\end{aligned}	
			\end{equation}
			\textbf{$I_2$:} Using the Cauchy-Schwatz inequality and equation \eqref{eqn-conv}, we find
			\begin{equation}
				\begin{aligned}
					\g\left| \int_{0}^{\tau}
					\bigl((\vartheta_n)_t(s), \varphi\bigr)\, ds\right|&\le\g \int_{0}^{\tau}
					\bigl|((\vartheta_n)_t(s), \varphi)\bigr|\, ds\\
					&\le\g\tau\sup_{t\in[0,\tau]}\|(\vartheta_t(t+t_n)\| _{L^2(\varOmega)}\|\varphi\| _{L^2(\varOmega)}\to 0 \ \text{ as }\ n\to\infty.
				\end{aligned}
			\end{equation}
			\textbf{$I_3,I_4$:} By the strong convergences given in Lemma \ref{lem:uniform_shift}, we infer 
			\begin{equation}
				\begin{aligned}
					\int_{0}^{\tau}\int_{\varOmega}
					\nabla \vartheta_n(s)\cdot \nabla\varphi \, dx\, ds& \to \int_{0}^{\tau}\int_{\varOmega}
					\nabla\vartheta_\infty(s)\cdot \nabla\varphi \, dx\, ds\ \text{ as }\ n\to\infty,\\
					\int_{0}^{\tau}\int_{\varOmega}
					|\vartheta_n(s)|^{p-2}\vartheta_n(s)\varphi \, dx\, ds &\to \int_{0}^{\tau}\int_{\varOmega}
					|\vartheta_\infty(s)|^{p-2}\vartheta_\infty(s)\varphi \, dx\, ds\ \text{ as }\ n\to\infty.
				\end{aligned}
			\end{equation}
			\textbf{$I_5$:} By using Lemma \ref {lamda t+tk conv} and with the help of \eqref{Norm =1}, we deduce 
			\begin{equation}		
				\begin{aligned}
				&	\left|\int_{0}^{\tau}\lambda(t+t_n)\int_{\varOmega}\vartheta_n(s)\varphi \, dx\, ds-\tau\l\int_{\varOmega}\vartheta_\infty(s)\varphi \, dx\right|
				\\&\le
					\left|\int_{0}^{\tau}(\lambda(t+t_n)-\l)\int_{\varOmega}\vartheta_n(s)\varphi \, dx\, ds\right|
					 +\tau \left|\l\int_{\varOmega}(\vartheta_n-\vartheta_\infty)\ dx\right|\\
					&\le \tau\sup_{t\in[0,\tau]}|\lambda(t+t_n)-\l|\|\varphi\|_{L^2(\varOmega)}
					 +\tau\bigl|\varOmega\bigr |^{\frac{1}{2}}|\l|\sup_{t\in[0,\tau]}
					\|\vartheta_n(t)-\vartheta_\infty\|_{L^2(\varOmega)}\\
					&	\to 0\ \text{ as }\ n\to\infty.
				\end{aligned}
			\end{equation}
			From all the above estimates for $I_1\ \text{to}\ I_5$, we get 
			\begin{equation}
				\mathcal{A}\vartheta_\infty+|\vartheta_\infty|^{p-2}\vartheta_\infty=\l\vartheta_\infty
				\quad\text{ in }\ H^{-1}(\varOmega).
			\end{equation}
			Therefore, from  \eqref{Lp asym conv} and \eqref{Norm =1}, we clearly have 
			$\|\vartheta_\infty\| _{L^2(\varOmega)}=1.$
			Testing by $\vartheta_\infty$ in equation \eqref{stationary pde} and using $\|\vartheta_\infty\| _{L^2(\varOmega)}=1$ , we get
			\begin{equation}
				\|\nabla\vartheta_\infty\|^2 _{L^2(\varOmega)}+\|\vartheta_\infty\|^p_{L^p(\varOmega)}=\l.
			\end{equation}			
			From \eqref{eqn-conv-4}, we infer   $	\l=\lim_{n\to\infty}\|\nabla \vartheta (t_n)\|^2 _{L^2(\varOmega)}+\|\vartheta_\infty\|^p_{L^p(\varOmega)}$. 
			On comparing above two results, we get
			\begin{equation*}
				\lim_{n\to\infty}\|\nabla \vartheta (t_n)\|^2 _{L^2(\varOmega)}=\|\vartheta_\infty\|^2 _{L^2(\varOmega)}.
			\end{equation*}					
			Using the above equality together with the weak convergence in
			$H_0^1(\varOmega)$ given by \eqref{weak asym conv}, and invoking the
			Radon--Riesz property \cite[Lemma 27.3]{JCR-20}, we obtain
			$$\|\vartheta (t_n)-\vartheta_\infty\|_{H^1_0(\varOmega)}\to0\ \text{ as }\ n\to\infty,$$ 
			which completes the proof. 
		\end{proof}
		\subsection{Long-time asymptotic behavior}\label{sub-long-time}
	Since every bounded trajectory admits a convergent subsequence as
	$t\to\infty$, the remaining challenge is to determine whether the entire
	trajectory converges, rather than merely a subsequence. In other words,
	we seek conditions under which the solution converges to a single limit
	as $t\to\infty$. In contrast to constrained parabolic equations with
	positive initial data (\cite{AB+ZB+MTM}), a maximum principle is not available in our
	setting. We therefore require a different mechanism to establish the
	convergence of the full trajectory.
		
		For the nonlinear damped wave equation, the damping mechanism induces a monotone decay of the total energy, suggesting that the long-time dynamics should be closely related to the set of equilibria. Consequently, our attention is naturally directed toward the associated stationary problem \eqref{stationary pde}. The stationary solutions correspond to critical points of the energy functional, and in particular, the asymptotic limit is expected to be characterized by a minimizer of the energy constrained to the underlying manifold. 
		
		\subsubsection{Ground state trajectories}
		The notion of ground states is first introduced, followed by a proof of the existence of a ground state solution to problem \eqref{asym1}. This is achieved through the coercivity and weak lower semicontinuity of the energy functional $E$  defined in \eqref{energy}.
		We now establish that any minimizer of the constrained variational problem
		\begin{equation}\label{Ground state cond}
			H^1_0(\varOmega)\times L^2(\varOmega)\ni (\vartheta,\varrho)\to E(\vartheta,\varrho):=	\frac12\|\nabla\vartheta\|_{L^2(\varOmega)}^2
			+
			\frac1p\|\vartheta\|_{L^p(\varOmega)}^p+ \frac12\|\varrho\|^2_{L^2(\varOmega)}\in[0,\infty),
		\end{equation}
		subject to the constraint that the solution belongs to the unit $L^2$-sphere, that is,
		\begin{equation}\label{Energy}
			\inf_{(\vartheta,\varrho)\in  H_0^1(\varOmega)\times L^{2}(\varOmega)}\Bigl\{E(\vartheta,\varrho): \|\vartheta\|_{L^2(\varOmega)}=1\Bigr\},
		\end{equation}
		gives rise to a weak solution of the stationary equation associated with \eqref{asym1}, namely
		\begin{equation}\label{Ground state pde}
			-\A\vartheta-|\vartheta|^{p-2}\vartheta
			+\bigl (\|\nabla\vartheta\|_{L^2(\varOmega)}^2u
			+\|\vartheta\|_{L^p(\varOmega)}^p\bigr )\vartheta=0
			\  \text{ in }\ H^{-1}(\varOmega).
		\end{equation}
		Such a solution will be referred to as a ground state.
		\begin{definition}
			We refer to a solution $\vartheta$ of \eqref{Ground state pde} as a \emph{ground state} whenever $\vartheta$ attains the minimum of the energy functional \eqref{Ground state cond}.
		\end{definition}

			\begin{definition}\label{definition of equlibria}
			An \emph{equilibrium} or a \emph{stationary solution} of \eqref{asym1} is a function $v\in H^1_0(\varOmega),$ satisfying $$\A v +|v|^{p-2}v = \left( \|\nabla v\|^2 _{L^2(\varOmega)}+\|v\|^p_{L^p(\varOmega)} \right)v,$$ where $\|v\|_{L^2(\varOmega)}=1.$ 
		\end{definition}
	The set of all equilibria is denoted $\mathcal{E}.$

		\begin{theorem}[Existence of a ground state]	\label{thm:ground-state}
			Let $\varOmega\subset\mathbb R^d$ be a bounded domain with smooth boundary  and let $p$ satisfy \eqref{p range}. 
			Then the constrained minimization problem
			\[
			m:=
			\inf\left\{
			E(\vartheta,\varrho):
			(\vartheta,\varrho)\in H_0^1(\varOmega)\times L^2(\varOmega),\ 
			\|\vartheta\|_{L^2(\varOmega)}=1
			\right\}
			\]
			admits a minimizer $(U,V)$. Moreover,	$V=0,$ 	and there exists $\lambda\in\mathbb R$ such that
			\begin{align*}
			\A U+|U|^{p-2}U=\lambda U
			\ \text{ in }\ H^{-1}(\varOmega). 
			\end{align*}
			In particular,
		$	\lambda	=	\|\nabla U\|_{L^2(\varOmega)}^2	+	\|U\|_{L^p(\varOmega)}^p,	$	and hence
		\eqref{Ground state pde} is satisfied. 
		\end{theorem}
		
		\begin{proof}
			Let us set
			\[
			\mathcal Q
			:=
			\left\{
			(\vartheta,\varrho)\in H_0^1(\varOmega)\times L^2(\varOmega):
			\|\vartheta\|_{L^2(\varOmega)}=1
			\right\}.
			\]
			Let $(\vartheta_n,\varrho_n)\subset\mathcal M$ be a minimizing sequence, so that
			\[
			E(\vartheta_n,\varrho_n)\to  m.
			\]
			Since $E(\vartheta_n,\varrho_n)$ is bounded and all three terms in the definition
			of $E$ are nonnegative, we obtain
			\begin{align*}
			\|\nabla \vartheta_n\|_{L^2(\varOmega)}\le C,
			\ 
			\|\vartheta_n\|_{L^p(\varOmega)}\le C,
			\ 
			\|\varrho_n\|_{L^2(\varOmega)}\le C,
			\end{align*}
			where $C$ is some positive constant. Hence, an application of the \emph{Banach-Alaoglu theorem}, after passing to a subsequence, yields 
			\begin{align*}
			\vartheta_n\xrightarrow{w} U
			\ \text{ weakly in }\ H_0^1(\varOmega),
		\ \text{ and }\ 
			\varrho_n\xrightarrow{w} V
			\ \text{ weakly in }\ L^2(\varOmega).
			\end{align*}
			Since $\varOmega$ is bounded, the Rellich--Kondrachov theorem yields the
			compact embedding	$
			H_0^1(\varOmega)\hookrightarrow L^2(\varOmega).
		$
			Thus, we  have 
			\[
			\vartheta_n\to  U
			\ \text{ strongly in }\ L^2(\varOmega).
			\]
			Since $\|\vartheta_n\|_{L^2(\varOmega)}=1$, it follows that
		$
			\|U\|_{L^2(\varOmega)}=1,
		$
		so that  $(\vartheta,\varrho)\in\mathcal Q$.
			
			The maps
		$
			\vartheta\mapsto\|\nabla\vartheta\|_{L^2(\varOmega)}^2,
			\ 
			\vartheta\mapsto\|\vartheta\|_{L^p(\varOmega)}^p,
			\ 
			\varrho\mapsto\|\varrho\|_{L^2(\varOmega)}^2
			$
			are weakly lower semicontinuous. Therefore, we have 
			\begin{align*}
				E(U,V)
				&\le
				\liminf_{n\to\infty}E(\vartheta_n,\varrho_n)
				=m.
			\end{align*}
			Since $(U,V)\in \mathcal Q$, the definition of $m$ gives
		$
			m\le E(U,V).
		$
			Hence
		$
			E(U,V)=m,
		$
			so that $(U,V)$ is a minimizer.
			
			We next show that necessarily $V=0$. Since $(U,0)\in\mathcal Q$, by the definition of $E(\cdot,\cdot)$, we infer 
			$$
			E(U,0)
			\le E(U,V).
		$$
			But, we know that 
			\[
			E(U,V)
			=
			E(U,0)+\frac12\|V\|_{L^2(\varOmega)}^2.
			\]
			Since $(U,V)$ is a minimizer, $E(U,V)\leq E(U,0)$ which implies $E(U,V)= E(U,0)$, and therefore
			\[
			\|V\|_{L^2(\varOmega)}=0\Rightarrow  V=0.
			\]

			It remains to derive the Euler--Lagrange equation for $\vartheta$. Let us define
			\[
			G(\vartheta):=\frac12\|\vartheta\|_{L^2(\varOmega)}^2.
			\]
			Then the  constraint is
			$
			G(U)=\frac12.
		$
			Moreover, for all $\phi \in L^2(\varOmega)$, we have 
			\[
			G'(U)[\varphi]
			=
			\int_\varOmega U \phi d x.
			\]
			Since $U\ne0$, we have $G'(U)\ne0$. Hence the Lagrange multiplier
			theorem yields $\lambda\in\mathbb R$ such that
			\[
			E'(U,0)[\varphi,\psi]
			=
			\lambda G'(U)[\varphi]
			\]
			for every
			$
			(\varphi,\psi)\in H_0^1(\varOmega)\times L^2(\varOmega).
		$
			Since
			\[
			E'(U,0)[\varphi,\psi]
			=
			\int_\varOmega\nabla U\cdot\nabla\varphi d x
			+
			\int_\varOmega |U|^{p-2}U \varphi d x,
			\]
			we obtain
			\begin{align}\label{eqn-weak}
			\int_\varOmega\nabla U\cdot\nabla\varphi d x
			+
			\int_\varOmega |U|^{p-2}U \varphi d x
			=
			\lambda\int_\varOmega U \varphi d x
			\end{align}
			for every $\varphi\in H_0^1(\varOmega)$. Hence
			\[
			\A U+|U|^{p-2}U=\lambda U
			\ \text{ in }\ H^{-1}(\varOmega).
			\]
			Finally, taking $\varphi=U$ in \eqref{eqn-weak} gives
			\begin{align*}
			\|\nabla U\|_{L^2(\varOmega)}^2
			+
			\|U\|_{L^p(\varOmega)}^p
			=
			\lambda\|U\|_{L^2(\varOmega)}^2
			=\lambda,
			\end{align*}
			since $U\in\mathcal{M}$. 
			Therefore
		$
			\lambda
			=
			\|\nabla U\|_{L^2(\varOmega)}^2
			+
			\|U\|_{L^p(\varOmega)}^p,
		$
		 and $U$ satisfies \eqref{Ground state pde}. 
		\end{proof}

		Let us set
	$
		w:=|U|.
	$
		Then, by the Sobolev chain rule, we have 
		\begin{align*}
		\|w\|_{L^2(\varOmega)}=\|U\|_{L^2(\varOmega)}=1,
		\ 
		\|w\|_{L^p(\varOmega)}=\|U\|_{L^p(\varOmega)}, \ \text{ and }\  \|\nabla w\|_{L^2(\varOmega)}
		\le
		\|\nabla U\|_{L^2(\varOmega)}.
		\end{align*}
		Consequently,	we get $	E(w,0)\le E(U,0).$
		Since $(U,0)$ is a minimizer of $E$ under the constraint
		$\|U\|_{L^2(\varOmega)}=1$, the reverse inequality follows from minimality:
		$
		E(U,0)\le E(w,0).
	$
		Thus
	$
		E(w,0)=E(u,0),
	$
		and hence $(w,0)$ is also a minimizer. 
		Since
	$
		E(|U|,0)=E(U,0),
	$
		we may assume, without loss of generality, that the first entry of the minimizer of the
		energy functional $E(\cdot,\cdot)$ is nonnegative. By the strong maximum principle,
		it is in fact positive in $\varOmega$. The uniqueness of the positive
		stationary solution is proved in
		\cite[Proposition 5.11]{AB+ZB+MTM}.

		\begin{lemma}\label{isolation of phi}
			Let $(\varphi_1,0)$ be a minimizer of
		\eqref{Energy},
			where $\varphi_1>0$ in $\varOmega$. Then, for any minimizer $(\vartheta,\varrho)$, we have 
			\[
			(\vartheta,\varrho)=(\varphi_1,0)
			\ \text{ or }\ 
			(\vartheta,\varrho)=(-\varphi_1,0).
			\]
		\end{lemma}
		
		\begin{proof}
			Let $(\vartheta,\varrho)$ be any minimizer. Since $(\vartheta,0)$ is admissible and
			$
			E(\vartheta,\varrho)
			=
			E(\vartheta,0)+\frac12\|\varrho\|_{L^2(\varOmega)}^2,
		$
			minimality implies
		$
			\varrho=0.
		$
			Thus $\vartheta$ is a minimizer of the $\vartheta$-part of the constrained energy.
			Since the function $w=|\vartheta|$ is  a minimizer, $w$ satisfies the
			Euler--Lagrange equation.  Hence $w\geq0$ satisfies
			\[
			\A w+w^{p-1}=\lambda w
			\ \text{ in }\ H^{-1}(\varOmega),
			\]
			where
		$
			\lambda
			=
			\|\nabla w\|_{L^2(\varOmega)}^2
			+
			\|w\|_{L^p(\varOmega)}^p.
		$
			Equivalently,
			\[
			\A w-\lambda w=-w^{p-1}\leq0.
			\]
			Thus $w$ is a nonnegative weak supersolution of
			$
			\A w-\lambda w=0.
		$
			Let
		$
			L\phi:=\A\phi-\lambda\phi.
		$
			Then
			\[
			L(-w)=w^{p-1}\geq0,
			\ 
			-w\leq0.
			\]
			Since the zeroth-order coefficient of $L$ is $-\lambda\leq0$,
			the strong maximum principle for weak solutions (\cite[Theorem 8.19]{DG+NST-01}) applies to $-w$.
			Because $w\not\equiv0$ (indeed $\|w\|_{L^2(\varOmega)}=1$), we obtain
			$
			-w<0\ \text{ in }\ \varOmega,
		$
		so that 
		$
			w>0\ \text{ in }\ \varOmega.
		$
			By the uniqueness of the positive normalized solution,
		$
			w=\varphi_1,
		$
		so that
		$
			\vartheta=\pm\varphi_1.
		$
		\end{proof}

		\begin{remark}\label{rem-isolated}
			From Lemma \ref{isolation of phi}, it follows that the positive ground state solution $\varphi_1$ is isolated. In particular, there exists a neighborhood of $\varphi_1$ in the underlying function space that contains no other solutions.
		\end{remark}

		Motivated by the isolation of the positive ground state solution $\varphi_1$, we next consider solutions corresponding to initial data $(\vartheta_0,\vartheta_1)$ in a neighborhood of $(\varphi_1,0)$. The following Lemma will play a crucial role in proving that such solutions remain in a suitable neighborhood of $\varphi_1,$ for all $t\geq 0$.

		\begin{lemma}\label{lem-energy}
			For a fixed $r_0>0$, there exists $\eta_0>0$ such that, for every
			$\vartheta\in\mathcal M$ and every $\varrho\in L^2(\varOmega)$ satisfying
			\[
			\|\vartheta-\varphi_1\|_{H_0^1(\varOmega)}=r_0,
			\]
			one has
			\begin{align}\label{eqn-upper}
			E(\vartheta,\varrho)\ge E(\varphi_1,0)+\eta_0.
			\end{align}
			Moreover, $r_0$ may be chosen such that
		$
			0<r_0<\frac12\|\varphi_1\|_{H_0^1(\varOmega)},
		$
			and consequently
			$-\varphi_{1}\notin\mathcal{B}_{r_0}(\varphi_{1}), $ where $\mathcal{B}_{r_0}(\varphi_1)
			:=
			\left\{
			\vartheta\in H_0^1(\varOmega):
			\|\vartheta-\varphi_1\|_{H_0^1(\varOmega)}<r_0
			\right\}.$
		\end{lemma}
		
		\begin{proof}
			Let us choose and fix
			$
			0<r_0<\frac12\|\varphi_1\|_{H_0^1(\varOmega)}.
		$
			Suppose, by contradiction, that no $\eta_0>0$ with the stated
			property exists. Then, for every $n\in\mathbb N$, there exist
			$\vartheta_n\in\mathcal M$ and $\varrho_n\in L^2(\varOmega)$ such that
			\begin{align*}
			\|\vartheta_n-\varphi_1\|_{H_0^1(\varOmega)}=r_0\ \text{ and }\  	E(\vartheta_n,\varrho_n)
			<
			E(\varphi_1,0)+\frac1n.
			\end{align*}
			Since
			\[
			E(\vartheta_n,\varrho_n)
			=
			E(\vartheta_n,0)+\frac12\|\varrho_n\|_{L^2(\varOmega)}^2
			\ge E(\vartheta_n,0),
			\]
			and $\varphi_1$ minimizes $E(\,\cdot\,,0)$ in $$\mathcal{R}:=\left\{\vartheta\in H_0^1(\varOmega):\|\vartheta\|_{L^2(\varOmega)}=1\right\},$$ we have
			\[
			E(\varphi_1,0)
			\le E(\vartheta_n,0)
			\le E(\vartheta_n,\varrho_n)
			< E(\varphi_1,0)+\frac1n.
			\]
			Therefore, we get 
			\[
			E(\vartheta_n,0)\to  E(\varphi_1,0)\ \text{ and }\ 
			\|\varrho_n\|_{L^2(\varOmega)}\to 0.
			\]
			Since $\{\vartheta_n\}$ is bounded in $H_0^1(\Omega)$, there exists $\vartheta\in H_0^1(\Omega)$ such that, up to a subsequence,
			
			\[
			\vartheta_n\xrightarrow{w} \vartheta
			\ \text{ in }\ H_0^1(\varOmega).
			\]
			Since $\varOmega$ is bounded,
			\[
			\vartheta_n\to \vartheta
			\ \text{ strongly in }\ L^2(\varOmega).
			\]
			Thus $\vartheta\in\mathcal R$, by weak lower semicontinuity, we have 
			\[
			E(\vartheta,0)
			\le \liminf_{n\to\infty}E(\vartheta_n,0)
			=E(\varphi_1,0).
			\]
			The minimality of $\varphi_1$ therefore yields
		$
			E(\vartheta,0)=E(\varphi_1,0).
			$
			By the uniqueness of the minimizer up to sign,
		$
			\vartheta=\pm\varphi_1
		$ (Lemma \ref{isolation of phi}). 
			The choice of $r_0$ excludes $\vartheta=-\varphi_1$, since
			\[
			\|-\varphi_1-\varphi_1\|_{H_0^1(\varOmega)}
			=
			2\|\varphi_1\|_{H_0^1(\varOmega)}
			>r_0,
			\]
		so that 
		$
			\vartheta=\varphi_1.
		$
				Moreover, the convergence of the energies together with weak
			convergence implies, by the Radon--Riesz property,
			\[
			\vartheta_n\to\varphi_1
			\ \text{ strongly in }\ H_0^1(\varOmega),
			\]
			which contradicts
		$
			\|\vartheta_n-\varphi_1\|_{H_0^1(\varOmega)}=r_0
		$
			for every $n$.	Therefore such an $\eta_0>0$ satisfying \eqref{eqn-upper} exists.
		\end{proof}

		\begin{lemma}[Solution trapped near $\varphi_1$]\label{trapped}
		Let $r>0$ be such that
		$
		r<\frac12\|\varphi_1\|_{H_0^1(\varOmega)}.
	$
		Suppose that the strong solution $\vartheta$ of \eqref{eq:mains}--\eqref{MID}, with
		initial data
	$	(\vartheta_0,\vartheta_1)\in D(\A)\times H_0^1(\varOmega)
		\subset H_0^1(\varOmega)\times L^2(\varOmega),	$
		satisfies
		\[
		\|\vartheta_0-\varphi_1\|_{H_0^1(\varOmega)}<r.
		\]
		If $\vartheta_0$ and $\vartheta_1$ are chosen in such a way that 
		\begin{align}\label{eqn-energy-d}
		E(\vartheta(0),\vartheta_t(0))
		=
		\frac12\|\vartheta_1\|_{L^2(\varOmega)}^2
		+\frac12\|\nabla \vartheta_0\|_{L^2(\varOmega)}^2
		+\frac1p\|\vartheta_0\|_{L^p(\varOmega)}^p
		<
		E(\varphi_1,0)+\eta,
		\end{align}
	for some $\eta>0$,	then
		\[
		\vartheta (t)\in\mathcal B_r(\varphi_1)
		\ \text{ in }\ H_0^1(\varOmega),
		\ \text{ for all }\ t\ge0.
		\]
		In particular,
		\[
		\|\vartheta (t)-\varphi_1\|_{H_0^1(\varOmega)}<r
		\ \text{ for all }\ t\ge0.
		\]
		\end{lemma}
		\begin{proof}
			Fix $r$ as above. By Lemma \ref{lem-energy}, there exists
			$\widetilde{\eta}=\widetilde{\eta}(r)>0$ such that
			$
			\|\widetilde{w}-\varphi_1\|_{H_0^1(\varOmega)}=r
		$
			implies
			\begin{align}\label{eqn-lower-1}
			E(\widetilde{w},v)\ge E(\varphi_1,0)+\widetilde\eta,
			\end{align}
			for every $\widetilde{w}\in\mathcal{R}$ and every $v\in L^2(\varOmega)$.
			Let us choose 
			\[
			0<\eta<\widetilde\eta.
			\]
			Since the energy is non-increasing, by assumption \eqref{eqn-energy-d}, we have
			\begin{align}\label{eqn-lower-2}
				E(\vartheta (t),\vartheta_t(t))\le E(\vartheta(0),\vartheta_t(0))<E(\varphi_1,0)+\eta,\ \text{ for all }\ t\ge0.
				\end{align}
			Suppose, by contradiction, that there exists \(t_0>0\), such that
			$$\|\vartheta(t_0)-\varphi_1\|_{H_0^1(\varOmega)}\ge r.$$
			Since \(\vartheta\in C([0,\infty);H_0^1(\varOmega))\) and $\|\vartheta_0-\varphi_1\|_{H_0^1(\varOmega)}<r,$
			there exists a first time \(t^*\in(0,t_0]\), such that $$\|\vartheta(t^*)-\varphi_1\|_{H_0^1(\varOmega)}=r.$$
			Applying  \eqref{eqn-lower-1} for  \(w=\vartheta(t^*)\) and $\varrho=0$ with $r_0=r$, we obtain   $$E(\vartheta(t^*),0)\ge E(\varphi_1,0)+\widetilde{\eta}>E(\varphi_1,0)+\eta.$$
			Since,
			$E(\vartheta(t^*),\vartheta_t(t^*))=\frac12\|\vartheta_t(t^*)\|_{L^2(\varOmega)}^2+ E(\vartheta(t^*),0),$
			it follows that $$E(\vartheta(t^*),\vartheta_t(t^*))\ge E(\vartheta(t^*),0)\ge E(\varphi_1,0)+\tilde{\eta}>E(\varphi_1,0)+\eta.$$
			This contradicts \eqref{eqn-lower-2}. 
			Therefore such a time \(t_0\) cannot exist and 
			$\vartheta (t)\in \mathcal{B}_r(\varphi_1),\ \text{ for all }\  t\ge0,$ which completes the proof. 	%In particular, $\vartheta (t)\in \mathcal{B}_r(\varphi_1), \text{ as } t\to\infty.$
		\end{proof}

		\subsubsection{Proof of Theorem \ref{GB}}
	We now turn to the main theorem of this section (Theorem \ref{GB}), which establishes the
	convergence of $\vartheta (t)$ as $t\to\infty$.
		
		\begin{proof}[Proof of Theorem \ref{GB}]
			We divide the proof into th following steps: 
			\vskip 0.1cm
			\noindent 
			\textbf{Step 1:} \emph{Characterisation of the $\omega$-limit set.}
			Let us define 
			\begin{equation*}
				\omega(\vartheta_0,\vartheta_1):=\left\{
				(\vartheta,\varrho)\in \mathcal{X}:
				\text{there exists } t_n\to\infty
			 	\text{ such that }
				\bigl(\vartheta (t_n),\vartheta_t(t_n)\bigr)
				\to (\vartheta,\varrho)
				\text{ in } \mathcal{X}
				\right\},
			\end{equation*}
			where $\mathcal{X} :=H_0^1(\varOmega)\times L^2(\varOmega)$. Let $(\widetilde{u},\widetilde{v})\in \omega(\vartheta_0,\vartheta_1)$. By the definition of the $\omega$-limit set, there exists a sequence
			$t_n\to\infty$ such that
			\begin{equation}\label{w limit 1}
				(\vartheta (t_n),\vartheta_t(t_n))
				\to
				(\widetilde{u},\widetilde{v})
			\ 	\text{ in }\ 
				H_0^1(\varOmega)\times L^2(\varOmega).
			\end{equation}
				Since \eqref{eqn-conv} holds, we infer that $\widetilde{v}=0$.
			Hence, every element of the $\omega$-limit set is of the form $(\widetilde{u},0)$.
		
			Let $\varphi\in H_0^1(\varOmega)$ be arbitrary. Integrating the weak
			formulation of \eqref{asym1} over the interval $(t_n,t_n+1)$ and using Lions--Magenes lemma, we obtain
			\begin{equation}\label{asym conv tn and tn+1}
				\begin{aligned}
					&\bigl( \vartheta_t(t_n+1)-\vartheta_t(t_n),\varphi\bigr)
					+\g\int_{t_n}^{t_n+1}
					\bigl( \vartheta_t(t),\varphi\bigr)  dt
					+\int_{t_n}^{t_n+1}
					\bigl(\nabla \vartheta (t),\nabla\varphi\bigr)  dt\\
					&\qquad+\int_{t_n}^{t_n+1}
					\bigl( |\vartheta (t)|^{p-2}\vartheta (t),\varphi\bigr)  dt
					=
					\int_{t_n}^{t_n+1}
					\lambda(t)\bigl( \vartheta (t),\varphi\bigr)  dt,
				\end{aligned}
			\end{equation}
			where, $\lambda(t)=\|\nabla \vartheta (t)\|_{L^2(\varOmega)}^2-\|\vartheta_t(t)\|_{L^2(\varOmega)}^2+\|\vartheta (t)\|_{L^p(\varOmega)}^p.$
			By Theorem \ref{apct}, the semiflow $\vartheta (t):=(\vartheta (t),\vartheta_t(t)),t\geq 0$ is asymptotically precompact. Combining this with arguments similar to Lemma  \ref{lem:uniform_shift} (see \eqref{uk limit asym H^1_0}) and Proposition \ref{convergence of ut to 0}, we can extract a sequence $t_n\to\infty$ such that the corresponding translated trajectories
			\begin{equation}\label{eqn-conv-5}
				(\vartheta(t_n+s),\vartheta_t(t_n+s)) \to (\widehat{u},0)\ \text{ in }\  H_0^1(\varOmega)\times L^2(\varOmega) \ \text{ for }\ s\in[0,1],
			\end{equation}
	for some $(\widehat{u},0)\in \omega(\vartheta_0,\vartheta_1)$. We now pass to the limit in each term of \eqref{asym conv tn and tn+1}. By the help of Lemma \ref{convergence of ut to 0}, we obtain
				\begin{equation*}
					|\big(\vartheta_t(t_n+1)-\vartheta_t(t_n),\varphi\big)|\le (\norm{\vartheta_t(t_n+1)} _{L^2(\varOmega)}+\norm{\vartheta_t(t_n)} _{L^2(\varOmega)})\norm{\phi} _{L^2(\varOmega)}\to 0\  \text{ as }\ n\to\infty.
				\end{equation*}	
				  Using the relation \eqref{velocity bound} and H\"older's inequality, we have
				\begin{equation*}
					\left|\int_{t_n}^{t_n+1}\big( \vartheta_t(t),\varphi \big)dt\right|\le\|\varphi\|_{L^2(\varOmega)}
					\left(\int_{t_n}^{t_n+1}\|\vartheta_t(t)\|_{L^2(\varOmega)}^2  dt \right)^{1/2}\to 0\ \text{ as }\ n\to\infty.
				\end{equation*}
				Using \eqref{eqn-conv-5} together with the Lebesgue Dominated Convergence Theorem, we conclude that
				\begin{equation*}
				 \int_{t_n}^{t_n+1}\big(\nabla \vartheta (t),\nabla\varphi\big)dt =
				 \int_0^1\big(\nabla \vartheta(t_n+s),\nabla\varphi\big)ds \to\big(\nabla \widehat{u},\nabla\varphi \big)
					\ \text{ as }\ n\to\infty.
				\end{equation*}
		For the values of $p$ given in  \eqref{p range}, the embedding \eqref{embedding} and the convergence \eqref{eqn-conv-5} imply 
				\begin{align*}
					\vartheta(t_n+s)&\to \widehat{u}\ \text{ in }\ L^p(\varOmega), \ \text{ for }\ s\in[0,1],\\
						\int_{t_n}^{t_n+1}\big(|\vartheta (t)|^{p-2}\vartheta (t),\varphi \big)dt&\to\big(|\widehat{u}|^{p-2}\widehat{u},\varphi \big) \ \text{ as }\ n\to\infty.
				\end{align*}
			Furthermore, by the help of \eqref{embedding}, \eqref{eqn-conv-5} and  Lemma \ref{convergence of ut to 0}, there exist $\mu_\infty\in\R$ such that 
				\begin{equation*}
					\lambda(t_n+s)
					\to
					\mu_\infty
				\end{equation*}
				for any $s\in[0,1]$. Therefore,\vspace{-0.2cm}
				\begin{equation*}
					\begin{aligned}
					& \left|	\int_{t_n}^{t_n+1}
					\lambda(t)\bigl( \vartheta (t),\varphi\bigr)  dt- \mu_\infty
					\bigl(\widehat{u},\varphi\bigr) \right|	\\&=
						\left|
						\int_0^1
						\lambda(t_n+s)\,
						\bigl(\vartheta(t_n+s),\varphi\bigr)ds
						-
						\mu_\infty
						\bigl(\widehat{u},\varphi\bigr)
						\right|
						\\
						&=
						\left|
						\int_0^1
						\bigl(\lambda(t_n+s)-\mu_\infty\bigr)
						\bigl(\vartheta(t_n+s),\varphi\bigr)ds
						+
						\mu_\infty
						\int_0^1
						\bigl(\vartheta(t_n+s)-\widehat{u},\varphi\bigr)ds
						\right|
						\\
						&\le
						\sup_{s\in[0,1]}
						\bigl|\lambda(t_n+s)-\mu_\infty\bigr|
						\int_0^1
						\bigl|\bigl(\vartheta(t_n+s),\varphi\bigr)\bigr|  ds
						+
						|\mu_\infty|
						\int_0^1
						\bigl|
						\bigl(\vartheta(t_n+s)-\widehat{u},\varphi\bigr )
						\bigr|  ds\\
						&\le	\sup_{s\in[0,1]}
						\bigl|\lambda(t_n+s)-\mu_\infty\bigr|\sup_{s\in[0,1]}\|\vartheta(t_n+s)\|_{L^2(\varOmega)}\|\phi\|_{L^2(\varOmega)}+|\mu_\infty| 	\sup_{s\in[0,1]}
						\bigl\|\vartheta(t_n+s)-\widehat{u}\bigr\|_{L^2(\varOmega)}\|\phi\|_{L^2(\varOmega)},\\
						&\to0\ \text{ as }\ n\to\infty.
					\end{aligned}
				\end{equation*}
			Passing to the limit in \eqref{asym conv tn and tn+1}, we obtain
			\begin{equation}\label{eqn-weak-form}
				\bigl(\nabla \widehat{u},\nabla\varphi\bigr)+\bigl(|\widehat{u}|^{p-2}\widehat{u},\varphi\bigr)=\mu_\infty
				\bigl(\widehat{u},\varphi\bigr), \ \text{ 	for all }\ \varphi\in H_0^1(\varOmega). 
			\end{equation}
		Therefore, we infer 
			\begin{equation*}
				\mathcal{A}\widehat{u}+|\widehat{u}|^{p-2}\widehat{u}=\mu_\infty \widehat{u},\ \text{ in }\ H^{-1}(\varOmega).
			\end{equation*}
			Taking $\varphi=\widehat{u}$ in  \eqref{eqn-weak-form} yields
			$
			\mu_\infty
			=
			\|\nabla \widehat{u}\|_{L^2(\varOmega)}^2
			+
			\|\widehat{u}\|_{L^p(\varOmega)}^p.
			$
			Therefore, by Definition \ref{definition of equlibria}, we get $\widehat{u}\in\mathcal{E}.$
			Here we see the key role played by Lemmas \ref{isolation of phi}
			and \ref{trapped}, which together guarantee that
		$	\vartheta (t)\in\mathcal{B}_r(\varphi_1),\  \text{for all }\ t\ge0.$
			%which is closed in
			%$H_0^1(\varOmega)$,
			This  implies 
		$$
				\widehat{u}\in\overline{\mathcal{B}_r(\varphi_1)}\  
	\text{	so that }\ 
			\widehat{u}\in\overline{\mathcal{B}_r(\varphi_1)}\cap\mathcal{E}.
		$$
			By Remark \ref{rem-isolated}, we infer 	$\overline{\mathcal{B}_r(\varphi_1)}\cap\mathcal{E}=\{\varphi_1\}.$
			Therefore, we have $$\widehat{u}=\varphi_1.$$
		Since $r<\frac{1}{2}\|\varphi_{1}\|_{H_0^1(\varOmega)},$  the $\omega$-limit set is singleton and 
			$$	\omega(\vartheta_0,\vartheta_1)=\{(\varphi_1,0)\}.$$
			
			\vskip 0.1cm
			\noindent 
			\textbf{Step 2:} \emph{Convergence of the full trajectory in
				$H_0^1(\varOmega)\times L^2(\varOmega)$.}
			Let us now show that the entire trajectory converges strongly to
			$(\varphi_1,0)$.
			Assume by contradiction that
			\begin{equation}
				(\vartheta (t),\vartheta_t(t))\not\to(\varphi_1,0)\ \text{ in }\ H_0^1(\varOmega)\times L^2(\varOmega).
			\end{equation}
			Then there exist $\varepsilon>0$ and a sequence $t_n\to\infty$, such that
			\begin{equation}\label{contradict asym}
				\|(\vartheta (t_n),\vartheta_t(t_n))-(\varphi_1,0)\|_{H_0^1(\varOmega)\times L^2(\varOmega)}\ge \varepsilon\
				\text{ as }\  n\to\infty.
			\end{equation}
			Since the orbit is asymptotically precompact in
			$H_0^1(\varOmega)\times L^2(\varOmega)$, the sequence $(\vartheta (t_n),\vartheta_t(t_n))$
			admits a subsequence, still denoted by $(\vartheta (t_n),\vartheta_t(t_n))$, and an element
			$(\widetilde{\vartheta},\widetilde{\varrho})\in H_0^1(\varOmega)\times L^2(\varOmega)$ such that
			\begin{equation*}
				(\vartheta (t_n),\vartheta_t(t_n))
				\to(\widetilde{\vartheta},\widetilde{\varrho})\ \text{  in}\ 	H_0^1(\varOmega)\times L^2(\varOmega).
			\end{equation*}	
		Since  $t_n\to\infty$, the limit point $(\widetilde{\vartheta},\widetilde{\varrho})$ belongs to the
			$\omega$-limit set $\omega(\vartheta_0,\vartheta_1)$, that is, $(\widetilde{\vartheta},\widetilde{\varrho})\in\omega(\vartheta_0,\vartheta_1).$
			Since $	\omega(\vartheta_0,\vartheta_1)=\{(\varphi_1,0)\},$
			it follows that$(\widetilde{\vartheta},\widetilde{\varrho})=	(\varphi_1,0).$	
			Therefore
			\begin{equation*}
				(\vartheta (t_n),\vartheta_t(t_n))\to(\varphi_1,0),\ \text{ strongly in }\	H_0^1(\varOmega)\times L^2(\varOmega),
			\end{equation*}
			which contradicts \eqref{contradict asym}.
			Hence, it is immediate that 
			\begin{equation}
				(\vartheta (t),\vartheta_t(t))\to(\varphi_1,0)\ \text{ in }\ H_0^1(\varOmega)\times L^2(\varOmega)\ \text{ as }\ t\to\infty.
			\end{equation}
			In particular,
			\begin{equation*}
				\vartheta (t)\to\varphi_1,
				\ \text{ strongly in }\ H_0^1(\varOmega),
			\end{equation*}
			which  completes the proof.
		\end{proof}
		
				%	\section{Appendix}
			\appendix
			\section{Number of solutions of the stationary problem}\label{appendix}
			\label{strongly damped wave}

				We briefly explain why the uniqueness of the positive normalized
			stationary solution does not imply uniqueness of the stationary
			problem. Recall that
			\[
			\mathcal M
			:=
			\left\{
			\vartheta\in L^2(\varOmega):
			\|\vartheta\|_{L^2(\varOmega)}=1
			\right\},
			\]
			and set
			\[
			M:=\mathcal M\cap H_0^1(\varOmega)
			=
			\left\{
			\vartheta\in H_0^1(\varOmega):
			\|\vartheta\|_{L^2(\varOmega)}=1
			\right\}.
			\]
			The constrained energy functional is
			\[
			J(\vartheta):=E(\vartheta,0)
			=
			\frac12\|\nabla\vartheta\|_{L^2(\varOmega)}^2
			+\frac1p\|\vartheta\|_{L^p(\varOmega)}^p,
			\  \vartheta\in M.
			\]
			Notice that $J$ is even,
			$
			J(-\vartheta)=J(\vartheta),
			$
			and $M$ is an infinite-dimensional symmetric constraint manifold.  Let us define
			\[
			G:H_0^1(\varOmega)\to\mathbb R,
			\ 
			G(\vartheta):=\frac12\left(\|\vartheta\|_{L^2(\varOmega)}^2-1\right).
			\]
			Then, it is immediate that 
			\[
			M=G^{-1}(\{0\})
			=
			\left\{
			\vartheta\in H_0^1(\varOmega):
			\|\vartheta\|_{L^2(\varOmega)}=1
			\right\}.
			\]
			For $\vartheta,\phi\in H_0^1(\varOmega)$, we find 
			\[
			G'(\vartheta)[\phi]
			=
			\int_\varOmega \vartheta\phi d x,
			\] 
	where $G'(\vartheta)$ denotes the Fr\'echet derivative of $G$ at $\vartheta$, which
	belongs to $(H_0^1(\varOmega))^{\prime}=H^{-1}(\varOmega)$. 	Moreover,
	$
		G'(\vartheta)[\vartheta]
		=
		\|\vartheta\|_{L^2(\varOmega)}^2
		=
		1,
	$
		so $G'(\vartheta)\neq0$ for all $\vartheta\in M$. Thus condition $(\textbf{M})$ in \cite[Page 104]{AA+AM-07} is satisfied.	Note that 
			\[
			\mathcal{T}_\vartheta M
			=
			\ker G'(\vartheta)
			=
			\left\{
			\phi\in H_0^1(\varOmega):
			\int_\varOmega \vartheta\phi d x=0
			\right\}.
			\]
			Identifying $G'(\vartheta)$ with $\vartheta$ through the natural embedding
			$L^2(\varOmega)\hookrightarrow H^{-1}(\varOmega)$, the normal space to $M$
			at $\vartheta$ is
			$
			\mathcal{N}_\vartheta M=\operatorname{span}\{G'(\vartheta)\}=\operatorname{span}\{\vartheta\}.
			$

			\begin{definition}[{\cite[Definition 7.7]{AA+AM-07}}]\label{lem-PS}
			A sequence $\vartheta_n\in M$ is called a \emph{Palais–Smale} (PS) sequence on $M$  if $J|_M$  is bounded and $\nabla_MJ(\vartheta_n)\to 0$.  We say that $J $ satisfies the (PS) condition on $M$, if every $PS$-sequence has a converging subsequence. 
			\end{definition}

			\begin{definition}
				A function $\vartheta\in M$ is called a \emph{critical point of $J$ on $M$}
				(or a \emph{constrained stationary point}) if
				\[
				J'(\vartheta)[\phi]=0
				\ 
				\text{ for every }\ \phi\in T_\vartheta M. 
				\]
				\end{definition}
			
			\begin{lemma}\label{lem-ps}
				$J|_M$ satisfies the  (PS) condition. 
			\end{lemma}
			\begin{proof}
		 Let $\{\vartheta_n\}_{n=1}^{\infty}\subset M$ be a PS sequence. Therefore, we have 
			\begin{align}
				\{J(\vartheta_n)\}_{n=1}^{\infty}\ \text{ is bounded and }
				\ 
				\nabla_MJ(\vartheta_n)\to 0.
			\end{align}
			Since
			$
			\frac12\|\nabla \vartheta_n\|_{L^2(\varOmega)}^2
			\leq J(\vartheta_n),
			$
			the sequence $\{\vartheta_n\}$ is bounded in $H_0^1(\varOmega)$. Hence, after
			passing to a subsequence, there exists $\vartheta\in H_0^1(\varOmega)$ such that
			\[
			\vartheta_n\xrightarrow{w} \vartheta
			\ \text{ weakly in }\ H_0^1(\varOmega).
			\]
			Moreover,   by the compact Sobolev embeddings
			$
			H_0^1(\varOmega)\hookrightarrow L^p(\varOmega),
			$
			valid for the range of $p$ specified in \eqref{p range}, we have
			\[
			\vartheta_n\to  \vartheta
			\ \text{ strongly in }
			L^p(\varOmega). 
			\]
			Since $\vartheta_n\in M$, we have
			$
			\|\vartheta_n\|_{L^2(\varOmega)}=1,
			$
			so that 
			\[
			\|\vartheta\|_{L^2(\varOmega)}
			=
			\lim_{n\to\infty}\|\vartheta_n\|_{L^2(\varOmega)}
			=1,
			\]
			and hence $\vartheta\in M$.

			Since $\nabla_MJ(\vartheta_n)\to0$, the derivative $J'(\vartheta_n)$ tends to zero
			when restricted to the tangent space $\mathcal{T}_{\vartheta_n}M$. Since 	$\mathcal{N}_\vartheta M=\operatorname{span}\{\vartheta\}$, the approximate Lagrange
			multiplier characterization yields $\lambda_n\in\mathbb R$ such that
				\begin{equation}\label{eqn-conv-10}
			J'(\vartheta_n)-\lambda_n\vartheta_n\to0
			\ \text{ in }\ H^{-1}(\varOmega).
			\end{equation}
			Here $\vartheta_n$ is identified with an element of $H^{-1}(\varOmega)$ through
			the duality pairing
			$
			\langle \vartheta_n,\phi\rangle
			=
			\int_\varOmega \vartheta_n\phi d x.
			$
			Since $\{\vartheta_n\}$ is bounded in $H_0^1(\varOmega)$, we may test \eqref{eqn-conv-10}
			with $\vartheta_n$ to obtain
			\begin{equation*}
			\left\langle J'(\vartheta_n),\vartheta_n\right\rangle
			-
			\lambda_n\|\vartheta_n\|_{L^2(\varOmega)}^2
			=o(1).
			\end{equation*}
			Using $\|\vartheta_n\|_{L^2(\varOmega)}=1$ and
			$
			\left\langle J'(\vartheta_n),\vartheta_n\right\rangle
			=
			\|\nabla \vartheta_n\|_{L^2(\varOmega)}^2
			+
			\|\vartheta_n\|_{L^p(\varOmega)}^p,
			$
			we obtain
			\begin{align}\label{eqn-conv-11}
				\lambda_n
				=
				\|\nabla \vartheta_n\|_{L^2(\varOmega)}^2
				+
				\|\vartheta_n\|_{L^p(\varOmega)}^p
				+o(1).
			\end{align}
			Since $\{J(\vartheta_n)\}$ is bounded, the two terms on the right-hand side
			of \eqref{eqn-conv-11} are bounded. Hence $\{\lambda_n\}$ is bounded. Passing to
			a further subsequence, we may assume that
			\[
			\lambda_n\to \lambda
			\ \text{ for some }\ \lambda\in\mathbb R.
			\]
			We now test \eqref{eqn-conv-10} with $\vartheta_n-\vartheta$. Since
			$\vartheta_n-\vartheta$ is bounded in $H_0^1(\varOmega)$, we obtain
			\begin{align}\label{eqn-conv-12}
				0
				=&
				\lim_{n\to\infty}
				\left\langle
				J'(\vartheta_n)-\lambda_n\vartheta_n,\vartheta_n-\vartheta
				\right\rangle
				\nonumber	\\
				=&
				\lim_{n\to\infty}
				\Bigg[
				\int_\varOmega
				\nabla \vartheta_n\cdot\nabla(\vartheta_n-\vartheta) d x
				+
				\int_\varOmega
				|\vartheta_n|^{p-2}\vartheta_n(\vartheta_n-\vartheta) d x
				-
				\lambda_n\int_\varOmega
				\vartheta_n(\vartheta_n-\vartheta) d x
				\Bigg].
			\end{align}
			We claim that the last two terms in \eqref{eqn-conv-12} converge to zero. Indeed,
			since
			$
			\vartheta_n\to  \vartheta
			\ \text{ strongly in }\ L^2(\varOmega)
			$
			and $\{\lambda_n\}$ is bounded,
			\begin{align}\label{eqn-conv-13}
				\left|
				\lambda_n\int_\varOmega \vartheta_n(\vartheta_n-\vartheta) d x
				\right|
				\leq
				|\lambda_n|
				\|\vartheta_n\|_{L^2(\varOmega)}
				\|\vartheta_n-\vartheta\|_{L^2(\varOmega)}
				\to 0.
			\end{align}
			Furthermore, since
			$
			\vartheta_n\to  \vartheta
			\ \text{ strongly in }\ L^p(\varOmega),
			$
			the continuity of the map
			$
			z\mapsto |z|^{p-2}z
			$
			from $L^p(\varOmega)$ into $L^{p'}(\varOmega)$, where
			$p'=p/(p-1)$, gives
			\[
			|\vartheta_n|^{p-2}\vartheta_n
			\to 
			|\vartheta|^{p-2}\vartheta
			\ \text{ strongly in }\ L^{p'}(\varOmega).
			\]
			Consequently,
			\begin{align}\label{eqn-conv-14}
				\left|
				\int_\varOmega |\vartheta_n|^{p-2}\vartheta_n(\vartheta_n-\vartheta) d x
				\right|
				\leq
				\left\|
				|\vartheta_n|^{p-2}\vartheta_n
				\right\|_{L^{p'}(\varOmega)}
				\|\vartheta_n-\vartheta\|_{L^p(\varOmega)}
				\to 0.
			\end{align}
			It follows from \eqref{eqn-conv-13}--\eqref{eqn-conv-14} that
			\begin{align}\label{eqn-conv-15}
				\int_\varOmega
				\nabla \vartheta_n\cdot\nabla(\vartheta_n-\vartheta) d x
				\to 0.
			\end{align}
			Therefore,
			$
			\|\nabla \vartheta_n\|_{L^2(\varOmega)}^2
			-
			\int_\varOmega\nabla \vartheta_n\cdot\nabla \vartheta d x
			\to 0.
			$
			Since
			$
			\vartheta_n\xrightarrow{w} \vartheta
			\ \text{ weakly in }\ H_0^1(\varOmega),
			$
			we have
			$
			\int_\varOmega\nabla \vartheta_n\cdot\nabla \vartheta d x
			\to 
			\|\nabla\vartheta\|_{L^2(\varOmega)}^2.
			$
			Therefore, we have 
			$
			\|\nabla \vartheta_n\|_{L^2(\varOmega)}^2
			\to 
			\|\nabla\vartheta\|_{L^2(\varOmega)}^2.
			$
			Together with
			$
			\vartheta_n\xrightarrow{w} \vartheta
			\ \text{ weakly in }\ H_0^1(\varOmega),
			$
			the Radon--Riesz property of the Hilbert space $H_0^1(\varOmega)$ yields
			\[
			\vartheta_n\to  \vartheta
			\ \text{ strongly in }\ H_0^1(\varOmega).
			\]
			Thus every PS sequence for $J|_M$ admits a strongly convergent
			subsequence in $H_0^1(\varOmega).$ Hence $J|_M$ satisfies the
			(PS) condition.
			\end{proof}

			\begin{proposition}\label{PS}
				The stationary problem \eqref{Ground state pde} admits a countable
				family of distinct normalized solutions
				$\{\vartheta_k\}_{k\in\mathbb N}\subset H_0^1(\varOmega)$.
			\end{proposition}

			\begin{proof}
			Under the (PS) condition for $J|_M$ (Lemma \ref{lem-PS}), the
			Lusternik--Schnirelmann theory (\cite[Theorems 9.10 and 10.9]{AA+AM-07}) based on the Krasnoselski genus (\cite[Definition 10.1]{AA+AM-07}) can be
			applied. For a compact symmetric set $A\subset M$, let $\gamma(A)$
			denote its Krasnoselski genus, and define
			\[
			\Sigma_k
			:=
			\left\{
			A\subset M:
			A\ \text{is compact and symmetric, and }
			\gamma(A)\ge k
			\right\}.
			\]
			For every $k\in\mathbb N$, the class $\Sigma_k$ is nonempty. Indeed,
			if $X_k\subset H_0^1(\varOmega)$ is a $k$-dimensional linear subspace, then
			\[
			A_k:=X_k\cap M
			=
			\left\{
			\vartheta\in X_k:\|\vartheta\|_{L^2(\varOmega)}=1
			\right\}
			\]
			is compact and symmetric and is homeomorphic to the sphere
			$\mathbb{S}^{k-1}$. Hence $\gamma(A_k)=k$ (\cite[Corollary 10.6]{AA+AM-07}).
			The corresponding Lusternik--Schnirelmann minimax levels are (\cite[Page 149]{AA+AM-07})
			\[
			c_k
			:=
			\inf_{A\in\Sigma_k}\sup_{\vartheta\in A}J(\vartheta),
			\  k\in\mathbb N.
			\]
			They are finite, since
			\[
			c_k
			=
			\inf_{A\in\Sigma_k}\sup_{\vartheta\in A}J(\vartheta)
			\leq
			\sup_{\vartheta\in A_k}J(\vartheta).
			\]
			By the Lusternik--Schnirelmann critical point theorem
			\cite[Theorems 9.10 and 10.9]{AA+AM-07}, if $J|_M$
			satisfies the (PS) condition at every finite level, then each
			$c_k$ is a critical value of $J|_M$. Moreover, since $M$ has infinite
			Krasnoselski genus (\cite[Corollary 10.6]{AA+AM-07}), the theorem yields infinitely many critical
			points of $J|_M$.
			
			Thus, for every $k\in\mathbb N$, there exists
			$\vartheta_k\in M$ and $\lambda_k\in\mathbb R$ such that
			\[
			J'(\vartheta_k)=\lambda_k \vartheta_k
			\ \text{ in }\ H^{-1}(\varOmega).
			\]
			Equivalently, the following system is satisfied:
			\begin{equation}
				\left\{
				\begin{aligned}
					\mathcal{A} \vartheta_k+|\vartheta_k|^{p-2}\vartheta_k
					&=
					\lambda_k \vartheta_k
					\ \text{ in }\ H^{-1}(\varOmega),\\
					\|\vartheta_k\|_{L^2(\varOmega)}&=1.
				\end{aligned}
				\right.
			\end{equation}
			Testing the equation with $u_k$ gives
			\[
			\lambda_k
			=
			\|\nabla \vartheta_k\|_{L^2(\varOmega)}^2
			+
			\|\vartheta_k\|_{L^p(\varOmega)}^p.
			\]
			Consequently, each critical point satisfies the normalized stationary
			problem
			\begin{align*}
			\mathcal{A} \vartheta_k+|\vartheta_k|^{p-2}\vartheta_k
			=
			\left(
			\|\nabla \vartheta_k\|_{L^2(\varOmega)}^2
			+
			\|\vartheta_k\|_{L^p(\varOmega)}^p
			\right)\vartheta_k \ \text{ in }\ H^{-1}(\varOmega). 
			\end{align*}
			which completes the proof. 
			\end{proof}
			
			\begin{remark}
				1. The above minimax construction yields a countable family of distinct
				normalized stationary solutions
			$
				\{\vartheta_k\}_{k\in\mathbb N}\subset H_0^1(\varOmega),
			$
				and hence the stationary problem \eqref{Ground state pde} possesses at
				least countably infinitely many solutions. We emphasize that this
				statement does not assert that the set of all stationary solutions is
				countable; there may exist additional solutions beyond the family
				obtained through the Lusternik--Schnirelmann construction.
				
		2. 	The first minimax level is the ground-state level,
			$c_1=J(\varphi_1),$
			and the corresponding critical points are $\pm\varphi_1$. Since the
			uniqueness of the positive normalized stationary solution $\varphi_1$
			has already been established, the existence of further stationary
			solutions is not excluded. In particular, any normalized stationary
			solution different from $\pm\varphi_1$ cannot be nonnegative. Indeed,
			if $\vartheta$ is a nonnegative normalized stationary solution, then, by the
			strong maximum principle, $\vartheta>0$ in $\varOmega$. The established uniqueness
			of the positive normalized stationary solution then implies
			$\vartheta=\varphi_1.$
			Consequently, every normalized stationary solution other than
			$\pm\varphi_1$ must change sign.
		\end{remark}

		\medskip\noindent
		\textbf{Acknowledgments:} 
		The authors gratefully acknowledge Mr. Ashish Bawalia for his valuable
		discussions and insightful suggestions.		H. Tiwari wants to thank University Grants Commission (UGC),	Govt. of India for financial assistance.  Support for M. T. Mohan's research received from the National Board of Higher Mathematics (NBHM), Department of Atomic Energy, Government of India (Project No. 02011/13/2025/NBHM(R.P)/R\&D II/1137).

				\medskip\noindent	\textbf{Declarations:} 
				
				\noindent 	\textbf{Ethical Approval:}   Not applicable 
				
				\noindent  \textbf{Competing interests: } The authors declare no competing interests. 
				
				\noindent  \textbf{Conflict of interest: }On behalf of all authors, the corresponding author states that there is no conflict of interest.
				
				\noindent 	\textbf{Authors' contributions:} All authors have contributed equally. 
				
				\noindent 	\textbf{Availability of data and materials:} Not applicable. 
				
				\bibliographystyle{plain}
				\bibliography{Harsh}
				
				\end{document}